\documentclass[11pt]{article}
\usepackage{kbordermatrix}

\usepackage{caption}
\usepackage{amsfonts,amsthm}
\usepackage{amsmath,amssymb,enumerate,color}
\usepackage{algorithm,algorithmic}
\usepackage{epsfig,graphicx,psfrag,mathrsfs,subfigure}
\usepackage{eucal}
\usepackage{yhmath}

\usepackage{cleveref}

\RequirePackage{multirow}

\usepackage{geometry}
\usepackage{lipsum}

\def\bc{\pmb{c}}

\def\be{\pmb{e}}
\def\Bf{\pmb{f}}

\def\bu{\pmb{u}}
\def\bv{\pmb{v}}

\def\bx{\pmb{x}}
\def\by{\pmb{y}}
\def\bz{\pmb{z}}

\def\hm{\hphantom{-}}

\DeclareMathOperator{\arccosh}{arccosh}
\DeclareMathOperator{\diag}{diag}
\DeclareMathOperator{\rank}{rank}

\def\wtd{\widetilde}
\def\what{\widehat}
\def\hm{\hphantom{-}}

\def\bbC{\mathbb{C}}

\def\bbI{\mathbb{I}}
\def\bbJ{\mathbb{J}}

\def\bbR{\mathbb{R}}

\def\sss{\scriptscriptstyle}

\DeclareMathOperator{\F}{F}
\DeclareMathOperator{\HH}{H}
\DeclareMathOperator{\T}{T}

\newtheorem{theorem}{Theorem}[section]
\newtheorem{lemma}{Lemma}[section]
\newtheorem{corollary}{Corollary}[section]

\theoremstyle{definition}

\newtheorem{remark}{Remark}[section]
\newtheorem{example}{Example}[section]

\numberwithin{algorithm}{section}
\numberwithin{equation}{section}
\numberwithin{figure}{section}
\numberwithin{table}{section}

\graphicspath{{figures/}{figure/}{pictures/}%
	{picture/}{pic/}{pics/}{image/}{images/}}

\title{A Minimal Perturbation Approach For The Rectangular Multiparameter Eigenvalue Problem}

\author{Shanheng Han\thanks{%
   School of Mathematical Sciences, Soochow University, Suzhou 215006, Jiangsu, China.
   Email: {\tt 3468805603@qq.com}.}
\and Lei-Hong Zhang\thanks{%
   Corresponding author.
   School of Mathematical Sciences, Soochow University, Suzhou 215006, Jiangsu, China.
   Supported in part by the National Natural Science Foundation of China (NSFC-12471356, NSFC-12371380),
   Jiangsu Shuangchuang Project (JSSCTD202209),  Academic Degree and Postgraduate Education Reform Project of Jiangsu Province,
   and China Association of Higher Education under grant 23SX0403.
   Email: {\tt longzlh@suda.edu.cn}.}
\and Ren-Cang Li\thanks{%
   Department of Mathematics, University of Texas at Arlington,             Arlington, 76019-0408, TX, USA.
   Supported in part by NSF DMS-2407692.
   Email:  {\tt rcli@uta.edu}.}
}

 \date{\today}

\begin{document}

\maketitle

\begin{abstract}
The rectangular multiparameter eigenvalue problem (RMEP)
involves rectangular coefficient matrices (usually with more rows than columns) and may potentially have no solution in its original form.
A minimal perturbation framework is proposed to defines approximate solutions. Computationally,
two particular scenarios are considered: computing one approximate eigen-tuple or a complete set of approximate eigen-tuples.
For computing one approximate eigen-tuple, an alternating iterative scheme with proven convergence is devised, while for a complete set of approximate eigen-tuples, 
the framework leads to a standard MEP (RMEP with square coefficient matrices) for numerical solutions.
The proposed approach is validated on RMEPs from discretizing the multiparameter Sturm-Liouville  equation and 
the Helmholtz equations  by the least-squares spectral method.
\end{abstract}

%\centerline{\em Draft: not for wide distribution}

\medskip
{\bf Keywords.} Rectangular multiparameter eigenvalue problems; RMEP; rectangular generalized eigenvalue problems;
RGEP; multiparameter Sturm-Liouville equation
\medskip

{\bf AMS subject classifications.} 65F15, 65F20, 15A18, 47A75, 65N35

\clearpage
\tableofcontents

\clearpage
%%%%%%%%%%%%%%%%%%%%%%%%%%%%%%%%
\section{Introduction}\label{sec:introduction}
Given matrices $A_i, B_{i1}, B_{i2}, \ldots, B_{ik} \in \bbC^{m_i \times n_i}$ for $1\le i\le k$, the general  {\em rectangular multiparameter eigenvalue problem\/} (RMEP)
takes the form
\begin{equation}\label{eq:RMEP}
\left\lbrace \begin{aligned}
		&A_1 \bx_1 = \lambda_1 B_{11} \bx_1 + \lambda_2 B_{12} \bx_1 + \cdots + \lambda_k B_{1k} \bx_1, \\
		&\quad\quad\quad\vdots\quad\\
		&A_k \bx_k = \lambda_1 B_{k1} \bx_k + \lambda_2 B_{k2} \bx_k + \cdots + \lambda_k B_{kk} \bx_k,
	\end{aligned}\right.
\end{equation}
where $\lambda_i\in\bbC$ and $0\ne \bx_i\in \bbC^{n_i}$ for $1\le i\le k$ are unknown scalars and
vectors to be determined. When \eqref{eq:RMEP} holds, $(\lambda_1,\ldots,\lambda_k)$ is termed an {\em eigenvalue tuple\/}
and $(\bx_1,\ldots,\bx_k)$ the corresponding {\em eigenvector tuple}.
In its generality, $m_i$ may or may not be equal to $n_i$ for every $i$ and, for our interest in this paper, $m_i\ge n_i$.
In the event when all $m_i=n_i$, we will refer \eqref{eq:RMEP} as
the {\em multiparameter eigenvalue problem\/}, or MEP for short.
Frequently, we may use the term ``MEP \eqref{eq:RMEP}'' with an implicit understanding that it is about RMEP~\eqref{eq:RMEP} with all $m_i=n_i$.
MEP \eqref{eq:RMEP} can be transformed into $k$ related GEPs \cite{atki:1968,atki:1972,plgh:2015}  (see Appendix \ref{sec:MEPvGEP}),
and consequently, existing theory and numerical methods on GEP can be tapped.

When $k=1$, RMEP~\eqref{eq:RMEP} is referred as the {\em rectangular generalized eigenvalue problem\/} (RGEP).
If also $m_1=n_1$,
it is the usual {\em generalized eigenvalue problem\/} (GEP)
$A\bx=\lambda B\bx$ where $A,\,B\in\bbC^{n\times n}$.
% is a special of case of \eqref{eq:RMEP}:
%$k=1$ and $m_1=n_1$, for which
There is a wealth of material in the literature on the subject, including
the Kronecker canonical form for RGEP \cite{gant:1959I,gant:1959II},
theory and numerical methods for GEP
\cite{bddrv:2000,bhli:1996,demm:1997,govl:2013,li:1994a,li:2003,li:2014HLA,stsu:1990} (and references therein).
We will provide a brief review on the state-of-the-art developments for RGEP in section \ref{sec:RGEP:MPF}
to motivate what we will do to RMEP \eqref{eq:RMEP} in this paper.

In passing, we point out that there is another type of multiparameter eigenvalue problem that was also named RMEP in
\cite{hokp:2024}:
% while the notion of RMEP has appeared in  \cite{hokp:2024}, the setting there is different from \eqref{eq:RMEP}. Specifically, in \cite{hokp:2024}
%there is only one defining equation
\begin{equation}\label{eq:RMEP:hokp2024}
A \bx= \lambda_1 B_1 \bx + \lambda_2 B_2 \bx  + \cdots + \lambda_k B_k \bx,
\end{equation}
where $A,\,B_i \in \bbC^{(n+k-1)\times n}$ for $1\le i\le k$. It has two major differences from RMEP~\eqref{eq:RMEP}
of interest in this paper: 1) there is only one defining equation in \eqref{eq:RMEP:hokp2024},
and 2) the number of rows in the matrices $A,\, B_i$ is precisely $n+k-1$ rather than arbitrary.
Applications  of this type of  RMEP \eqref{eq:RMEP:hokp2024}
include the optimal least squares autoregressive moving average model as well as the optimal least squares realization of autonomous linear time-invariant dynamical system (see  \cite{demo:2020,blum:1978,blum:1978b,hokp:2024,homp:2012,khaz:1998,shsh:2009,vemo:2022,vemo:2019,vemo:2023}
for more detail).

The focus of this paper is to solve RMEP \eqref{eq:RMEP}. In general, there is a possibility that RMEP \eqref{eq:RMEP}
may not have any eigenvalue tuple $(\lambda_1,\ldots,\lambda_k)$ along with
the corresponding eigenvector tuple $(\bx_1,\ldots,\bx_k)$. Hence its solution existence must be reinterpreted differently
from the usual sense of satisfying the equations in \eqref{eq:RMEP} exactly.

\subsection{Sources of RMEP}
Sources for RMEP \eqref{eq:RMEP} include
the multiparameter ODE eigenvalue problem \cite{dris:2010,hant:2022,hana:2022} discretized by
the least-squares spectral method with the help of
\texttt{Chebfun}\footnote{\texttt{Chebfun} is available at {\tt http://www.chebfun.org}}
\cite{drht:2014}, the Helmholtz equation \cite{amls:2014,eina:2022,ghhp:2012,volk:1988}
solved by the method of separation of variables, and a number of other types of differential equations
that can also lead to RMEP \eqref{eq:RMEP} (see, e.g., \cite{fohm:1972,ghhp:2012,jaho:2009,plgh:2015,slee:2008,volk:1988}).
Later in subsection~\ref{ssec:mSLeq}, we will provide detailed account on
the multiparameter ODE eigenvalue problem and the Helmholtz equation.

\subsection{Contributions}
%In this paper we will extend the minimal perturbation approach to treat RMEP \eqref{eq:RMEP}. While our presentation primarily focuses on the two-parameter case (i.e., $k = 2$ in \eqref{eq:RMEP}), the technique is readily extendable
%to any $k > 2$ as we will comment at appropriate places.
Our main contributions are summarized as follows:
\begin {enumerate}
\item We establish a minimal perturbation framework for determining $\ell$ approximate eigen-tuples for  RMEP \eqref{eq:RMEP},
    where $1\le\ell\le N:=n_1\cdots n_k$, and investigate theoretical properties as well as design efficient numerical methods
      for $\ell=1$ and $N=n_1\cdots n_k$, extending the works in \cite{boeg:2005,itmu:2016}.
     The case $\ell=1$  is for finding one approximate eigen-tuple while
    the case $\ell=N$ seeks a complete set of approximate eigen-tuples.

\item Our treatment for $\ell=1$, even in the case of RGEP (i.e., $k=1$), significantly simplifies
      the development in \cite{boeg:2005} and our results read much more elegant, too.
      Also our treatment requires essentially no additional care when it comes to the case $k=1$ (i.e., for RGEP)
      or general $k>1$.

\item We have to modify the minimal perturbation framework, initially as a straightforward extension to what's in
      \cite{boeg:2005} for RGEP, for $\ell=N$ so that we can lay the foundation
      to design numerical method based on truncated SVDs.

\item We show how to efficiently solve the multiparameter Sturm-Liouville equation and the Helmholtz equation in two steps:
      first discretize them by the least-squares spectral method with the help of \texttt{Chebfun} and then solve the resulting RMEP by our proposed algorithm
      for a complete set of approximate eigen-tuples.
\end {enumerate}
Numerical results are also reported to demonstrate the effectiveness of the approach.

\subsection{Organization and notation}
%We organize the paper as follows.
In \cref{sec:RGEP:MPF}, we review the state-of-the-art minimal perturbation framework for RGEP.
In \cref{sec:RMEP:MPF}, we present our minimal perturbation framework for RMEP, leaving the detailed
investigations to \cref{sec:RMEP:ell=1} for one approximate eigen-tuple
and \cref{sec:RMEP:ell=N} for a complete set of $N=n_1\cdots n_k$ approximate
eigen-tuples.
In \cref{sec:egs}, we present numerical experiments to evaluate our algorithms for RMEP, first on random problems and
  on RMEPs arising from discretizing
the multiparameter Sturm-Liouville equation and the Helmholtz equation by the least-squares spectral method
with the help of \texttt{Chebfun}. Numerical results demonstrate the efficiency of the new methods.
Conclusions are drawn in  \cref{sec:conclusion}.
In Appendix~\ref{sec:MEPvGEP}, we explain the transformation of MEP into GEPs
for numerical solutions as it is helpful to the understanding of \cref{sec:RMEP:ell=N,sec:egs}.

Throughout this paper,
 ${\mathbb C}^{m\times n}$ (${\mathbb R}^{m\times n}$) is the set
of all $m\times n$ complex (real) matrices.
$I_n\equiv [\be_1,\be_2,\dots,\be_n]\in\bbR^{n\times n}$ is the identity matrix.  Vectors are denoted by the bold lower case letters while capital letters are for matrices.
For a matrix $A\in \bbC^{m\times n}$,  $\|A\|_2=\sigma_1(A)$ and $\|A\|_{\F}$ are the  2-norm and Frobenius norm of a matrix $A$, respectively, where $\sigma_j(A)$ denotes the $j$th largest singular value of $A$;  $\rank(A)$ is   the rank of $A$, and $A^{\HH}$ and $A^{\T}$ stand for the conjugate transpose and transpose of  $A$, respectively. The column space of $A$  is $\text{span}(A)$. We will use MATLAB-like convention to represent the sub-matrix $A_{(\bbI_1,\bbI_2)}$ of $A$, consisting of intersections of rows and columns indexed by
$\bbI_1 \subseteq\{1,2,...,m\}$ and $\bbI_2 \subseteq\{1,2,...,n\}$, respectively.
%%%%%%%%%%%%%%%%%%%%%

\section{Minimal perturbation framework for RGEP, a review}\label{sec:RGEP:MPF}
In this section, we briefly review the state-of-the-art treatments for RGEP, i.e.,
\eqref{eq:RMEP} with $k=1$.
The basic idea from these treatments motivates us ways to deal with the general case $k>1$ in
the following sections.
%We will highlight similarities and the differences/subtleties that come with $k>1$.

First we  simplify notation by dropping off the subscripts to the coefficient matrices in
\eqref{eq:RMEP} since $k=1$. Specifically, we state RGEP  as
\begin{equation}\label{eq:RGEP-simplified}
A\bx=\lambda B\bx,
\end{equation}
where $A,\,B\in\bbC^{m\times n}$ with $m\ge n$.
RGEP in general may not admit a solution, which is clear from the Kronecker canonical form for matrix pair
$(A, B)$ \cite{gant:1959I,gant:1959II}.
%Also for an RGEP arising from a real-world application,
%the coefficient matrices $A,\,B$ are usually contaminated. In such a situation,
%chances are great that there are eigenpairs $(\lambda_j,\bx_j)$ with noiseless albeit unknown $A$ and $B$,
%but the contaminations in $A$ and $B$ cause some or all of these eigenpairs to disappear.
%
%As explained in \cite{itmu:2016},
%RGEP may not admit a solution. This can happen many real-world situations. For example, it can happen that
%for
%there are $n$ eigenpairs $(\lambda_j,\bx_j)$ to begin exist,
%
Nonetheless, real-world applications demand solutions regardless. In this aspect,
some  numerical treatments have been designed to look for approximate eigenpairs
rather than satisfy \eqref{eq:RGEP-simplified} exactly
\cite{bole:1990,boeg:2005,deka:1987,edek:1997,elmg:2004,legc:2008,maoc:1991,sclv:2006,stew:1994,wrtr:2002}.
%In general, any given RGEP as is may not have eigenpairs.
In particular, the authors of \cite{boeg:2005,chgo:2006} proposed the following minimal perturbation framework:
\begin{subequations}\label{eq:RGEP:eta}
\begin{align}
\eta_{\ell}:=&\inf_{\what{A}, \what{B},  \{(\lambda_j,\bx_j)\}_{j=1}^{\ell}} \quad
               \big\| \big[ \what{A}-A, \what{B}-B \big] \big\|_{\F}^2 \\
		%\mbox{s.t.} \quad&\what{A},\what{B}\in \bbC^{m\times n},{(\lambda_j,\bx_j)}_{j=1}^{\ell}\subseteq \bbC\times \bbC^n,\\
\mbox{s.t.} &\quad\left\{
        \begin{array}{l}
        \mbox{$\what{A}\bx_j=\lambda_j \what{B}\bx_j$ for $j=1,2,\ldots,\ell$, and $\lambda_j\ne \lambda_i$ for $j\ne i$},\\
		\mbox{$\bx_1,\bx_2,\ldots,\bx_\ell$ are linearly independent,}
        \end{array}\right.
\end{align}
\end{subequations}
for the determination of $\ell$ approximate eigenpairs
$\{(\lambda_j,\bx_j)\}_{j=1}^{\ell}$, where $\ell<n$ is a preselected integer.
It is noted that these  eigenpairs depend on $\ell$, and any eigenpair for $\ell$ may no longer be one for $\ell+1$, unfortunately.

Boutry, Elad, Golub, and Milanfar \cite{boeg:2005} analyzed  the infimum $\eta_1$ and proposed an alternating and convergent numerical algorithm.
For $\ell=n$, Ito and Murota \cite{itmu:2016} showed that
$\eta_n$ can be reformulated equivalently as
\begin{equation}\label{eq:RGEP:eta:ell=n}
\eta_n=\inf_{Z\in\bbC^{n\times n}}\quad \big\| \big[ \what{A}-A, \what{B}-B \big] \big\|_{\F}^2\quad
		\mbox{s.t.}\quad \what A=\what BZ,
\end{equation}
%\begin{equation}\label{eq:RGEP:eta:ell=n}
%	  \begin{aligned}
%%		\eta_n=
%\inf_{Z\in\bbC^{n\times n}\quad &\left\| \big[ \what{A}-A, \what{B}-B \big] \right\|_{\F}^2 \\
%		\mbox{s.t.}\quad& \what A=\what BZ
%	\end{aligned}
%\end{equation}
%where $\what A=A+R,~\what B=B+E$,
and moreover, the $n$ approximate eigenpairs can be computed from  GEP
\begin{equation}\label{eq:tSVDGEP}
	V_{11}^{\HH} \bx = \lambda V_{21}^{\HH} \bx,
\end{equation}
where $V_{11},V_{21}\in \bbC^{n\times n}$ are  from  the SVD of $[A,B]\in\bbC^{m\times 2n}$ as follows:
$\wtd m = \min\{m,2n\}$,
 \begin{equation}\label{eq:RGEP_TLS(n)_SVD}
[A,B]=%\kbordermatrix{ &\sss n &\sss \wtd m-n\\
%                    & U_1 & U_2}
          U          \times
        \kbordermatrix{ &\sss n &\sss n\\
		\sss n         & \Sigma_1 & 0 \\
		\sss \wtd m-n & 0 & \Sigma_2 }\times
	V^{\HH}, \quad
V=\kbordermatrix{ &\sss n &\sss n\\
		\sss n         & V_{11}&V_{12}\\
		\sss n & V_{21}&V_{22} },
\end{equation}
 $U\in\bbC^{m\times \wtd m}$,
%$$
%V=\kbordermatrix{ &\sss n &\sss n\\
%		\sss n         & V_{11}&V_{12}\\ \\
%		\sss n & V_{21}&V_{22} }
%$$
and $\Sigma_1=\diag(\sigma_1,\sigma_2,\dots,\sigma_{n})$ consisting of the $n$ largest singular values
of $[A,B]$. It is noted that the matrix pair $(V_{11}^{\HH}, V_{21}^{\HH})$ is regular, i.e.,
$\det(V_{11}^{\HH} - \lambda V_{21}^{\HH})\not\equiv 0$ for $\lambda\in\bbC$, which is important because
the eigen-system of a regular matrix pair is well-defined and continuous \cite{li:2014HLA,stsu:1990}.
If $V_{21}^{\HH}$ is singular,
GEP \eqref{eq:tSVDGEP} will have $\infty$ as an eigenvalue, in which case it is more suitable to
express $\lambda$ formally as a pair $(\alpha,\gamma)$ in the sense $\lambda=\alpha/\gamma$, and,
without loss of generality, we may assume $|\gamma|^2+|\alpha|^2=1$ and $\gamma\ge 0$.

%in non-increasing order.

It is safe to say that the studies on \eqref{eq:RGEP:eta} for the cases $\ell=1$ and $n$
are relatively complete and satisfactory: there are effective ways to compute $\eta_{\ell}$ and
the associated eigenpairs. However, the case for $1<\ell<n$
is rather complicated, although there are some results to estimate $\eta_{\ell}$ through often expensive
optimization techniques \cite{chgo:2006,lild:2021,lilv:2020}.
One option is to first seek  $n$ approximate eigenpairs by GEP \eqref{eq:tSVDGEP}, and then pick $\ell$ out of the
$n$ approximate eigenpairs with the smallest normalized residuals.
%$$
%\frac {\|A\bx-\lambda B\bx\|_2}{\|A\|_2+|\lambda|\cdot\|B\|_2}
%=\frac {\|\gamma A\bx-\alpha B\bx\|_2}{|\gamma|\cdot\|A\|_2+|\alpha|\cdot\|B\|_2},
%$$
%where it is assumed $\|\bx\|_2=1$

\section{Minimal perturbation formulation for RMEP}\label{sec:RMEP:MPF}
Similarly to RGEP \eqref{eq:RGEP-simplified} which is a special case of RMEP \eqref{eq:RMEP} with $k=1$,
RMEP \eqref{eq:RMEP} for $k>1$ may not possess an eigen-tuple $(\lambda_1,\ldots,\lambda_k,\bx_1,\ldots,\bx_k)$, either.
For that reason,
RMEP \eqref{eq:RMEP} has to be modified.

A straightforward extension of the  minimal perturbation formulation \eqref{eq:RGEP:eta} for RGEP
%in \cref{sec:RGEP:MPF} \cite{boeg:2005,chgo:2006,itmu:2016}
to RMEP \eqref{eq:RMEP} %in general
is as follows:
%a modification to \eqref{eq:RMEP} for approximate solutions:
finding a minimal perturbations to $A_i, B_{is}\in \bbC^{m_i\times n_i}$
%for
%$1\le i,\,s\le k$
%in \eqref{eq:RMEP} in the sense
such that
%for the following problem
\begin{subequations}\label{eq:RMEP-approx:ell}
\begin{align}		
&\inf_{  \{(\what A_i, \what B_{i1},\ldots, \what B_{ik})\}_{i=1}^k,
           \{(\lambda_{1,j},\ldots,\lambda_{k,j},\bx_{1,j},\ldots,\bx_{k,j})\}_{j=1}^{\ell}}
         \sum_{i=1}^k\big\| \big[\what A_i-A_i, \what B_{i1}-B_{i1},\ldots,\what B_{ik}-B_{ik}\big] \big\|_{\F}^2  \\
		&\text{s.t.}\quad \left\{
        \begin{array}{l}
		  \widehat{A}_i \bx_{i,j} = \sum_{s=1}^k\lambda_{s,j} \widehat{B}_{is} \bx_{i,j}\quad
              \mbox{for $1\le i\le k,\, 1\le j\le \ell$}, \\
		  \mbox{and $\bx_{1,j}\otimes\bx_{2,j}\otimes\cdots\otimes \bx_{k,j}$ for $j=1,\ldots\ell$ are linearly independent},
        \end{array}\right.
\end{align}
\end{subequations}
where $\ell$ is the number of the desired approximate eigen-tuples
$$\{(\lambda_{1,j},\ldots,\lambda_{k,j},\bx_{1,j},\ldots,\bx_{k,j})\}_{j=1}^\ell$$ and is application-dependent.
It is noted that
these eigen-tuples determined by \eqref{eq:RMEP-approx:ell} depend on $\ell$,
and any eigen-tuple for $\ell$ may no longer be one for $\ell+1$, unfortunately.

This straightforward extension indeed works equally well in the case
of $\ell=1$ (i.e., for one approximate eigen-tuple) as it has for RGEP ($k=1$).
\Cref{sec:RMEP:ell=1} contains our detailed investigation for $\ell=1$ and any $k\ge 1$.
It worth mentioning that our treatment requires essentially no additional care when it comes to the case $k=1$ (i.e., for RGEP) or any general $k>1$. When specialized to RGEP (i.e., $k=1$), our argument significantly simplifies
the development in \cite{boeg:2005} and our results read much more elegant, too (see \Cref{rk:alg-RMEP2RGEP:ell=1}).

However, such a straightforward extension does not work so well as is in the case
of $k>1$ and $\ell=N$ (i.e., for a complete set of approximate eigen-tuples), revealing inherent differences between RGEP and RMEP in general. Taking a hint from treating RGEP, we properly modify \eqref{eq:RMEP-approx:ell}
to make it possible to create necessary theory and a numerical method.

%In this section, we investigate RMEP in its generality
%through a minimal perturbation formulation. The basic idea is inspired by that in \cref{sec:RGEP:MPF}
%\cite{boeg:2005,chgo:2006,itmu:2016}.
%
%
%
%
%Our   investigation will reveal
% In particular, our treatment for one approximate eigen-tuple,
%when adapted to the case of RGEP, significantly simplify the technical developments in \cite{boeg:2005}
%and is much more elegant.

An equivalent homogeneous formulation of RMEP~\eqref{eq:RMEP}
is
%\begin{equation}\label{eq:RMEP:k=2:homo}
%	\left\{ \begin{aligned}
%		\gamma A_1 \bx &= \alpha  B_1 \bx + \beta C_1 \bx, \\
%		\gamma A_2 \by &= \alpha B_2 \by + \beta C_2 \by,
%	\end{aligned}\right.\quad
%\end{equation}
\begin{equation}\label{eq:RMEP:k=2:homo}
	\left\{ \begin{aligned}
		\gamma A_1 \bx_1 &= \alpha_1 B_{11} \bx_1 + \alpha_2 B_{12} \bx_1 + \cdots + \alpha_k B_{1k} \bx_1, \\
		&\vdots\\
		\gamma A_k \bx_k &= \alpha_1 B_{k1} \bx_k + \alpha_2 B_{k2} \bx_k + \cdots + \alpha_k B_{kk} \bx_k,
	\end{aligned}\right.\quad
\end{equation}
by expressing each $\lambda_i$ as a pair $(\alpha_i,\gamma)$ in the sense formally
\begin{equation}\label{eq:(lambda,mu):home}
\lambda_i=\alpha_i/\gamma\quad\mbox{for $1\le i\le k$}.
\end{equation}
Without loss of generality, we may assume
$
\mbox{$|\gamma|^2+\sum_{i=1}^k|\alpha_i|^2=1$ and $\gamma\ge 0$}.
$
This homogeneous formulation
makes it possible to include infinite eigenvalues $\lambda_i=\infty$ as a result of $\gamma=0$ and
at the same time can provide with an insight towards the infinities through relative magnitudes among all $\alpha_i$.
Such an equivalent homogeneous formulation will prove handy
during our
detailed investigation for $\ell=N$ in \cref{sec:RMEP:ell=N}.
% that contains  and any $k\ge 1$.

%Apart from some technically more involved expressions, our treatments in this section and
%in \cref{sec:RMEP:ell=1,sec:RMEP:ell=N} can be extended straightforwardly to the general case $k>2$, as we will comment
%later at proper places.
%%In \cref{sec:RMEP:k>1}, we will briefly outline the key modifications of what we have done for $k=2$ for the case of
%%RMEP \eqref{eq:RMEP} for $k>1$ in general.

%\subsection{Minimal perturbation framework for RMEP}\label{ssec:minperRMEP}
%We will seek to extend the basic idea for dealing with RGEP in section~\ref{sec:RGEP:MPF} to RMEP.
%It is assumed {$m_i\ge n_i$} for $1\le i\le k$.
%%As for RGEP, RMEP \eqref{eq:RMEP} for $k>1$ as is
%%may not possess an eigenpair, either.
%%%Moreover, in certain cases, existing eigenpairs can be lost under perturbations or noise.
%%For that reason, it has to be modified.
%
%
%
%
%As we just mentioned, \eqref{eq:RMEP-approx:ell} is a straightforward extension of the one for RGEP to the current case.
%However, with $k>1$ comes new challenges. Most notably later in this paper, we will see that we will have
%to modify \eqref{eq:RMEP-approx:ell} when it comes to look for a complete set of $N=n_1n_2\cdots n_k$ approximate eigen-tuple, whereas in section~\ref{sec:RGEP:MPF} there is no need to modify \eqref{eq:RGEP:eta} for the case $\ell=n$ due to
%the fact that  \eqref{eq:RGEP:eta} for $\ell=n$ and \eqref{eq:RGEP:eta:ell=n} are equivalent \cite{itmu:2016}.
%
%
%In what follows, we will spread our investigations for $\ell=1$ and $\ell=N$ into two separate sections.

%%%%%%%%%%%%%%%%%%%%%%%%%%
\section{Seek one approximate eigen-tuple}\label{sec:RMEP:ell=1}
When only one approximate eigen-tuple is desired, i.e., $\ell=1$ in \eqref{eq:RMEP-approx:ell},
we can simplify notation to restate
\eqref{eq:RMEP-approx:ell} as
%\begin{equation}\label{eq:RMEP-approx:ell=1:theta1:0}
%\begin{aligned}
%\theta_1 =&\inf_{  \{(\what A_i, \what B_i, \what{C}_i)\}_{i=1}^2, (\lambda,\mu,\bx,\by)}
%    \sum_{i=1}^2\big\| \big[\widehat{A}_i - A_i, \widehat{B_i} - B_i, \widehat{C}_i - C_i\big] \big\|_{\F}^2\\
%		\text{s.t.}
%		& \quad \what A_1 \bx = \lambda \what B_1 \bx + \mu \what{C}_1 \bx,~ \what A_2 \by = \lambda \what B_2 \by + \mu \what{C}_2 \by,\\
%		&\hm \left\| \bx\right\| _2 =\left\| \by\right\| _2 =1, ~\bx \in \bbC^{n_1}, \by \in \bbC^{n_2}, \lambda, \mu \in \bbC,
%	\end{aligned}
%\end{equation}
\begin{subequations}\label{eq:RMEP-approx:ell=1:theta1:0}
\begin{align}
\theta_1 :=&\inf_{  \{(\what A_i, \what B_{i1},\ldots, \what B_{ik})\}_{i=1}^k,
           (\lambda_1,\ldots,\lambda_k,\bx_1,\ldots,\bx_k)}
           \sum_{i=1}^k\big\| \big[\what A_i-A_i, \what B_{i1}-B_{i1},\ldots,\what B_{ik}-B_{ik}\big] \big\|_{\F}^2
               \label{eq:RMEP-approx:ell=1:theta1:0a} \\
\text{s.t.}
		& \quad \left\{
        \begin{array}{l}
          \widehat{A}_i \bx_i = \sum_{s=1}^k\lambda_s \widehat{B}_{is} \bx_i\quad
              \mbox{for $1\le i\le k$}, \\
          \|\bx_i\| _2 =1,\, ~\bx_i \in \bbC^{n_i},\, \lambda_i\in \bbC\quad\mbox{for $1\le i\le k$},
        \end{array}\right.   \label{eq:RMEP-approx:ell=1:theta1:0b}
\end{align}
\end{subequations}
where we have also normalized all $\bx_i$  to be unit vectors.

\subsection{Reformulations of \eqref{eq:RMEP-approx:ell=1:theta1:0}}\label{ssec:reform}
\Cref{thm:RMEP-approx:ell=1} below reformulates \eqref{eq:RMEP-approx:ell=1:theta1:0}
and it generalizes \cite[Theorem 4]{boeg:2005} which is for RGEP.
Our eventual reformulations of \eqref{eq:RMEP-approx:ell=1:theta1:0} is given in \Cref{thm:RMEP-approx:ell=1a}
which leads to our alternating
algorithm for solving \eqref{eq:RMEP-approx:ell=1:theta1:0}.

\begin{theorem}\label{thm:RMEP-approx:ell=1}
For $\theta_1$ in  \eqref{eq:RMEP-approx:ell=1:theta1:0}, we have
\begin{subequations}\label{eq:RMEP-approx:ell=1:theta1:1}
\begin{align}
&\theta_1=\inf_{(\lambda_1,\ldots,\lambda_k,\bx_1,\ldots,\bx_k)}
          \frac {\sum_{i=1}^k\|A_i \bx_i - \sum_{s=1}^k\lambda_s B_{is} \bx_i\|_2^2}
          {1+\sum_{i=1}^k|\lambda_i|^2}
              \label{eq:RMEP-approx:ell=1:theta1:1a}\\
&\mbox{\rm s.t.}\,\,
 \|\bx_i\| _2 =1,\, ~\bx_i \in \bbC^{n_i},\, \lambda_i\in \bbC\quad\mbox{for $1\le i\le k$}.   \label{eq:RMEP-approx:ell=1:theta1:1b}
\end{align}
\end{subequations}
\end{theorem}

\begin{proof}
Let, for $1\le i\le k$,
\begin{equation}\label{eq:EFM}
E_i=\what A_i-A_i,\quad
F_{is}=\what B_{is}-B_{is}\quad \mbox{for $1\le s\le k$}.
\end{equation}
For any $(\lambda_1,\ldots,\lambda_k,\bx_1,\ldots,\bx_k)$ subject to \eqref{eq:RMEP-approx:ell=1:theta1:1b},
we get from $\widehat{A}_i \bx_i = \sum_{s=1}^k\lambda_s \widehat{B}_{is} \bx_i$ that
\begin{equation}\nonumber%\label{eq:constrait-1a}
E_i \bx_i - \sum_{s=1}^k\lambda_s F_{is} \bx_i
=-\Big(A_i  - \sum_{s=1}^k\lambda_s B_{is}\Big)\bx_i=:-\Bf_i.
\end{equation}
Let $\tau=\sqrt{1+\sum_{s=1}^k|\lambda_s|^2}$. We have
\begin{align}
\|\Bf_i\|_2&=\Big\|E_i \bx_i - \sum_{s=1}^k\lambda_s F_{is} \bx_i\Big\|_2  \nonumber\\
           &\le\|E_i\|_2\|\bx_i\|_2+\sum_{s=1}^k|\lambda_s|\| F_{is}\|_2\|\bx_i\|_2 \nonumber\\
           &=\|E_i\|_2+\sum_{s=1}^k|\lambda_s|\| F_{is}\|_2 \nonumber\\
           &\le\tau\sqrt{\|E_i\|_2^2+\sum_{s=1}^k\| F_{is}\|_2^2}
           \le\tau\sqrt{\|E_i\|_{\F}^2+\sum_{s=1}^k\| F_{is}\|_{\F}^2}. \label{eq:constrait-1b}
\end{align}
It can be verified that all inequalities in \eqref{eq:constrait-1b} become equalities by
\begin{equation}\label{eq:constrait-1c}
E_i=\frac 1{\tau}\cdot {\Bf_i\bx_i^{\HH}}, ~
F_{is}=-\frac {\bar\lambda_s}{\tau}\cdot {\Bf_i\bx_i^{\HH}}\quad \mbox{for $1\le s\le k$}.
\end{equation}%
%Similarly it follows from $\what A_2\by=\lambda\what B_2\by+\mu\what C_2\by$ that
%\begin{align}
%&E_2\by-\lambda F_2\by-\mu G_2\by=-(A_2-\lambda B_2-\mu C_2)\by=:-\Bf_2, \nonumber \\
%&\|\Bf_2\|_2
%    \le\tau
%               \sqrt{\|E_2\|_{\F}^2+\| F_2\|_{\F}^2+\|G_2\|_{\F}^2}, \label{eq:constrait-2b}
%\end{align}
%and inequality \eqref{eq:constrait-2b} becomes an equality by
%\begin{equation}\label{eq:constrait-2c}
%E_2=\frac 1{\tau}\cdot {\Bf_2\by^{\HH}}, ~
%F_2=\frac {\bar\lambda}{\tau}\cdot {\Bf_2\by^{\HH}}, ~
%G_2=\frac {\bar\mu}{\tau}\cdot {\Bf_2\by^{\HH}}.
%\end{equation}
Therefore under the constraints in \eqref{eq:RMEP-approx:ell=1:theta1:0b}, we have
\begin{align}
\sum_{i=1}^k\big\| \big[\what A_i-A_i, \what B_{i1}-B_{i1},\ldots,\what B_{ik}-B_{ik}\big] \big\|_{\F}^2
  &=\sum_{i=1}^k\Big[\|E_i\|_{\F}^2+\sum_{s=1}^k\| F_{is}\|_{\F}^2\Big] \nonumber\\
  &\ge\frac {\sum_{i=1}^k\big\|\Bf_i\big\|_2^2}{1+\sum_{s=1}^k|\lambda_s|^2} \label{eq:obj>=}\\
  &=\frac {\sum_{i=1}^k\big\|A_i \bx_i - \sum_{s=1}^k\lambda_s B_{is} \bx_i\big\|_2^2}
          {1+\sum_{s=1}^k|\lambda_s|^2}, \nonumber
\end{align}
and the equality sign in \eqref{eq:obj>=} is achieved by $\{(E_i,F_{i1},\ldots, F_{ik})\}_{i=1}^k$ in \eqref{eq:constrait-1c}. Hence optimization problem
\eqref{eq:RMEP-approx:ell=1:theta1:0} can be rephrased as \eqref{eq:RMEP-approx:ell=1:theta1:1}.
\end{proof}

Still in \eqref{eq:RMEP-approx:ell=1:theta1:1}, $\theta_1$ is an infimum, i.e., not necessarily attainable by
some eigen-tuple $(\lambda_1,\ldots,\lambda_k,\bx_1,\ldots,\bx_k)$. This can be overcome if we express
each $\lambda_i$
formally as  a pair $(\alpha_i,\gamma)$ in the sense of \eqref{eq:(lambda,mu):home}
%$$
%\lambda=\alpha/\gamma, \quad
%\mu=\beta/\gamma
%$$
subject to $|\gamma|^2+\sum_{i=1}^k|\alpha_i|^2=1$ and $\gamma\ge 0$. Then
the objective function of \eqref{eq:RMEP-approx:ell=1:theta1:1} becomes
$$
\sum_{i=1}^k\big\|\gamma A_i \bx_i - \sum_{s=1}^k\alpha_s B_{is} \bx_i\big\|_2^2
%\|(\gamma A_1-\alpha B_1-\beta C_1)\bx\|_2^2+\|(\gamma A_2-\alpha B_2-\beta C_2)\by\|_2^2,
$$
and consequently we have

\begin{theorem}\label{thm:RMEP-approx:ell=1a}
For $\theta_1$ in  \eqref{eq:RMEP-approx:ell=1:theta1:0} and hence in \eqref{eq:RMEP-approx:ell=1:theta1:1} too, we have
%The infimum $\theta_1$ of  \eqref{eq:RMEP-approx:ell=1:theta1:0} can be given as
\begin{subequations}\label{eq:RMEP-approx:ell=1:theta1:2}
\begin{align}
&\theta_1=\min_{((\alpha_1,\gamma),\ldots,(\alpha_k,\gamma),\,\bx_1,\ldots,\bx_k)}
          \sum_{i=1}^k\big\|\gamma A_i \bx_i - \sum_{s=1}^k\alpha_s B_{is} \bx_i\big\|_2^2
              \label{eq:RMEP-approx:ell=1:theta1:2a}\\
&\mbox{\rm s.t.}\,\,
 \|\bx_i\|_2=1,\,\,|\gamma|^2+\sum_{s=1}^k|\alpha_s|^2=1\,\,\mbox{for $1\le i\le k$},\,\gamma\ge 0.
        \label{eq:RMEP-approx:ell=1:theta1:2b}
\end{align}
\end{subequations}
\end{theorem}

\begin{proof}
It follows from \Cref{thm:RMEP-approx:ell=1} and the discussion proceeding to this theorem that
$$
\theta_1=\inf_{((\alpha_1,\gamma),\ldots,(\alpha_k,\gamma),\,\bx_1,\ldots,\bx_k)}
          \sum_{i=1}^k\big\|\gamma A_i \bx_i - \sum_{s=1}^k\alpha_s B_{is} \bx_i\big\|_2^2
$$
subject to the  constraints  in \eqref{eq:RMEP-approx:ell=1:theta1:2b}. The infimum here is attainable
because the constraints in \eqref{eq:RMEP-approx:ell=1:theta1:2b} prescribe a bounded and closed set in $\bbC^{k+1+\sum_{i=1}^kn_k}$.
\end{proof}

\subsection{An alternating scheme}\label{ssec:alternating}
Reformulation in \eqref{eq:RMEP-approx:ell=1:theta1:2} for $\theta_1$ naturally lends itself
to a numerical scheme to optimize the objective function there alternatingly over
$((\alpha_1,\gamma),\ldots,(\alpha_k,\gamma))$
and $(\bx_1,\ldots,\bx_k)$.
In fact, given $((\alpha_1,\gamma),\ldots,(\alpha_k,\gamma))$,
the optimal $(\bx_1,\ldots,\bx_k)$ are given by each $\bx_i$ being
the unit right singular vector of
\begin{equation}\label{eq:Ri}
R_i=\gamma A_i  - \sum_{s=1}^k\alpha_s B_{is}
\end{equation}
associated with its smallest singular value. Note that the dependency of $R_i$ on $[\gamma,\alpha_1,\dots,\alpha_k]^{\T}$ is suppressed for presentation clarity.

On the other hand, given $(\bx_1,\ldots,\bx_k)$, we will show, in what follows, that the minimum of the objective function
is the smallest eigenvalue of some $(k+1)$-by-$(k+1)$ Hermitian matrix whose corresponding eigenvector yields
the optimal $[\gamma,\alpha_1,\ldots,\alpha_k]^{\T}$.
To that end, we write
\begin{equation}\label{eq:Si(xi):v}
R_i\bx_i
=\underbrace{\Big[A_i\bx_i,-B_{i1}\bx_i,\ldots,-B_{ik}\bx_i\Big]}_{=:S_i(\bx_i)}
                        \underbrace{\begin{bmatrix}
                           \gamma \\
                           \alpha_1\\
                           \vdots\\
                           \alpha_k
                         \end{bmatrix}}_{=:\bv}.
\end{equation}
%Let
%$$
%\bv=\begin{bmatrix}
%                           \gamma \\
%                           \alpha \\
%                           \beta
%                         \end{bmatrix}
%                         =:\begin{bmatrix}
%                                                       v_1 \\
%                                                       v_2 \\
%                                                       v_3
%                                                     \end{bmatrix}
%                         \in\bbC^3.
%$$
Evidently, $\|\bv\|_2^2=|\gamma|^2+\sum_{s=1}^k|\alpha_s|^2=1$.
Furthermore, the objective function of \eqref{eq:RMEP-approx:ell=1:theta1:2} can be expressed as
\begin{equation}\label{eq:S}
\sum_{i=1}^k\|S_i(\bx_i)\bv\|_2^2
   =\bv^{\HH}\Big[\underbrace{\sum_{i=1}^kS_i(\bx_i)^{\HH}S_i(\bx_i)}_{=:H(\bx_1,\ldots,\bx_k)}\Big]\bv
%   =:\bv^{\HH} S(\bx,\by)\bv
\end{equation}
whose minimum is the smallest eigenvalue of $H(\bx_1,\ldots,\bx_k)$
%\begin{equation}\label{eq:S}
%S(\bx,\by):=\big[S_1(\bx)^{\HH}S_1(\bx)+S_2(\by)^{\HH}S_2(\by)\big]\in \bbC^{3\times 3}
%\end{equation}
attained at the associated unit eigenvector $\bv$. It can be seen that $H(\bx_1,\ldots,\bx_k)$ is positive semidefinite.

\Cref{alg:RMEP-approx:ell=1} outlines our alternating scheme to solve \eqref{eq:RMEP-approx:ell=1:theta1:0}
via \eqref{eq:RMEP-approx:ell=1:theta1:2}.

We now state the KKT conditions of \eqref{eq:RMEP-approx:ell=1:theta1:2}:
\begin{subequations}\label{eq:RMEP-approx:ell=1:KKT}
\begin{alignat}{3}\label{eq:RMEP-approx:ell=1:KKT-1}
R_i^{\HH}R_i\bx_i&= \bx_i \omega_i\,\,&&\mbox{with}\,\,
     &\omega_i&=\|R_i\bx_i\|_2^2\quad\mbox{for $1\le i\le k$},\\
H(\bx_1,\ldots,\bx_k)\bv&=\bv\omega_{k+1}\,\,&&\mbox{with}\,\,
     &\omega_{k+1}&=\bv^{\HH}H(\bx_1,\ldots,\bx_k)\bv. \label{eq:RMEP-approx:ell=1:KKT-2}
\end{alignat}
%where  %$\omega_i\ge 0$ for $1\le i\le 3$ and can be expressed as
%\begin{equation}\label{eq:RMEP-approx:ell=1:KKT-2}
%\omega_1=\|R_1\bx\|_2^2,\quad
%\omega_2=\|R_2\by\|_2^2,\quad
%\omega_3=\bv^{\HH}S(\bx,\by)\bv.
%\end{equation}
\end{subequations}

\begin{algorithm}[t]
\caption{Alternating scheme for solving \eqref{eq:RMEP-approx:ell=1:theta1:0}}
	\label{alg:RMEP-approx:ell=1}
\begin{algorithmic}[1]
\REQUIRE $A_i, B_{i1},\ldots,B_{ik}\in \bbC^{m_i\times n_i}$ with $m_i\ge n_i$ for $1\le i\le k$;
\ENSURE  the optimal solution $(\lambda_1,\ldots,\lambda_k,\bx_1,\ldots,\bx_k)$ for \eqref{eq:RMEP-approx:ell=1:theta1:0}.
%\STATE \textbf{Input:}   $ A_i, B_i, C_i\in \bbC^{m_i\times n_i} \in \bbC^{m_i\times n_i}$ with $m_i\ge n_i,~i=1,2$.
%\STATE \textbf{Output:}  the optimal solution $(\lambda, \mu, \bx, \by)$ for \eqref{eq:RMEP-approx:ell=1:theta1:0}.
\STATE choose  an initial guess $(\lambda_1,\ldots,\lambda_k)$;
\STATE $\tau=\sqrt{1+\sum_{s=1}^k|\lambda_s|^2}$, $\gamma=1/\tau$, $\alpha_i=\lambda_i/\tau$ for $1\le i\le k$;
\WHILE{not converged}
	\STATE compute the unit right singular vector $\bx_i$ of $R_i=\gamma A_i  - \sum_{s=1}^k\alpha_s B_{is}$
               associated with its smallest singular value  for $1\le i\le k$;
	\STATE form matrix $H(\bx_1,\ldots,\bx_k)$ as in \eqref{eq:S};
	\STATE compute the unit eigenvector $\bv=[\gamma,\alpha_1,\ldots,\alpha_k]^{\T}$ (with $\gamma\ge 0$)
           of $H(\bx_1,\ldots,\bx_k)$ associated with its smallest
               eigenvalue which is the most-updated $\theta_1$;
\ENDWHILE		
\IF{$\gamma > 0$}
	\STATE  $\lambda_i=\alpha_i/\gamma$ for $1\le i\le k$;
    \RETURN $(\lambda_1,\ldots,\lambda_k, \bx_1,\ldots,\bx_k)$;
\ELSE
	\RETURN $((\alpha_1,0), (\alpha_k,0), \bx_1,\ldots,\bx_k)$ and declaim likely $\theta_1$ is an infimum
            in \eqref{eq:RMEP-approx:ell=1:theta1:0}.
\ENDIF
\end{algorithmic}
\end{algorithm}

%According to Corollary \ref{cor:theta1} and \eqref{eq:monotonic}, it is easy to see the following convergence property.

%\begin{theorem}
%For Algorithm \ref{alg:RMEP-approx:ell=1}, if $v_1=\bv^{\T}\be_1\ne 0$ during the iteration, then the sequence of the objective function $\theta(\lambda,\mu)$ in \eqref{eq:RMEP-approx:ell=1:theta1:2} decreasing and convergent.
%\end{theorem}

%We have the following remarks for Algorithm \ref{alg:RMEP-approx:ell=1}:

\begin{remark}
Some comments for Algorithm \ref{alg:RMEP-approx:ell=1} are in order:
\begin{enumerate}[(1)]
\item A simple and cheap stopping criterion to use at line 3 is if the change
      $|\theta^{(j+1)}_1-\theta^{(j)}_1|$ in the objective value $\theta_1$
      between consecutive iterations, i.e., $\theta^{(j+1)}_1$, and $\theta^{(j)}_1$,  is relatively less than some prescribed tolerance $\epsilon$. In our experiments later, we use
      \begin{equation}\label{eq:reltheta1}
      |\theta^{(j+1)}_1-\theta^{(j)}_1|\le (\theta^{(j+1)}_1+1)\,\epsilon,
      \end{equation}
      with $\epsilon=10^{-6}$.
      %\Red{
      There is a possibility that this simple stopping criterion may lead to false convergence if
      at some point during the iteration $\theta_1$ changes little and yet convergence has not occurred.
      Hence it is recommended that, when this simple stopping criterion is met,
      one  also check if the KKT conditions in \eqref{eq:RMEP-approx:ell=1:KKT} are satisfied approximately
      by verifying if the normalized KKT residual
      \begin{equation}\label{eq:KKT}
     \epsilon_{\rm KKT}:=\sum_{i=1}^k\frac {\|R_i^{\HH}R_i\bx_i-\bx_i\omega_i\|_2}{\xi_i}
      +\frac {\|H(\bx_1,\ldots,\bx_k)\bv-\bv\omega_{i+1}\|_2}{\sum_{s=1}^k\xi_s}
      \end{equation}
      is small enough,
      where $\xi_i=\|A_i\|_2^2+\sum_{s=1}^k\|B_{is}\|_2^2$ for $1\le i\le k$.
      %, and $\check\epsilon$ is another preselected tolerance       which usually is different from $\epsilon$ earlier.
      %}

\item In actual computations, it is unlikely to render $\gamma= 0$ exactly but rather possibly
      $\gamma$ is in the order of the machine roundoff. When that is the case, if all $\alpha_i$
      are far bigger than the machine roundoff in magnitude then all
      $\lambda_i=\alpha_i/\gamma$ are huge. If some of the $\alpha_i$'s are about
      the machine roundoff then the corresponding $\lambda_i$ would deceitfully appear modest but otherwise is not meaningful at all.
      In any case, each $\alpha_i$ provides useful information even if $\gamma$ is in the order of the machine roundoff.
\item The output $(\lambda_1,\ldots,\lambda_k, \bx_1,\ldots,\bx_k)$ can be used to construct the minimal perturbations
    $$
    \what A_i=A_i+E_i,\quad
    \what B_{is}=B_{is}+F_{is}\quad\mbox{for $1\le i\le k$}
    $$
    according to \eqref{eq:constrait-1c}, or in terms of $\gamma,\alpha_1,\ldots,\alpha_k$,
    $$
    E_i=\gamma\cdot\tilde \Bf_i\bx_i^{\HH}, \quad
    F_{is}=-\bar\alpha_i\cdot\tilde \Bf_i\bx_i^{\HH}
    \quad\mbox{for $1\le i\le k$ and $1\le s\le k$},
    $$
    where $\tilde \Bf_i=\gamma A_i \bx_i - \sum_{s=1}^k\alpha_s B_{is} \bx_i$.

%\item Our transformation of \eqref{eq:RMEP-approx:ell=1:theta1:0} into \eqref{eq:RMEP-approx:ell=1:theta1:2} can be extended to multiparameter RMEP for $k>2$. Particularly,  our treatment for determining approximate eigenvalue tuple $(\lambda,\mu)$ for approximate eigenvector tuple $(\bx,\by)$ simplifies greatly the one used originally for RGEP in \cite{boeg:2005}.

%\item At line 11, we perturb $S(\bx_{\lambda,\mu},\by_{\lambda,\mu})$ slightly once the first element  of
%      $\bv$ is $0$. Asymptotically, this happens when the infimum $\theta_1$ is unattainable. During the iteration, however, we shall see in the next subsection that with probability one, $v_1\ne 0$. This also partially implies that the infimum $\theta_1$ is achievable almost surely.
\item %If, at the conclusion of \Cref{alg:RMEP-approx:ell=1}, $\gamma=0$ (or about the machine roundoff),
%      $\theta_1$ is likely a local infimum in \eqref{eq:RMEP-approx:ell=1:theta1:0}.
      There is no guarantee that \Cref{alg:RMEP-approx:ell=1} returns a global infimum/minimum.
      One option is to rerun
      \Cref{alg:RMEP-approx:ell=1} with one or more different initial guesses
      $(\lambda_1,\ldots,\lambda_k)$ to increase the possibility of getting a global infimum/minimum.
      % (minimum if $\gamma>0$).
%Corresponding,  the best $(\lambda,\mu)$ for the objective function in \eqref{eq:RMEP-approx:ell=1:theta1:1},
%given $(\bx,\by)$, is given by
%$$
%\lambda=v_2/v_1,\quad
%\mu=v_3/v_1,
%$$
%provided $v_1\ne 0$. What if $v_1=0$ for the optimal $\bv$? In which case, we will have
%$$
%(v_2B_1+v_3 C_1)\bx=0,\quad
%(v_2B_2+v_3 C_2)\by=0,
%$$
%meaning $B_1\bx$ and $C_1\bx$ are linearly dependent and so are $B_2\by$ and $C_2\by$, which is theoretically possible by unlikely,
%and at the same time $\theta_1=0$.
\end{enumerate}
\end{remark}

\begin{remark}\label{rk:alg-RMEP2RGEP:ell=1}
Specialized  to the case $k=1$ which is for RGEP, the alternating scheme we have described 
so far is mathematically equivalent to
the one in \cite{boeg:2005}, but much elegantly presented and concise.
%\iffalse
In fact, corresponding to \Cref{thm:RMEP-approx:ell=1a},
we will have for \eqref{eq:RGEP:eta} with $\ell=1$:
\begin{subequations}\label{eq:RGEP-approx:ell=1:eta1:2}
\begin{align}
&\eta_1=\min_{((\alpha,\gamma),\,\bx)}
          \|(\gamma A-\alpha B)\bx\|_2^2
              \label{eq:RGEP-approx:ell=1:eta1:2a}\\
&\mbox{\rm s.t.}\,\,
 \|\bx\|_2=1,\,|\gamma|^2+|\alpha|^2=1,\,\gamma\ge 0,
        \label{eq:RGEP-approx:ell=1:eta1:2b}
\end{align}
\end{subequations}
where the sought eigenvalue $\lambda$ is represented by pair $(\alpha,\gamma)$ with $\gamma\ge 0$ and understood formally as
$\lambda=\alpha/\gamma$. The alternating scheme to solve \eqref{eq:RGEP-approx:ell=1:eta1:2} goes as follows.
Given $(\alpha,\gamma)$, update $\bx$ to the unit right singular vector of $\gamma A-\alpha B$ associated with its smallest
singular value. On the other hand, given $\bx$, we update $[\gamma,\alpha]^{\T}$ to the unit eigenvector of
$$
[A\bx,-B\bx]^{\HH}[A\bx,-B\bx]\in\bbC^{2\times 2}
$$
associated with its smaller eigenvalue.
%\fi
\end{remark}

%%%%%%%%%%%%%%%%%%%%%%%%%%%%%
\section{Seek a complete set of approximate eigen-tuples}\label{sec:RMEP:ell=N}
In this section, we
explain how to compute a complete set of $\ell =N=n_1 n_2\cdots n_k$ approximate eigen-tuples via a minimal perturbation formulation
that differs slightly from the earlier \eqref{eq:RMEP-approx:ell} with $\ell=N$. Specifically,
%, deferring the general case ($1 < \ell < N$) for future work.
%
we propose to determine the approximate eigen-tuples by
\begin{subequations}\label{eq:RMEP:theta:ell=N}
\begin{align}
\wtd\theta_{N}
   :=& \inf_{  \{(\what A_i, \what B_{i1},\ldots, \what B_{ik})\}_{i=1}^k,
           \{(\lambda_{1,j},\ldots,\lambda_{k,j},\bx_{1,j},\ldots,\bx_{k,j})\}_{j=1}^{N}}
         \sum_{i=1}^k\big\| \big[\what A_i-A_i, \what B_{i1}-B_{i1},\ldots,\what B_{ik}-B_{ik}\big] \big\|_{\F}^2
           \label{eq:RMEP:theta:ell=N-1}\\
\text{s.t.}
		& \left\{
          \begin{array}{l}
		  \widehat{A}_i \bx_{i,j} = \sum_{s=1}^k\lambda_{s,j} \widehat{B}_{is} \bx_{i,j}\quad
              \mbox{for $1\le i\le k,\, 1\le j\le N$}, \\
		  \mbox{$\bx_{1,j}\otimes\bx_{2,j}\otimes\cdots\otimes \bx_{k,j}$ for $j=1,\ldots N$ are linearly independent}, \\
          \mbox{and $\rank\big(\big[\what A_i, \what B_{i1},\ldots, \what B_{ik}\big]\big)\leq n_i$ for $1\le i\le k$}.
          \end{array}\right.
           \label{eq:RMEP:theta:ell=N-2}
\end{align}
\end{subequations}
Its difference from \eqref{eq:RMEP-approx:ell} with $\ell=N$ is the  new rank conditions
$$
\mbox{$\rank\big(\big[\what A_i, \what B_{i1},\ldots, \what B_{ik}\big]\big)\leq n_i$ for $1\le i\le k$}.
$$
Such conditions are motivated by a result in \cite{itmu:2016} for the case of
RGEP: \eqref{eq:RGEP:eta} for $\ell=n$ (a complete set of $n$ approximate eigen-tuples) is equivalent to
\eqref{eq:RGEP:eta:ell=n} whose constraint $\what A=\what BZ$ implies $\rank\big(\big[\what A,\what B\big]\big)\le n$.
%
%Indeed,
%for $\ell=n$ in RGEP in \eqref{eq:RGEP:eta}, \cite{itmu:2016} proves that it can be solved from \eqref{eq:RGEP:eta:ell=n}, and the perturbed matrices $\what A=A+R,~\what B=B+E$ satisfy $\what A=\what BZ$, leading to $\rank\big(\big[\what A,\what B\big]\big)\le n$. Furthermore, we shall see in section \ref{ssec:thetaequal} that such rank conditions in \eqref{eq:RMEP:theta:ell=N} can also be interpolated by a truncated SVD, and ensure a close relation (Theorem \ref{thm:foundation}) with the unperturbed system \eqref{eq:RMEP}.
%
This very result also motivates us to introduce yet another minimal perturbation formulation:
%a closely related perturbed formulation for \eqref{eq:RMEP:theta:ell=N}:
\begin{subequations}\label{eq:RMEP:phi:ell=N}
\begin{align}
\varphi=&\inf_{\{(\what A_i,\what B_{i1},\ldots,\what B_{ik},X_1, \ldots X_k)\}_{i=1}^k} \quad
  \sum_{i=1}^k \big\| \big[\what A_i-A_i, \what B_{i1}-B_{i1},\ldots,\what B_{ik}-B_{ik}\big] \big\|_{\F}^2
           \label{eq:RMEP:phi:ell=N-1}\\
\text{s.t.}
		& \left\{
          \begin{array}{l}
          \what A_i= \what B_{i1}X_{i1}+\cdots+\what B_{ik}X_{ik}, ~ X_{ij}\in \bbC^{n_i\times n_i}\quad\mbox{for $1\le i,j\le k$}, \\
          \rank\big(\big[\what A_i, \what B_{i1},\ldots,\what B_{ik}\big]\big)\leq n_i\quad\mbox{for $1\le i\le k$}.
          \end{array}\right.
           \label{eq:RMEP:phi:ell=N-2}
\end{align}
\end{subequations}
This  formulation resembles \eqref{eq:RGEP:eta:ell=n} for RGEP but naturally replaces $\what A=\what BZ$ in \eqref{eq:RGEP:eta:ell=n} with
$\what A_i= \what B_{i1}X_{i1}+\cdots+\what B_{ik}X_{ik}$.
However,
we do not always have equivalency between \eqref{eq:RMEP:phi:ell=N} and \eqref{eq:RMEP:theta:ell=N}, unlike between
\eqref{eq:RGEP:eta} for $\ell=n$ and \eqref{eq:RGEP:eta:ell=n} in the case of RGEP.
This indicates that there are some fundamental differences in RMEP \eqref{eq:RMEP} between $k=1$ (corresponding to RGEP)
and $k>1$, and hence RMEP in general can be  more complicated than RGEP.

Evidently, problem \eqref{eq:RMEP:phi:ell=N} can be decoupled into
$\varphi=\sum_{i=1}^k\varphi_i$  with
\begin{subequations} \label{eq:RMEP-psi1}
 \begin{align}
\varphi_i:=&\inf_{\what A_i,\what B_{i1},\ldots,\what B_{ik},X_1, \ldots X_k} \quad
            \big\| \big[\what A_i-A_i, \what B_{i1}-B_{i1},\ldots,\what B_{ik}-B_{ik}\big] \big\|_{\F}^2 \label{eq:RMEP-psi1-1}\\
\text{s.t.}
		&\quad\what A_i= \what B_{i1}X_{i1}+\cdots+\what B_{ik}X_{ik}, ~
                   \rank\big(\big[\what A_i, \what B_{i1},\ldots,\what B_{ik}\big]\big)\leq n_i,
                    \label{eq:RMEP-psi1-2}
\end{align}
\end{subequations}
for $1\le i\le k$.
Mathematically, each of the problems for $\varphi_i$  is the same one and so we only need to study one of them.

The importance of \eqref{eq:RMEP:phi:ell=N} lies in that it will eventually lead to a numerical method to
compute a complete set of $N$ approximate eigen-tuples for RMEP \eqref{eq:RMEP} via the truncated SVD, as
\eqref{eq:RGEP:eta:ell=n} does for RGEP.

\subsection{$\varphi_i$ of \eqref{eq:RMEP-psi1}}
We already mentioned that each $\varphi_i$ is mathematically the same. For that reason,
abstractly we
will consider
\begin{subequations} \label{eq:RMEP-approx:phi}
\begin{align}
\phi:=&\inf_{\what A,\what B_1,\ldots,\what B_k,X_1, \ldots X_k} \quad
         \big\| \big[\what{A}-A, \what B_1-B_1,\ldots,\what B_k-B_k\big] \big\|_{\F}^2 \label{eq:RMEP-approx:phi-1} \\
\text{s.t.}
		&\quad\what{A}= \what B_1X_1+\cdots+\what B_kX_k, ~
                  \rank\big(\big[\what{A}, \what B_1,\ldots,\what B_k\big]\big)\leq n,
              \label{eq:RMEP-approx:phi-2}
\end{align}
\end{subequations}
where $A,B_s,\what A,\what B_s\in \bbC^{m\times n}$ and $X_s\in \bbC^{n\times n}$ for $1\le s\le k$, and $m\ge n$.

\begin{lemma}\label{lm:RMEP-approx:phi}
%For \eqref{eq:RMEP-approx:phi},
Let  the SVD of $\big[A,  B_1,\ldots, B_k\big]\in\bbC^{m\times (k+1)n}$ be
\begin{equation}\label{eq:RMEP-approx:SVD}
\big[A,  B_1,\ldots, B_k\big]=U\Sigma  V^{\HH},
\end{equation}
where $\wtd m:=\min\{m,(k+1)n\}$ and
$$
U=\kbordermatrix{ &\sss n &\sss \wtd m-n \\
                    & U_1 & U_2 }, \quad
\Sigma=\kbordermatrix{ &\sss n &\sss kn \\
                    \sss n & \Sigma_1 & 0 \\
                    \sss\wtd m-n & 0 &\Sigma_2}, \quad
V=\kbordermatrix{ &\sss n &\sss kn \\
                    \sss n & V_{11}  & V_{12} \\
                    \sss n & V_{21} & V_{22}\\
                          & \vdots & \vdots \\
                    \sss n & V_{k+1\,1} & V_{k+1\,2}},
$$
and $\Sigma_1\in\bbR^{n\times n}$ is diagonal, containing the $n$ largest singular values of $\big[A,  B_1,\ldots, B_k\big]$.
%and its leading diagonal entries, i.e., those at positions $(i,i)$ for $1\le i\le\wtd m$
%are the singular values  of $[A, B, C]$,
%denoted by $\sigma_i\big(\big[A, B, C\big]\big)$ in the descending order.
Then the infimum $\phi$ in \eqref{eq:RMEP-approx:phi} is given by
\begin{equation}\label{eq:RMEP-approx:phi-value}
\phi=\sum_{j=n+1}^{\wtd m}\sigma_j^2\big(\big[A,  B_1,\ldots, B_k\big]\big).
\end{equation}
Moreover, if $\|V_{11}\|_2<1$, then the infimum $\phi$ is attained by
% with the corresponding $\what A,\what B, \what C$ given by
\begin{equation}\label{eq:RMEP-approx:phi-attained}
\what{A}=U_1 \Sigma_1V_{11}^{\HH},\quad
\what B_s=U_1 \Sigma_1V_{s+1\,1}^{\HH}\quad\mbox{for $1\le s\le k$}.
%\what{C}=U_1 \Sigma_1V_{31}^{\HH}.
% \big[\what{A},\what{B},\what{C}\big]=U_1 \Sigma_1
%          \Big[V_{11}^{\HH},V_{21}^{\HH}, V_{31}^{\HH}\Big].
\end{equation}
\end{lemma}

\begin{proof}
First, we note that for any $\big[\what{A}, \what B_1,\ldots,\what B_k\big]\in\bbC^{m\times (k+1)n}$,
if $\rank\big(\big[\what{A}, \what B_1,\ldots,\what B_k\big]\big)\le n$, then
\begin{equation}\label{eq:RMEP-approx:phi-LBD}
\left\|\big[\what{A}, \what B_1,\ldots,\what B_k\big]-\big[A,  B_1,\ldots, B_k\big]\right\|_{\F}^2
   \ge\sum_{j=n+1}^{\wtd m}\sigma_j^2\big(\big[A,  B_1,\ldots, B_k\big]\big),
\end{equation}
regardless of whether $\what{A}= \what B_1X_1+\cdots+\what B_kX_k$ or not, by
the Eckart–Young–Mirsky Theorem \cite{ecka:1936,mirs:1961}. Hence, always
$$
\phi\ge\sum_{j=n+1}^{\wtd m}\sigma_j^2\big(\big[A,  B_1,\ldots, B_k\big]\big).
$$
For $\what{A}$ and $\what B_s$ given by \eqref{eq:RMEP-approx:phi-attained},
evidently,
\begin{equation}\label{eq:RMEP-approx:phi:attainedby}
\left\|\big[\what{A}, \what B_1,\ldots,\what B_k\big]-\big[A,  B_1,\ldots, B_k\big]\right\|_{\F}^2
   =\sum_{j=n+1}^{\wtd m}\sigma_j^2\big(\big[A, B, C\big]\big),
\end{equation}
because $\big[\what{A}, \what B_1,\ldots,\what B_k\big]=U_1 \Sigma_1\Big[V_{11}^{\HH},V_{21}^{\HH},\ldots, V_{k+1\,1}^{\HH}\Big]$,
a truncation of SVD \eqref{eq:RMEP-approx:SVD}.
%$\big[A,  B_1,\ldots, B_k\big]=U\Sigma  V^{\HH}$.

We know that $\|V_{11}\|_2\le 1$ always.
Suppose for the moment that
\begin{equation}\label{eq:RMEP-approx:phi:V11<1}
\|V_{11}\|_2<1.
\end{equation}
Since $[V_{11}, V_{12}][V_{11}, V_{12}]^{\HH}=I_{n}$, we know $\sigma_{\min}(V_{12})=\sqrt{1-\|V_{11}\|_2^2}>0$ and thus $\rank(V_{12})=n$.
Hence $V_{12}$ has an $n\times n$ nonsingular submatrix,
$(V_{12})_{(:,\bbJ)}$, where $\bbJ$ is a subset of $\{1,2,\ldots,kn\}$ with $n$ elements. ${\bbJ}$ can be obtained via the rank-revealing QR decomposition \cite[section~5.4.2]{govl:2013}: $V_{12}P=QR$, and then $\bbJ$ corresponds to the first $n$ columns of $V_{12}$ selected by the permutation matrix $P$.
Set $\what{A}$ and $\what B_s$ as  in \eqref{eq:RMEP-approx:phi-attained} and let
\begin{align}\label{eq:RMEP-approx:phi:XY}
\begin{bmatrix}
  X_1 \\
  \vdots \\
  X_k
\end{bmatrix}
   =-\begin{bmatrix}
 (V_{22})_{(:,\bbJ)}\\
 \vdots \\
                 (V_{k+1\,2})_{(:,\bbJ)}
               \end{bmatrix}[(V_{12})_{(:,\bbJ)}]^{-1}.
 \end{align}
Since $V$ is unitary, we have
\begin{equation}\label{eq:RMEP-approx:phi:orth}
0=\begin{bmatrix}
V_{11} \\
V_{21}\\
 \vdots \\
V_{k+1\,1}
\end{bmatrix}^{\HH}\begin{bmatrix}
                 V_{12} \\
                 V_{22}\\
 \vdots \\
                 V_{k+1\,2}
               \end{bmatrix}_{(:,\bbJ)}
=\begin{bmatrix}
V_{11} \\
V_{21}\\
 \vdots \\
V_{k+1\,1}
\end{bmatrix}^{\HH}\begin{bmatrix}
                 (V_{12})_{(:,\bbJ)} \\(V_{22})_{(:,\bbJ)}\\
 \vdots \\
                 (V_{k+1\,2})_{(:,\bbJ)}
               \end{bmatrix}.
%\quad\Rightarrow\quad
%(V_{11})^{\HH}(V_{12})_{(:,\bbI)}+(V_{21})^{\HH}(V_{22})_{(:,\bbI)}=0,
\end{equation}
On the other hand, $(V_{12})_{(:,\bbJ)}$ is nonsingular as we assumed moments ago. Pre-multiply and post-multiply
\eqref{eq:RMEP-approx:phi:orth} by $U_1\Sigma_1$ and $[(V_{12})_{(:,\bbJ)}]^{-1}$, respectively,
to get  $\what{A}= \what B_1X_1+\cdots+\what B_kX_k$. This, together with \eqref{eq:RMEP-approx:phi-LBD}, \eqref{eq:RMEP-approx:phi:XY} and \eqref{eq:RMEP-approx:phi:attainedby},
imply that $\big(\what A,\what B_1,\ldots,\what B_k,X_1, \ldots, X_k\big)$ is the optimizer
of \eqref{eq:RMEP:phi:ell=N} and achieves  $\phi=\sum_{j=n+1}^{\wtd m}\sigma_j^2\big(\big[A,  B_1,\ldots, B_k\big]\big)$.

%\Blue{
On the other hand, if $\|V_{11}\|_2=1$, then $\rank(V_{12})<n$ and it has no $n\times n$ nonsingular submatrix.
Nevertheless, we can construct  $\big(\what A(\epsilon),\what B_1(\epsilon),\ldots,\what B_k(\epsilon),X_1(\epsilon), \ldots, X_k(\epsilon)\big)$
so that
\begin{align*}
&\big\|\big[\what A(\epsilon),\what B_1(\epsilon), \dots, \what B_k(\epsilon)\big]-\big[A, B_1,\dots, B_k\big]\big\|_{\F}^2\rightarrow \phi~ {\rm  as}~ \epsilon\rightarrow 0,\\
&\what A(\epsilon) =\what B_1(\epsilon)\, \what X_1(\epsilon) +\dots+ \what B_k(\epsilon)\,\what X_k(\epsilon).
\end{align*}
%To meet the constraint $\what A_j= \what B_jX_j+\what{C}_jY_j$ in \eqref{eq:RMEP-psi1}, if \eqref{eq:RMEP-approx:phi:bothV11<1} holds, then we can construct $[X_j,Y_j]$ specified in \eqref{eq:RMEP-approx:phi:XY} directly;
To this end,  denote $\wtd V_{i2}= (V_{i2})_{(:,\bbJ)}$ for $1\le i\le k+1$, where $\bbJ$ is any subset of $\{1,2,\ldots,kn\}$ with $n$ elements.  By \eqref{eq:RMEP-approx:phi:orth},
%\marginpar{\tiny \Red{please update this paragraph}}
\begin{equation*}
0=
 V_{11}^{\HH} \wtd V_{12}  +V_{21}^{\HH} \wtd V_{22}+\dots+ V_{k+1\,1}^{\HH} \wtd V_{k+1\,2}  =:F_1+F_2+\dots+F_{k+1},
\end{equation*}
where $F_s=V_{s1}^{\HH} \wtd V_{s2}$ for $1\le s\le k+1$.
Let $\epsilon\ne  0$  be sufficiently small so that
\begin{align*}
F_1(\epsilon):=F_1+\epsilon I_{n}, ~F_2(\epsilon):= F_2-\epsilon I_{n},~V_{11}(\epsilon)=V_{11}+\epsilon I_{n},~ V_{21}(\epsilon):= V_{21}+\epsilon I_{n}
\end{align*}
are all nonsingular, and let $\wtd V_{12}(\epsilon)=[V_{11}(\epsilon)]^{-\HH}F_1(\epsilon)$ which is nonsingular because both $V_{11}(\epsilon)$ and $F_1(\epsilon)$ are. Also, let $\wtd V_{22}(\epsilon)=[V_{21}(\epsilon)]^{-\HH}F_2(\epsilon)$.
We have
% which then leads to
\begin{align}
(V_{11}(\epsilon))^{\HH} \wtd V_{12}(\epsilon)
    &+(V_{21}(\epsilon))^{\HH} \wtd V_{22}(\epsilon) +V_{31}^{\HH} \wtd V_{32} \dots+V_{k+1\,1}^{\HH} \wtd V_{k+1\,2}
          \nonumber\\
=&V_{11}^{\HH} \wtd V_{12} +V_{21}^{\HH} \wtd V_{22}  +\dots+V_{k+1\,1}^{\HH} \wtd V_{k+1\,2}
 =0.  \label{eq:ABX}
\end{align}
Now choose
\begin{align*}
&\big[\what{A}(\epsilon),\what B_1(\epsilon),\what B_2(\epsilon),\dots,\what B_k(\epsilon)\big]
     =U_1\Sigma_1\big[ (V_{11}(\epsilon))^{\HH},(V_{21}(\epsilon))^{\HH},V_{31}^{\HH},\dots,V_{k+1\,1}^{\HH}\big], \\
&[X_1(\epsilon) ,X_2(\epsilon),\dots, X_k(\epsilon)]
   =-\big[\wtd V_{22}(\epsilon)  (\wtd V_{12}(\epsilon))^{-1},   \wtd V_{32} (\wtd V_{12}(\epsilon))^{-1},\dots, \wtd V_{k+1\,2} (\wtd V_{12}(\epsilon))^{-1}\big].
\end{align*}
It can be verified that  \eqref{eq:ABX} implies
$\what{A}(\epsilon)= \what B_1(\epsilon)\, X_1(\epsilon)+\dots+\what B_k(\epsilon)\, X_k(\epsilon)$.
Hence $$\big(\what A(\epsilon),\what B_1(\epsilon),\ldots,\what B_k(\epsilon),X_1(\epsilon), \ldots, X_k(\epsilon)\big)$$  is feasible for \eqref{eq:RMEP-approx:phi}. On the other hand, with $\big[\what A,\what B_1,\dots, \what B_k\big]$ in \eqref{eq:RMEP-approx:phi-attained},  since
$$
\big\|\big[\what{A}-\what{A}(\epsilon),\what B_1-\what B_1(\epsilon),\dots, \what B_k-\what B_k(\epsilon)\big]\big\|_{\F}\rightarrow 0,~~\mbox{as}~\epsilon \rightarrow0,
$$
we   conclude that the infimum $\phi$ in \eqref{eq:RMEP-approx:phi} is given by \eqref{eq:RMEP-approx:phi-value}.
% }
\end{proof}

\subsection{$\varphi$ is a lower bound of $\wtd\theta_N$}
This subsection establishes $\wtd\theta_N\ge \varphi$. % First, we have

\begin{lemma}\label{the:rank1}
For  $\bx_{i,j} \in \bbC^{n_i}$ for $1\le i\le k$ and $j=1,\dots, N:=n_1n_2\cdots n_k$,
if
$$
\bx_{1,j}\otimes\bx_{2,j}\otimes\cdots\otimes\bx_{k,j}\quad\mbox{for $j=1,\ldots, N$}
$$
form a basis of
$\bbC^{n_1}\otimes\bbC^{n_2}\otimes\cdots\otimes\bbC^{n_k}$, then
$$
\rank([\bx_{i,1}, \bx_{i,2},\dots, \bx_{i,N}]) = n_i\quad\mbox{for $1\le i\le k$}.
$$
\end{lemma}

\begin{proof}
We prove the results by contradiction. Suppose, to the contrary, that
\begin{align*}
\rank([\bx_{i,1}, \dots, \bx_{i,N}])& = \rank([\bx_{i,i_1},\dots,\bx_{i, i_{j_i}}])=j_i\quad\mbox{for $i=1,\ldots, k$},
\end{align*}
where $j_i\le n_i$ and  $\sum_{i=1}^kj_i< \sum_{i=1}^k n_i$. We have
$\bx_{i,j}\in {\rm span}\{\bx_{i,i_1},\dots,\bx_{i, i_{j_i}}\}$ for $1\le i\le k,~ 1\le j\le N$.
Denote by $\bbI_i=\{i_1,\dots,i_{j_i}\}$.
This implies that
$$
\bx_{1,j}\otimes\bx_{2,j}\otimes\cdots\otimes\bx_{k,j} \in \text{span}\{
     \bx_{1,\wtd j_1}\otimes\bx_{2,\wtd j_2}\otimes\cdots\otimes\bx_{k,\wtd j_k}\,:\,  \wtd j_i \in \bbI_i,~1\le i\le k\}.
$$
But
$$
\dim\left(\text{span}\left\{\bx_{1,\wtd j_1}\otimes\bx_{2,\wtd j_2}\otimes\cdots\otimes\bx_{k,\wtd j_k}
    \,:\, \wtd j_i \in \bbI_i,~1\le i\le k\right\}\right)\le \prod_{i=1}^kj_i < N,
$$
contradicting that $\bx_{1,j}\otimes\bx_{2,j}\otimes\cdots\otimes\bx_{k,j}$ for $j=1,\ldots, N$ form  a basis of $\bbC^{n_1} \otimes \dots\otimes \bbC^{n_k}$.
\end{proof}

\begin{theorem}\label{thm:thetaN>=varphi}
For $\wtd\theta_N$ and $\varphi$ given by \eqref{eq:RMEP:theta:ell=N} and \eqref{eq:RMEP:phi:ell=N}, respectively, we have  $\wtd\theta_N\geq\varphi$.
\end{theorem}

\begin{proof}
For any $\epsilon>0$, there is a feasible collection of variables:
$$
\big\{\big(\what A_i(\epsilon), \what B_{i1}(\epsilon), \what B_{ik}(\epsilon)\big)\big\}_{i=1}^k,~
\big\{\big(\lambda_{1,j}(\epsilon),\dots,\lambda_{k,j}(\epsilon),\bx_{1,j}(\epsilon),\dots,\bx_{k,j}(\epsilon)\big)\big\}_{j=1}^N
$$
to \eqref{eq:RMEP:theta:ell=N} so that the associated objective function value
%\marginpar{\tiny \Red{please update this proof}}
$$
\sum_{i=1}^k
  \big\| \big[\widehat{A}_i (\epsilon)- A_i, \widehat{B}_{i1}(\epsilon)- B_{i1},\dots, \widehat{B}_{ik}(\epsilon)- B_{ik}\big] \big\|_{\F}^2
\le \wtd\theta_N+\epsilon.
$$
By Lemma \ref{the:rank1}, for each $1\le i\le k$, we can choose $n_i$ vectors $\bx_{i,{i_1}}(\epsilon),\dots,\bx_{{i,{i_{n_i}}}}(\epsilon)$ from $\{\bx_{i,j}(\epsilon)\}_{j=1}^N$
such that $Z_i(\epsilon)=[\bx_{i,{i_1}}(\epsilon),\dots,\bx_{{i,{i_{n_i}}}}(\epsilon)]$ is nonsingular.
Let
\begin{equation*}
\Lambda_{i,j}(\epsilon)=\diag(\lambda_{{j},i_1}(\epsilon),\dots,\lambda_{j,{i_{n_i}}}(\epsilon)), ~~
X_{ij}=Z_i(\epsilon)\,\Lambda_{i,j}(\epsilon)\,(Z_i(\epsilon))^{-1}, ~~1\le i,j\le k.
%\quad
%     \Gamma(\epsilon)&=\diag(\mu_{j_1}(\epsilon),\dots,\mu_{j_{n_2}}(\epsilon)), \\
%X_1&=X(\epsilon)\Lambda(\epsilon)(X(\epsilon))^{-1},&~X_2&=Y(\epsilon)\Lambda(\epsilon)(Y(\epsilon))^{-1},\\
%Y_1&=X(\epsilon)\Gamma(\epsilon)(X(\epsilon))^{-1}, &~Y_2&=Y(\epsilon)\Gamma(\epsilon)(Y(\epsilon))^{-1}.
\end{equation*}
It can be seen that
$\big\{\big(\what A_i(\epsilon), \what B_{i1}(\epsilon), \dots \what B_{ik}(\epsilon),  X_{i1},\dots, X_{ik}\big)\big\}_{i=1}^k$
satisfy \eqref{eq:RMEP:phi:ell=N-2}, which proves that $\varphi\le \wtd\theta_N+\epsilon$. Letting $\epsilon\rightarrow 0$ gives $\varphi\le \wtd\theta_N$.
\end{proof}

 %%%%%%%%%%%%%%%%%%%%
\subsection{When does $\wtd\theta_N= \varphi$?}\label{ssec:thetaequal}
We just showed that $\wtd\theta_N\ge \varphi$, but ideally we would like to have $\wtd\theta_N= \varphi$.
Unfortunately, this cannot be guaranteed in general.
In this subsection, we will derive sufficient conditions under which $\wtd\theta_N = \varphi$, among others.

Lemma~\ref{lm:RMEP-approx:phi} applied to each $\varphi_i$ of \eqref{eq:RMEP-psi1} leads to the following theorem.

\begin{theorem}\label{thm:RMEP-approx:ell=N}
Let  the SVDs of $[A_i, B_{i1},\ldots,B_{ik}]\in\bbC^{m_i\times (k+1)n_i}~(1\le i\le k)$ be
\begin{subequations}\label{eq:AiBis:SVD}
\begin{equation}\label{eq:AiBis:SVD-1}
[A_i, B_{i1},\ldots,B_{ik}]=U^{i}\Sigma^i  (V^i)^{\HH}
\end{equation}
where $\wtd m_i:=\min\{m_i,(k+1)n_i\}$,
\begin{equation}\label{eq:AiBis:SVD-2}
U^i=\kbordermatrix{ &\sss n_i &\sss \wtd m_i-n_i \\
                    & U^i_1 & U^i_2 }, \,
\Sigma^i=\kbordermatrix{ &\sss n_i &\sss kn_i \\
                    \sss n_i & \Sigma^i_1 & 0 \\
                    \sss\wtd m_i-n_i & 0 &\Sigma^i_2}, \,
V^i=\kbordermatrix{ &\sss n_i &\sss kn_i \\
                    \sss n_i & V^i_{11}  & V^i_{12} \\
                    \sss n_i & V^i_{21} & V^i_{22}\\
                             &\vdots & \vdots \\
                    \sss n_i & V^i_{k+1\,1} & V^i_{k+1\,2}},
\end{equation}
\end{subequations}
and $\Sigma_1^i\in\bbR^{n_i\times n_i}$ is diagonal, containing the $n_i$ largest singular values of $[A_i, B_{i1},\ldots,B_{ik}]$.
Then  the infimum $\varphi$ of \eqref{eq:RMEP:phi:ell=N} satisfies
$$
\varphi = \sum_{i=1}^k\varphi_i
        = \sum_{i=1}^k\sum_{j=n_i+1}^{\wtd m_i} \sigma_j^{2}\big([A_i, B_{i1},\ldots,B_{ik}]\big).
$$
If
\begin{equation}\label{eq:RMEP-approx:phi:bothV11<1}
		\|V_{11}^{i}\|_2<1  ~\mbox{for}~i=1,\dots,k,
\end{equation}
then $\varphi$ is attained by, for $1\le i\le k$,
\begin{equation}\label{eq:RMEP-TSVD:pert-ABC}
\what A_i=U_1^i\Sigma_1^i (V_{11}^i)^{\HH}, \quad
\what B_{is}=U_1^i\Sigma_1^i(V_{s+1\,1}^i)^{\HH}\quad
 \quad\mbox{for $1\le s\le k$}.
\end{equation}
\end{theorem}

 An important implication of  Theorem \ref{thm:RMEP-approx:ell=N} is that
 a complete set of $N$ approximate eigen-tuples for RMEP~\eqref{eq:RMEP}  can be calculated from MEP:
\begin{equation}\label{eq:RMEP-TSVD}
	\left\lbrace \begin{aligned}
		(V_{11}^1)^{\HH}\bx_1&=\lambda_1 (V_{21}^1)^{\HH}\bx_1+\cdots+\lambda_k (V_{k+1\,1}^1)^{\HH}\bx_1,\\
            &\vdots \\
		(V_{11}^k)^{\HH}\bx_k&=\lambda_1 (V_{21}^k)^{\HH}\bx_k+\cdots+\lambda_k (V_{k+1\,1}^k)^{\HH}\bx_k.
	\end{aligned}\right.
\end{equation}
We call the method of solving \eqref{eq:RMEP} via  \eqref{eq:RMEP-TSVD} {\em RMEP via Truncated SVD\/}
(RMEPvTSVD). According to \eqref{eq:MEP-GEP}, \eqref{eq:RMEP-TSVD} can be
transformed into GEPs in the form of \eqref{eq:MEP-GEP} for its numerical solution.
The second implication of  Theorem \ref{thm:RMEP-approx:ell=N}
is the sufficient conditions that ensure $\varphi=\wtd\theta_N$ in the next theorem.
%With \eqref{eq:RMEP-TSVD}, we can provide sufficient conditions for $\varphi=\wtd\theta_N$.

\begin{theorem}\label{thm:fullequl}
If \eqref{eq:RMEP-approx:phi:bothV11<1} holds and MEP \eqref{eq:RMEP-TSVD} has $N=n_1n_2\cdots n_k$ linearly independent eigenvectors
\begin{equation}\label{eq:fullequl:zj}
\bz_j=  \bx_{1,j}\otimes\bx_{2,j}\otimes\cdots\otimes\bx_{k,j}\quad\mbox{for $j=1,\dots,N$},
\end{equation}
then $\varphi=\wtd\theta_N$.
\end{theorem}

\begin{proof}
If \eqref{eq:RMEP-approx:phi:bothV11<1} holds, then we will have \eqref{eq:RMEP-TSVD:pert-ABC}.
If also MEP~\eqref{eq:RMEP-TSVD} has $N=n_1n_2\cdots n_k$ linearly independent eigenvectors as in
\eqref{eq:fullequl:zj},
then
the constraints in \eqref{eq:RMEP:theta:ell=N} are satisfied by
$\what A_i$ and $\what B_{is}$ in \eqref{eq:RMEP-TSVD:pert-ABC}, and
those vectors in \eqref{eq:fullequl:zj}. Hence $\wtd\theta_N\le\varphi$ which, together with
\Cref{thm:thetaN>=varphi}, yield $\varphi=\wtd\theta_N$.
\end{proof}

\subsection{Algorithm RMEPvSVD}
To make MEP~\eqref{eq:RMEP-TSVD} look similar to the original RMEP~\eqref{eq:RMEP},
we denote, for $1\le i\le k$,
$$
\wtd A_i=(V_{11}^i)^{\HH},\quad
       \wtd B_{is}=(V_{s+1\,1}^i)^{\HH}\in\bbC^{n_i\times n_i}\quad\mbox{for $1\le s\le k$},
$$
and rewrite MEP~\eqref{eq:RMEP-TSVD} as
\begin{equation}\label{eq:RMEP-approx:ell=N:MEP}
\left\{ \begin{aligned}
\wtd A_1 \bx_1 &= \lambda_1 \wtd B_{11} \bx_1 + \lambda_2 \wtd B_{12} \bx_1 + \cdots + \lambda_k \wtd B_{1k} \bx_1, \\
		&\vdots\\
\wtd A_k \bx_k &= \lambda_1 \wtd B_{k1} \bx_k + \lambda_2 \wtd B_{k2} \bx_k + \cdots + \lambda_k \wtd B_{kk} \bx_k.
	\end{aligned}\right.
\end{equation}
We will solve this MEP by GEPs as explained in Appendix \ref{sec:MEPvGEP}.
The complete procedure of solving RMEP~\eqref{eq:RMEP} via  MEP~\eqref{eq:RMEP-approx:ell=N:MEP} is outlined in \Cref{alg:RMEP-approx:ell=N},
where the normalized residual for an approximate eigen-tuple $(\lambda_1,\ldots,\lambda_k,\bx_1,\ldots,\bx_k)$
with all $\|\bx_i\|_2=1$ is defined as:
\begin{subequations}\label{eq:NRes}
\begin{align}
\rho_i(\lambda_1,\ldots,\lambda_k,\bx_i)
           &= \frac {\| A_i\bx_i-\sum_{s=1}^k\lambda_s B_{is}\bx_i\|_2}
                    {\|A_i\|_2+\sum_{s=1}^k|\lambda_s|\cdot\|B_{is}\|_2}\quad\mbox{for $1\le i\le k$}, \label{eq:NRes-ind}\\
\rho(\lambda_1,\ldots,\lambda_2,\bx_1,\ldots,\bx_k)
             &=\sum_{i=1}^k \rho_i(\lambda_1,\ldots,\lambda_k,\bx_i).
\end{align}
\end{subequations}

%, we define
%$$
%\left\lbrace
%\begin{aligned}
%	& \Delta_0 = \wtd B_1 \otimes \wtd C_2 - \wtd C_1 \otimes \wtd B_2, \\
%	& \Delta_1 = \wtd A_1 \otimes \wtd C_2 - \wtd C_1 \otimes \wtd A_2, \\
%	& \Delta_2 = \wtd B_1 \otimes \wtd A_2 - \wtd A_1 \otimes \wtd B_2.
%\end{aligned}\right.
%$$
%If  $\Delta_0$ is nonsingular, $ \Delta_0^{-1} \Delta_1$ and $\Delta_0^{-1} \Delta_2$  commutes \cite{atki:1972,atki:1968}, and,  moreover,  eigen-tuple $(\lambda, \bx, \by)$ can be computed
%by first solving GEPs
%\begin{equation}\label{eq:RMEP-TSVD:GEP}
%\Delta_1\bz  = \lambda \Delta_0 \bz,~ \Delta_2 \bz  = \mu \Delta_0 \bz,
%\end{equation}
%and decompose \Red{$\bz=\bx  \otimes \by$}. If, however, $\Delta_0$ is singular but
%MEP~\eqref{thm:RMEP-approx:ell=N:MEP} is regular,

\begin{algorithm}[t]
\caption{RMEPvTSVD: solving RMEP~\eqref{eq:RMEP} via Truncated SVD for a complete set of approximate eigen-tuples}
\label{alg:RMEP-approx:ell=N}
\begin{algorithmic}[1]
\REQUIRE $A_i, B_{is}\in \bbC^{m_i \times n_i} ~(m_i> n_i)$ for $1\le s\le k$ and $1\le i\le k$;
\ENSURE  $ N(= n_1n_2\cdots n_k)$ approximate eigen-tuples
         $\{(\lambda_{1,j},\ldots,\lambda_{k,j},\bx_{1,j},\ldots,\bx_{k,j})\}_{j=1}^{N}$.
\STATE compute the SVDs of $[A_i, B_{i1},\ldots, B_{ik}]$  as in \eqref{eq:AiBis:SVD}  for $1\le i\le k$;	
\STATE solve the standard MEP~\eqref{eq:RMEP-approx:ell=N:MEP} for its eigen-tuples
       $\{(\lambda_{1,j},\ldots,\lambda_{k,j},\bx_{1,j},\ldots,\bx_{k,j})\}_{j=1}^{N}$ by GEPs as explained in Appendix \ref{sec:MEPvGEP};
\STATE compute the normalized residuals $\rho(\lambda_{1,j},\ldots,\lambda_{k,j},\bx_{1,j},\ldots,\bx_{k,j})$ as defined
       by \eqref{eq:NRes};
\RETURN the $N$ eigen-tuples $\{(\lambda_{1,j},\ldots,\lambda_{k,j},\bx_{1,j},\ldots,\bx_{k,j})\}_{j=1}^{N}$ along with their normalized residuals $\{\rho(\lambda_{1,j},\ldots,\lambda_{k,j},\bx_{1,j},\ldots,\bx_{k,j})\}_{j=1}^{N}$.
\end{algorithmic}
\end{algorithm}

\begin{theorem}\label{thm:foundation}
If $\rank([A_i, B_{i1},\ldots,B_{ik}])\le n_i$ for $1\le i\le k$, then any solution to MEP~\eqref{eq:RMEP-approx:ell=N:MEP} is a solution to   RMEP~\eqref{eq:RMEP};
otherwise if $\rank([A_i, B_{i1},\ldots,B_{ik}])\ge n_i$ for $1\le i\le k$, then any solution to RMEP~\eqref{eq:RMEP} is a solution to MEP~\eqref{eq:RMEP-approx:ell=N:MEP}.
Particularly, if  $\rank([A_i, B_{i1},\ldots,B_{ik}])= n_i$ for $1\le i\le k$, then the set of solutions of MEP~\eqref{eq:RMEP-approx:ell=N:MEP} and RMEP~\eqref{eq:RMEP}
are the same.
\end{theorem}

\begin{proof}
With $\what A_i$ and $\what B_{is}$ given in \eqref{eq:RMEP-TSVD:pert-ABC},
the perturbed system
$$
\wtd A_i \bx_i = \lambda_1 \wtd B_{i1} \bx_i + \lambda_2 \wtd B_{i2} \bx_i + \cdots + \lambda_k \wtd B_{ik} \bx_i
\quad\mbox{for $1\le i\le k$}
%
%\what A_1 \bx = \lambda \what B_1 \bx + \mu \what{C}_1 \bx,~ \what A_2 \by = \lambda \what B_2 \by + \mu \what{C}_2 \by
$$	
is equivalent to
$$
U_1^i\Sigma_1^i\big[ (V_{11}^i)^{\HH}\bx_i-\lambda_1 (V_{21}^i)^{\HH}\bx_i-\cdots-\lambda_k (V_{k+1\,1}^i)^{\HH}\bx_i\big]=0
\quad\mbox{for $1\le i\le k$}.
$$
%\begin{align*}%\label{eq:sysII}
%U_1^1\Sigma_1^1[ (V_{11}^1)^{\HH}\bx-\lambda (V_{21}^1)^{\HH}\bx-\mu (V_{31}^1)^{\HH}\bx]&=0,\\
%U_1^2\Sigma_1^2[ (V_{11}^2)^{\HH}\by-\lambda (V_{21}^2)^{\HH}\by-\mu (V_{31}^2)^{\HH}\by]&=0.
%\end{align*}
Let $T_i=A_i-\lambda_1 B_{i1}-\cdots-\lambda_k B_{ik}$ for $1\le i\le k$.
Since $\big[\what A_i, \what B_{i1},\ldots, \what B_{ik}\big]$ is the best rank-$n_i$ approximation of
$\big[A_i, B_{i1},\ldots, B_{ik}\big]$, for any $(\lambda_1,\ldots,\lambda_k,\bx_1,\ldots,\bx_k)$, we have
\begin{align}
\left\| T_i\bx_i \right\|_2
  &= \left\| \big[A_i, B_{i1},\ldots, B_{ik}\big]
             \begin{bmatrix}
		       \hm\bx_i \\
		      -\lambda_1 \bx_i \\
               \vdots \\
		      -\lambda_k \bx_i
             \end{bmatrix} \right\|_2  \nonumber\\
	&\ge\left\| \big[\what A_i, \what B_{i1},\ldots, \what B_{ik}\big]
             \begin{bmatrix}
		       \hm\bx_i \\
		      -\lambda_1 \bx_i \\
               \vdots \\
		      -\lambda_k \bx_i
             \end{bmatrix} \right\|_2 \label{eq:tSVDineq} \\
	&= \left\|U_1^i\Sigma_1^i\big[ (V_{11}^i)^{\HH}\bx_i-\lambda_1 (V_{21}^i)^{\HH}\bx_i-\cdots-\lambda_k (V_{k+1\,1}^i)^{\HH}\bx_i\big]\big]\right\|_2.
        \label{eq:tSVDineqb}
\end{align}
Hence, if $\rank([A_i, B_{i1},\ldots,B_{ik}])\le n_i$,
then the equality in \eqref{eq:tSVDineq} holds. Thus, under
$\rank([A_i, B_{i1},\ldots,B_{ik}])\le n_i~(1\le i\le k)$,
whenever $(\lambda_1,\ldots,\lambda_k,\bx_1,\ldots,\bx_k)$ is a solution to MEP~\eqref{eq:RMEP-approx:ell=N:MEP},  we have
$T_i\bx_i=\mathbf{0}~(1\le i\le k)$, implying $(\lambda_1,\ldots,\lambda_k,\bx_1,\ldots,\bx_k)$
is a solution to RMEP~\eqref{eq:RMEP}. Conversely, if $\rank([A_i, B_{i1},\ldots,B_{ik}])\ge n_i$, then $\Sigma_1^i$ is nonsingular, and therefore, any $(\lambda_1,\ldots,\lambda_k,\bx_1,\ldots,\bx_k)$
satisfying RMEP~\eqref{eq:RMEP} is a solution to MEP~\eqref{eq:RMEP-approx:ell=N:MEP}.
In case of $\rank([A_i, B_{i1},\ldots,B_{ik}])= n_i$ $(1\le i\le k)$, the two systems have the same set of solutions.
\end{proof}

Computing $\ell$ approximate eigen-tuples when $1 < \ell < N$ via the minimal perturbation formulation
\eqref{eq:RMEP-approx:ell=1:theta1:0} (or some modification of it)
is challenging as in the case of RGEP. One option is to compute a complete set of $N$ approximate eigen-tuples
and then pick $\ell$ out of them with
the smallest normalized residuals.
% $\{\rho(\lambda_{1,j},\ldots,\lambda_{k,j},\bx_{1,j},\ldots,\bx_{k,j})\}_{j=1}^\ell$ as defined in
%\Cref{alg:RMEP-approx:ell=N}.

The advantage of RMEPvTSVD is its ability in employing existing  methods for MEP as outlined in Appendix \ref{sec:MEPvGEP}.
Unfortunately, the involved GEPs can be of prohibitively large scale  unless $n_i$ is in the 10s or smaller and $k\le 2$.

Finally, we point out that
there is no cheap means to ensure
the requirement that MEP~\eqref{eq:RMEP-approx:ell=N:MEP} possesses $N = n_1\cdots n_k$ linearly independent eigenvectors in general.
There are some sufficient conditions for this requirement in the literature. One of them is as follows
(\cite{atki:1972,kosi:1994}, \cite[Section~4.4]{volk:1988}): for Hermitian $\wtd A_i,\,\wtd B_{is}$ ($1\le i,s\le k$), if the determinant
%\marginpar{\tiny \Red{update this paragraph}}
$$
 \det\left(\begin{bmatrix}
           \bx_1^{\HH} \wtd B_{11}\bx_1& \dots& \bx_1^{\HH} \wtd B_{1k}\bx_k\\
           \vdots& &\vdots \\
           \bx_k^{\HH} \wtd B_{k1}\bx_1& \dots& \bx_k^{\HH} \wtd B_{kk}\bx_k
           \end{bmatrix}
 \right)>0,
$$
for any  $0\ne\bx_i\in \bbC^{n_i}~(1\le i\le k)$, then MEP~\eqref{eq:RMEP-approx:ell=N:MEP} admits $N$ linearly independent eigenvectors. However, verifying this condition is highly nontrivial, if at all possible.

%%%%%%%%%%%%%%%%%%%%%%%%%%%%%%%
\section{Numerical Experiments}\label{sec:egs}
In this section, we present numerical experiments to evaluate the performance of our proposed methods:
the alternating scheme (Algorithm \ref{alg:RMEP-approx:ell=1}) for computing one approximate eigen-tuple,
and  RMEPvTSVD  (Algorithm \ref{alg:RMEP-approx:ell=N}) for computing a complete set of $N$ approximate eigen-tuples.
All our experiments are, however, with $k=2$ for three reasons: 1) The case $k=2$ appears to be general enough to
demonstrate the effectiveness of our RMEPvTSVD  (Algorithm \ref{alg:RMEP-approx:ell=N}), without loss of generality, and 2) GEPs transformed from a very small scale MEP
can be of huge scale for $k=2$ already and of even much huger scale for $k>2$, and 3) there are reasonably efficient
packages for multiparameter eigenvalue problems, available online for $k=2$, such as \texttt{MultiParEig} \cite{ples:2021},
that facilitate our  experiments for RMEPvTSVD.

We will conduct our evaluations first on randomly generated problems and then on two problems
arising from discretizing differential equations: the multiparameter Sturm-Liouville equation \cite{atmi:2010}, and the Helmholtz equation \cite{amls:2014,eina:2022,ghhp:2012,volk:1988}.
All computations are performed in  MATLAB (R2024b) on a 14-inch MacBook Pro with an M3 Pro chip and 18GB memory.

To be consistent with the conventional notation for $k=2$, we will use $(\lambda,\mu)$ for an eigenvalue tuple, instead of
$(\lambda_1,\lambda_2)$ inherited from previous sections for any $k$.
For an approximate eigen-tuple $(\lambda,\mu,\bx_1,\bx_2)$, we measure its accuracy in terms of the normalized residual
$\rho(\lambda,\mu,\bx_1,\bx_2)=\rho_1(\lambda,\mu,\bx_1)+\rho_2(\lambda,\mu,\bx_2)$
  defined in \eqref{eq:NRes}.
The stopping criterion for  {\Cref{alg:RMEP-approx:ell=1}}
is when either  the number of iterations exceeds 1000, or  \eqref{eq:reltheta1} is fulfilled.

%\marginpar{\tiny Objective function in \eqref{eq:RMEP-approx:ell=1:theta1:2} is different from before.
%     Need to rerun the code.}

\subsection{Experiments on random problems}\label{ssec:evalAlg12}
We test Algorithm \ref{alg:RMEP-approx:ell=1} on randomly generated matrices $A_i,\,B_{i1},\,B_{i2}~(i=1,2)$ with different pairs $(m_i,n_i)$. Elements of the real and imaginary parts of the matrices are sampled from the standard normal distribution.
Figure~\ref{figure1} plots the objective function in \eqref{eq:RMEP-approx:ell=1:theta1:2},
the normalized residual $\rho_i=\rho_i(\lambda,\mu,\bx_i)$ for $i=1,2$ in \eqref{eq:NRes-ind} and the normalized KKT residual $\epsilon_{\rm KKT}$ in \eqref{eq:KKT}.
It is observed that the objective function value is monotonically decreasing and the normalized KKT residual $\epsilon_{\rm KKT}$
moves towards $0$, as expected.

%\marginpar{\tiny Update the numerical results due to the different objective function in \eqref{eq:RMEP-approx:ell=1:theta1:2}
%    and \Red{new} normalized residual in \eqref{eq:NRes}.}

\begin{figure}[t]
\begin{center}
%\hskip -10mm
{\includegraphics[width= 2.95in, height=3.15in]{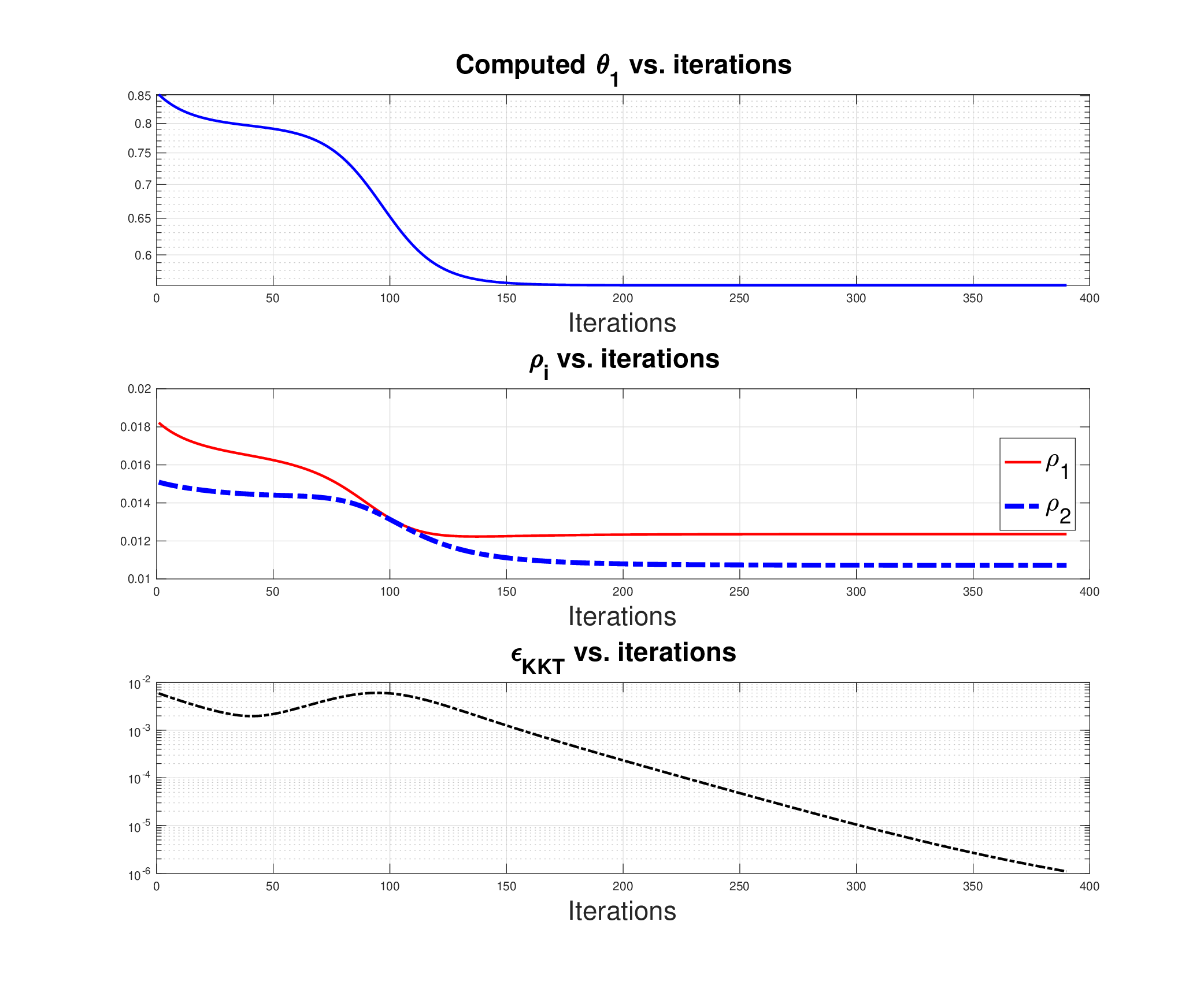}
 \includegraphics[width= 2.95in, height=3.15in]{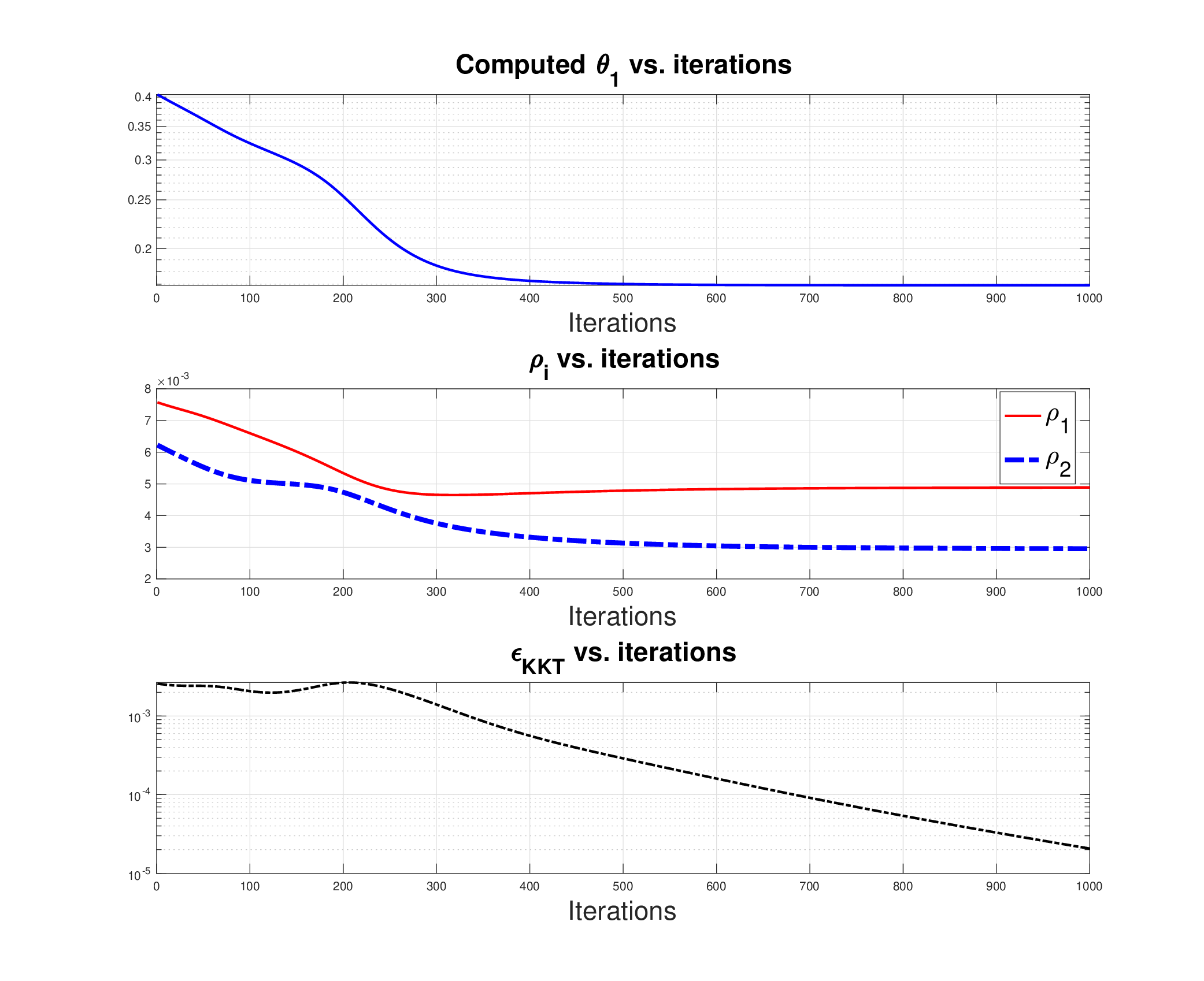}}
\caption{\small Convergence history by \Cref{alg:RMEP-approx:ell=1}.
{\em Left panel:\/} $(m_i,n_i)=(200,190)$ for $i=1,2$;
{\em Right panel:\/} $(m_1,n_1)=(500,490),~(m_2,n_2)=(600,590)$.}
\label{figure1}
\end{center}
\end{figure}

%\marginpar{\tiny \Red{new notation: $(\lambda,\mu,\bx_1,\bx_2)$ --- update figures}}

To evaluate RMEPvTSVD (\Cref{alg:RMEP-approx:ell=N}), we adopt a similar way for generating test problems in \cite{boeg:2005} for  RGEP \eqref{eq:RGEP:eta}. Also for the purpose of demonstrating the feasibility of RMEPvTSVD, we use very small $n_i$
(in fact $n_i=5$) to avoid large scale GEPs later on.
Matrices  $A_i, B_{i1},B_{i2} \in \bbC^{m_i\times n_i}$ with $m_i = 20, n_i = 5$ for $i=1,2$ are generated as follows:
\begin{enumerate}[(a)]
	\item generate random  $\check{A}_i, \check{B}_{i1},\check{B}_{i2} \in \bbC^{n_i\times n_i}$,
          and $\check{Q}_i\in \bbC^{m_i\times n_i}$ for $i=1,2$, with elements of the real and imaginary parts  sampled from the standard normal distribution;
	\item compute the thin QR factorization of $\check Q_i$ and let $Q_i\in \bbC^{m_i\times n_i}$ be the $Q$-factor for $i=1,2$;
%	\item let  $A_i = Q_i \check{A}_i$, $B_{i1}=Q_i \check{B}_{i1}$, and $B_{i2}=Q_i \check{B}_{i2};$
	\item set $ {A}_i =Q_i \check{A}_i+ E_i,~  {B}_{i1}=Q_i \check{B}_{i1}+ F_{i1},~  {B}_{i2}=Q_i \check{B}_{i2} + F_{i2}$,
          where  elements of the real and imaginary parts of $E_i$, $F_{i1}$ and $F_{i2}$  are sampled from normal distribution  $\textbf{N}(0,\sigma I)$, where parameter $\sigma$ dictates the level of noise.
\end{enumerate}

Generically, the MEP with data matrices $\{\check{A}_i, \check{B}_{i1},\check{B}_{i2}\}_{i=1}^2$ possesses $25$ eigenvalue tuples
$(\lambda_j,\mu_j)$ which, calculated by \texttt{MultiParEig} \cite{ples:2021}, will serve
as the reference eigenvalue tuples. Applying
RMEPvTSVD to  RMEP \eqref{eq:RMEP}, we compute a complete set of $25$ approximate eigenvalue tuples $(\wtd\lambda_j,\wtd\mu_j)$
and compare them with the reference ones $(\lambda_j,\mu_j)$.
The results of a typical test are illustrated in Figure~\ref{fig:noise_comparison_1} for noise levels $\sigma\in\{0,0.2,0.4,0.6\}$, respectively.
We also report in Table~\ref{tab:error_summary} various average errors from 1000 random tests:
the averages  of
\begin{subequations}\nonumber\label{eq:errs-eigvals}
\begin{gather}
\max_{j}\frac{|\lambda_j-\wtd\lambda_j|}{|\lambda_j|+|\wtd\lambda_j|},\quad
\max_{j}\frac{|\mu_j-\wtd\mu_j|}{|\lambda_j|+|\wtd\mu_j|},\quad
\min_{j}\frac{|\lambda_j-\wtd\lambda_j|}{|\lambda_j|+|\wtd\lambda_j|},\quad
\min_{j}\frac{|\mu_j-\wtd\mu_j|}{|\lambda_j|+|\wtd\mu_j|}, \\
\frac{1}{25}\sum_{j=1}^{25}\frac{|\lambda_j-\wtd\lambda_j|}{|\lambda_j|+|\wtd\lambda_j|}, \quad \frac{1}{25}\sum_{j=1}^{25}\frac{|\mu_j-\wtd\mu_j|}{|\mu_j|+|\wtd\mu_j|}
\end{gather}
\end{subequations}
over these 1000 independent runs.

\begin{figure}[H]
	\begin{center}
\small
\begin{tabular}{c|c}
%  \hline
  % after \\: \hline or \cline{col1-col2} \cline{col3-col4} ...
  $\sigma=0$ & $\sigma=0.2$ \\
 \includegraphics[width=2.8in, height=2.8in]{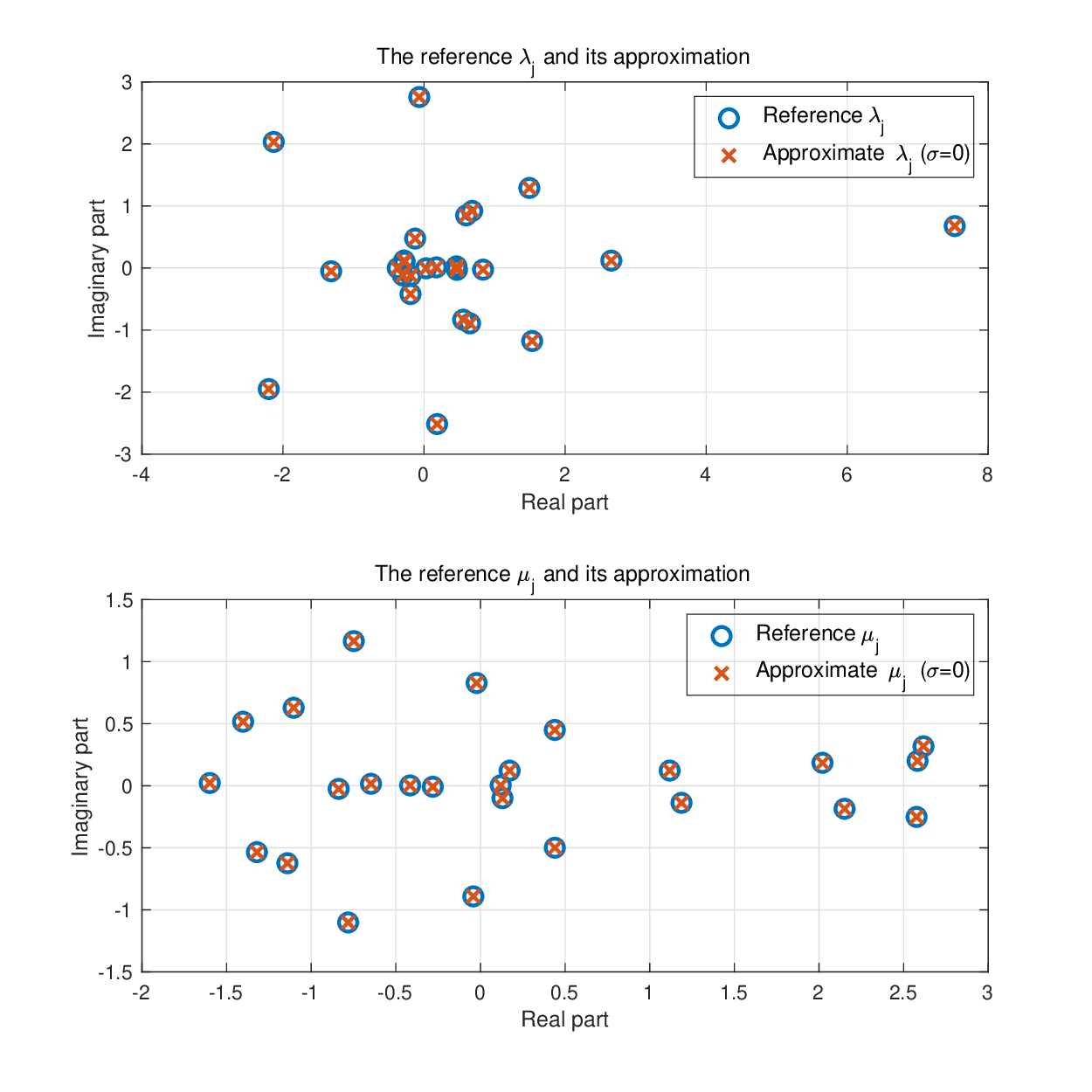} &
		\includegraphics[width=2.8in, height=2.8in]{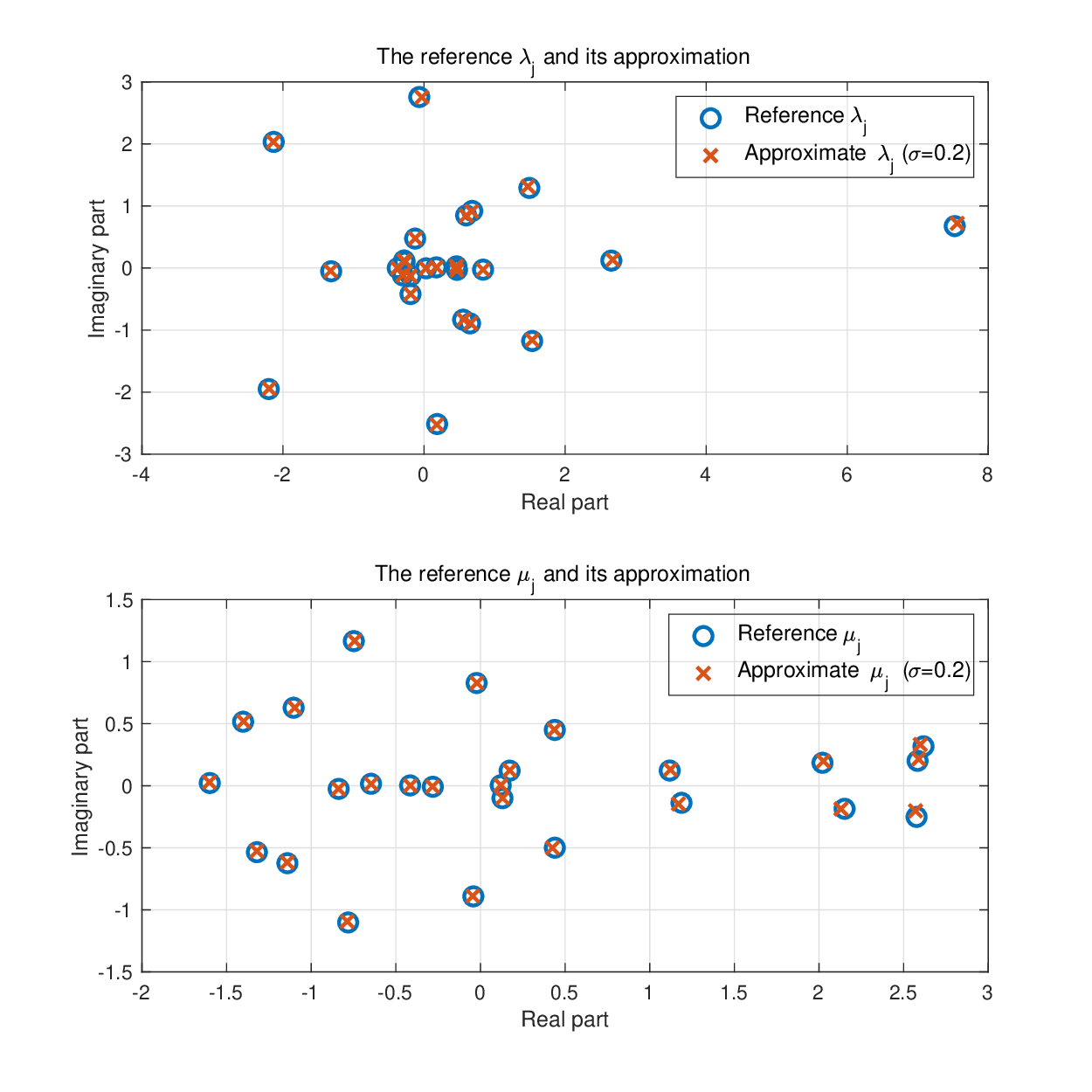} \\
  \hline\hline
  $\sigma=0.4$ & $\sigma=0.6$ \\
 \includegraphics[width=2.8in, height=2.8in]{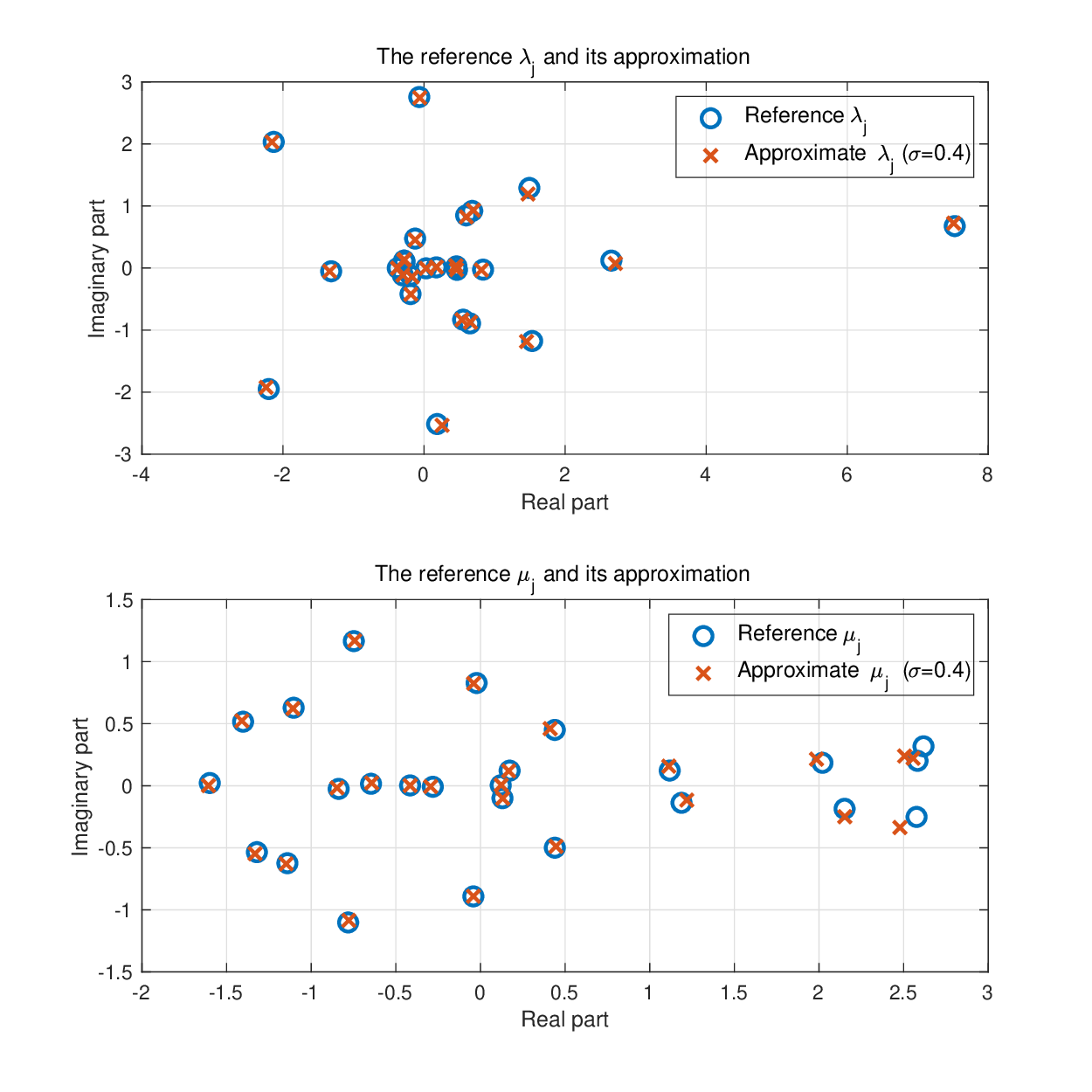} &
		\includegraphics[width=2.8in, height=2.8in]{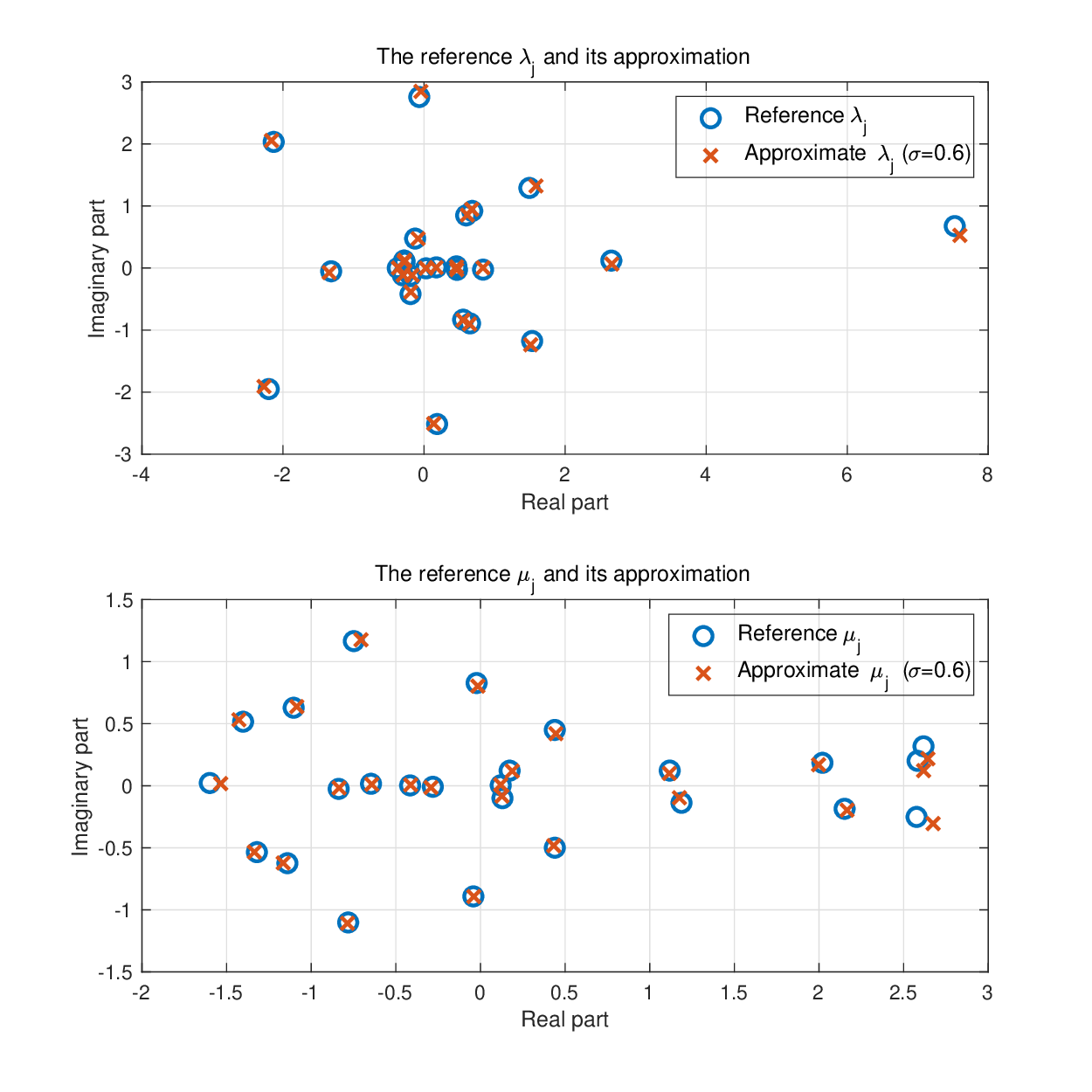}
\end{tabular}
	\end{center}
	\caption{The reference eigenvalue tuples $(\lambda_j,\mu_j)$ and the computed approximations $(\wtd\lambda_j,\wtd\mu_j)$
             by RMEPvTSVD
             as noise level $\sigma$ varies in $\{0,0.2,0.4,0.6\}$.}
	\label{fig:noise_comparison_1}
\end{figure}

\renewcommand{\arraystretch}{1.1}
\begin{table}[t]
%	\centering
	\caption{\small The average errors over 1000 random problems by RMEPvTSVD}	\label{tab:error_summary}
\begin{center}
\small
	\begin{tabular}{c||c|c|c}
		\hline
		 {noise level $\sigma$} & $\max_{j}\frac{|\lambda_j-\wtd\lambda_j|}{|\lambda_j|+|\wtd\lambda_j|}$& $\min_{j}\frac{|\lambda_j-\wtd\lambda_j|}{|\lambda_j|+|\wtd\lambda_j|}$ & \textbf{$\frac{1}{25}\sum_{j=1}^{25}\frac{|\lambda_j-\wtd\lambda_j|}{|\lambda_j|+|\wtd\lambda_j|}$ }
            \vphantom{\vrule height0.5cm width0.9pt depth0.3cm}\\
		\hline
   0                & $6.1305\cdot 10^{-15}$ & $6.9792\cdot 10^{-17}$ & $8.3825\cdot 10^{-16}$ \\
$1.0\cdot 10^{-2}$ & $3.3756\cdot 10^{-3}$ & $3.2333\cdot 10^{-5}$ & $4.1698\cdot 10^{-4}$ \\
$5.0\cdot 10^{-2}$ & $1.5615\cdot 10^{-2}$ & $1.6763\cdot 10^{-4}$ & $2.0001\cdot 10^{-3}$ \\
$1.0\cdot 10^{-1}$ & $3.3347\cdot 10^{-2}$ & $3.4663\cdot 10^{-4}$ & $4.1423\cdot 10^{-3}$ \\
$2.0\cdot 10^{-1}$ & $6.2087\cdot 10^{-2}$ & $6.6219\cdot 10^{-4}$ & $8.0285\cdot 10^{-3}$ \\
$4.0\cdot 10^{-1}$ & $1.1555\cdot 10^{-1}$ & $1.3373\cdot 10^{-3}$ & $1.5545\cdot 10^{-2}$ \\
$6.0\cdot 10^{-1}$ & $1.6234\cdot 10^{-1}$ & $1.9843\cdot 10^{-3}$ & $2.2794\cdot 10^{-2}$ \\
		\hline\hline
		 {noise level $\sigma$} &  $\max_{j}\frac{|\mu_j-\wtd\mu_j|}{|\mu_j|+|\wtd\mu_j|}$& $\min_{j}\frac{|\mu_j-\wtd\mu_j|}{|\mu_j|+|\wtd\mu_j|}$ & \textbf{$\frac{1}{25}\sum_{j=1}^{25}\frac{|\mu_j-\wtd\mu_j|}{|\mu_j|+|\wtd\mu_j|}$ }
    \vphantom{\vrule height0.5cm width0.9pt depth0.3cm} \\ \hline
	0               & $5.8208\cdot 10^{-15}$ & $6.7847\cdot 10^{-17}$ & $8.2996\cdot 10^{-16}$ \\
$1.0\cdot 10^{-2}$ & $3.0145\cdot 10^{-3}$ & $3.3620\cdot 10^{-5}$ & $4.0289\cdot 10^{-4}$ \\
$5.0\cdot 10^{-2}$ & $1.4758\cdot 10^{-2}$ & $1.6793\cdot 10^{-4}$ & $1.9726\cdot 10^{-3}$ \\
$1.0\cdot 10^{-1}$ & $3.1139\cdot 10^{-2}$ & $3.2430\cdot 10^{-4}$ & $4.0505\cdot 10^{-3}$ \\
$2.0\cdot 10^{-1}$ & $5.9885\cdot 10^{-2}$ & $6.7175\cdot 10^{-4}$ & $7.9401\cdot 10^{-3}$ \\
$4.0\cdot 10^{-1}$ & $1.1571\cdot 10^{-1}$ & $1.3120\cdot 10^{-3}$ & $1.5579\cdot 10^{-2}$ \\
$6.0\cdot 10^{-1}$ & $1.6049\cdot 10^{-1}$ & $1.9446\cdot 10^{-3}$ & $2.2646\cdot 10^{-2}$ \\
		\hline
	\end{tabular}
\end{center}
\end{table}

%%%%%%%%%%%%%%%%%%%%%%%%%%%%%%%%%%
\subsection{Applications of RMEP}\label{ssec:mSLeq}
We next apply our method, RMEPvTSVD,  to solve RMEP arising from discretizing multiparameter  ODE eigenvalue problems  by
the least-squares spectral method.
%In the literature, the algebraic MEP  is closely related to the  multiparameter version of classical Sturm-Liouville equation \cite{atmi:2010}.
Consider the following two parameter second-order ODE  eigenvalue problem:
\begin{equation}\label{eq:S-L}
	\left\{
	\begin{aligned}
		&u_1^{\prime\prime}(s)+(\lambda p_1(s)+\mu q_1(s)+f_1(s))u_1(s)=0,\\
		&u_2^{\prime\prime}(t)+(\lambda p_2(t)+\mu q_2(t)+f_2(t))u_2(t)=0,\\
		&u_1(a_1)=u_1(b_1)=0,\quad
		u_2(a_2)=u_2(b_2)=0,
	\end{aligned}
	\right.
\end{equation}
where  $(\lambda,\mu, u_1(\cdot), u_2(\cdot))$ is the eigen-tuple to be computed, and  $p_1, q_1, f_1: [a_1,b_1]\rightarrow \bbR$ and
$p_2, q_2, f_2: [a_2,b_2]\rightarrow \bbR$ are given continuous functions.

Traditionally, the finite difference method (FDM) (see, e.g., \cite{blum:1978b,doyy:2016,fohm:1972,ghhp:2012,jaho:2009}) is the standard approach for solving \eqref{eq:S-L}. Despite its advantages such as handling boundary conditions naturally and transforming the discrete system into an MEP, FDM also has notable limitations: (i)
achieving high accuracy typically requires a fine time-space discretization, leading to large-scale algebraic eigen-systems, and (ii) it only provides nodal approximations of the eigenfunctions, making it less convenient if ``continuous'' eigenfunction approximations are desired. To address these issues, the least-squares spectral method implemented via \texttt{Chebfun} \cite{drht:2014} offers a more versatile and efficient alternative.

\subsubsection{Discretization by the least-squares spectral method}\label{ssubsec:LSchebfun}
With the help of \texttt{Chebfun}, the least-squares spectral method for
the two-parameter {ODE} eigenvalue problem \eqref{eq:S-L} has been previously discussed in \cite{dris:2010,hant:2022,hana:2022}. This method   expresses the eigen-functions as linear combinations of the Chebyshev polynomial basis $\{\tau_j(t)\}$:
\begin{equation}\label{eq:approximation}
u_1(s)\approx \sum_{j=1}^{n_1} \tau_j(s)c_{1j},~~
u_2(t) \approx\sum_{j=1}^{n_2} \tau_j(t)c_{2j},
\end{equation}
where $c_{1j}$ and $c_{2j}$ are the coefficients. Plugging \eqref{eq:approximation} into  \eqref{eq:S-L} leads to \begin{equation}\label{eq:S-L-1}
	\left\{
	\begin{aligned}
		&\sum_{j=1}^{n_1}\left(\tau_j^{\prime\prime}(s)c_{1j}+f_1(s)\tau_j(s)c_{1j}\right) + (\lambda p_1(s)+\mu q_1(s))\sum_{j=1}^{n_1}\tau_j(s)c_{1j} = 0,\\
		&\sum_{j=1}^{n_2}\left(\tau_j^{\prime\prime}(t)c_{2j}+f_2(t)\tau_j(t)c_{2j}\right) + (\lambda p_2(t)+\mu q_2(t))\sum_{j=1}^{n_2}\tau_j(t)c_{2j} = 0,\\
		&\sum_{j=1}^{n_1}\tau_j(a_1)c_{1j}=\sum_{j=1}^{n_1}\tau_j(b_1)c_{1j}= 0,\quad
		\sum_{j=1}^{n_2}\tau_j(a_2)c_{2j}=\sum_{j=1}^{n_2}\tau_j(b_2)c_{2j}= 0.
	\end{aligned}
	\right.
\end{equation}
Define  {quasimatrix}-matrices $\mathcal{L}_i~(i=1,2)$ associated with the differential operators in \eqref{eq:S-L} and vectors ${J}_1(s)$ and $J_2(t)$ for any given $s$ and $t$ as
%\marginpar{\tiny $\Blue{n_i}\times n_i$ ? No, $\Blue{\infty\times n_i}$. }
 \begin{align*}
 \mathcal{L}_i &= [\tau_1^{\prime\prime}+f_i\tau_1, \dots, \tau_{n_i}^{\prime\prime}+f_i\tau_{n_i}]\in\mathbb{R}^{{\infty}\times n_i} ~(i=1,2),\\
{\cal T}_i &= [\tau_1, \dots, \tau_{n_i}]\in\mathbb{R}^{\infty\times n_i}  ~(i=1,2),\\
  {J}_1(s) &= [\tau_1(s), \dots, \tau_{n_1}(s)]\in\mathbb{R}^{1\times n_1}, \quad  {J}_2(t) = [\tau_1(t), \dots, \tau_{n_2}(t)]\in\mathbb{R}^{1\times n_2}.
 \end{align*}
Note that, in \texttt{Chebfun}, certain factorizations of the quasimatrix, for example,
the QR and SVD decompositions, can be conveniently implemented.

With these settings, we can formulate the solution for the coefficient vectors
$$
\bc_i=[c_{i1},\dots,c_{i n_i}]^{\T}\in \bbC^{n_i}~ (i=1,2)
$$
by the following nonlinear least-squares problem:
\begin{equation}\label{eq:S-L-2}
	\min_{\lambda,\mu,~\bc_i, ~i=1,2} \left\| \begin{bmatrix}
		\mathcal{L}_1 \bc_1 + (\lambda p_1 + \mu q_1) \mathcal{T}_1 \bc_1 \\
		\mathcal{L}_2 \bc_2 + (\lambda p_2 + \mu q_2) \mathcal{T}_2 \bc_2 \\
		 {J}_1(a_1) \bc_1 \\
		 {J}_1(b_1) \bc_1 \\
		 {J}_2(a_2) \bc_2 \\
		 {J}_2(b_2) \bc_2
	\end{bmatrix} \right\|_{{L}_2 \oplus \ell_2},
\end{equation}
where the norm $\|\cdot\|_{{L}_2 \oplus \ell_2}$ is composed of two components: one derived from the continuous function space (a quasimatrix norm induced by the $L_2$ integral) and the other being the standard vector 2-norm.
Next, we can use the QR factorizations
\begin{equation}
	\mathcal{L}_i = {\cal Q}_iR_i, \quad
{
\begin{bmatrix}
{J}_i(a_i)\\ {J}_i(b_i)
\end{bmatrix}
}
%[ {J}_i(a_i); {J}_i(b_i)]
	  =  \hat{Q}_i\hat{R}_i,~i=1,2,
\end{equation}
 to reformulate \eqref{eq:S-L-2} as
\begin{equation}\label{eq:S-L-3}
	\min_{\lambda,\mu,~\bc_i,~ i=1,2}~\left\|
	\begin{array}{c}
		R_1\bc_1+\lambda G_1\bc_1+\mu K_1\bc_1\\
		\hat{R}_1\bc_1\\
		R_2\bc_2+\lambda G_2\bc_2+\mu K_2\bc_2\\
		\hat{R}_2\bc_2
	\end{array}
	\right\|_2
\end{equation}
where $G_i = {\cal Q}_i^{\texttt{H}}   [p_i\tau_1, \dots, p_i\tau_{n_i}]\in \bbC^{n_i\times n_i}$ and $K_i = {\cal Q}_i^{\texttt{H}} [q_i\tau_1, \dots, q_j\tau_{n_j}]\in \bbC^{n_j\times n_j}$,
and the conjugate transpose  $\bullet^{\tt H}$   is defined through integral operations and can be done conveniently in \texttt{Chebfun} \cite{drht:2014}.
%Analogously, one can alternatively choose to factorize the quasimatrix ${\cal T}_i={\cal Q}_iR_i~(i=1,2)$ to derive a   nonlinear least squares problem similar to \eqref{eq:S-L-3}.

To solve  \eqref{eq:S-L-3}, we follow the idea in \cite{hant:2022,hana:2022} to transform it into  RMEP:
\begin{equation}\label{eq:SLRMEP}
	\left\lbrace \begin{aligned}
		A_1 \bc_1 &= \lambda B_1 \bc_1+ \mu C_1 \bc_1, \\
		A_2 \bc_2 &= \lambda B_2 \bc_2 + \mu C_2 \bc_2,
	\end{aligned}\right.
\end{equation}
where $A_i=\left[ \begin{array}{c}
	R_i\\
	\hat{R}_i
\end{array}\right]$,
$B_{i1}=\left[ \begin{array}{c}
	-G_i\\
	0
\end{array}\right]$,
and
$B_{i2}=\left[ \begin{array}{c}
	-K_i\\
	0
\end{array}\right]$, all in $\bbC^{(n_i+2)\times n_i}$ for $i=1,2$. In what follows, we will solve
two concrete \eqref{eq:SLRMEP} by our proposed method, RMEPvTSVD. In measuring approximation accuracy, besides
the normalized residuals
$\rho_i(\tilde\lambda,\tilde\mu,\tilde\bc_i)$ for $i=1,2$ as defined in \eqref{eq:NRes-ind}
that measure how accurate a computed eigen-tuple $(\tilde\lambda,\tilde\mu,\tilde\bc_1, \tilde\bc_2)$
is as an approximation to the solution of RMEP~\eqref{eq:SLRMEP}, we will also calculate another type of error defined by
\begin{equation}\label{eq:integerr}
\varsigma(\tilde\lambda,\tilde\mu, \tilde u_1,\tilde u_2)
%  &=\sum_{i=1}^2\varsigma_i(\tilde\lambda,\tilde\mu, \tilde u_1,\tilde u_2)
%              \nonumber\\
:=\sum_{i=1}^2\underbrace{
     \int_0^1\left|\tilde u_i^{\prime\prime}(t)+(\tilde\lambda p_i(t)+\tilde\mu q_i(t)+f_i(t))\tilde u_i(t)\right|\,dt
         }_{=:\varsigma_i(\tilde\lambda,\tilde\mu, \tilde u_i)},
\end{equation}
where the functions $\tilde u_i~(i=1,2)$ are given by \eqref{eq:approximation}
with the combination coefficients from the computed eigenvector tuples $(\tilde\bc_1, \tilde\bc_2)$.
It is noted that this $\varsigma(\tilde\lambda,\tilde\mu, \tilde u_1,\tilde u_2)$ tells the accuracy of each computed eigen-tuple as an approximate one to
the original continuous two parameter ODE  eigenvalue problem \eqref{eq:S-L}.

\begin{table}[t]
\centering
\caption{\small Computed ones eigenvalue tuples
           $(\tilde\lambda, \tilde\mu)$ {\em vs.}  the exact eigenvalue ones $(\lambda, \mu)$
           in \eqref{eq:SL1-eigenvalues}}
%           for Example \ref{eg:SL1}}
	\label{tab:ev_summary}
\small
	\begin{tabular}{c|c||c|c|c|c}
	\hline
	$i$ &  $j$ & $\lambda(i,j)$   &        $\mu(i,j)$
        &  $|\tilde\lambda(i,j)- \lambda(i,j)|$ & $|\tilde\mu(i,j)- \mu(i,j)|$
           \vphantom{\vrule height0.4cm width 1.2pt depth0.2cm}\\
	\hline
$1$ & $1$ & $\hm 9.8696$  & $ 0$       & $1.7906\cdot 10^{-12}$ & $6.2607\cdot 10^{-16}$  \\
$2$ & $1$ & $24.6740$ & $\hm 14.8044$  & $2.1707\cdot 10^{-12}$ & $3.7481\cdot 10^{-13}$ \\
$1$ & $2$ & $24.6740$ & $-14.8044$ & $2.1814\cdot 10^{-12}$ & $3.8014\cdot 10^{-13}$ \\
$2$ & $2$ & $39.4784$ & $0 $       & $2.5580\cdot 10^{-12}$ & $7.1504\cdot 10^{-16}$ \\
$3$ & $1$ & $49.3480$ & $\hm 39.4784$  & $7.8160\cdot 10^{-13}$ & $1.0303\cdot 10^{-12}$ \\
$1$ & $3$ & $49.3480$ & $-39.4784$ & $7.9581\cdot 10^{-13}$ & $1.0303\cdot 10^{-12}$ \\
$3$ & $2$ & $64.1524$ & $\hm 24.6740$  & $1.1653\cdot 10^{-12}$ & $1.4069\cdot 10^{-12}$ \\
$2$ & $3$ & $64.1524$ & $-24.6740$ & $1.1795\cdot 10^{-12}$ & $1.4033\cdot 10^{-12}$ \\
$1$ & $4$ & $83.8916$ & $-74.0220$ & $8.5976\cdot 10^{-12}$ & $6.7644\cdot 10^{-12}$ \\
$4$ & $1$ & $83.8916$ & $\hm 74.0220$  & $8.6118\cdot 10^{-12}$ & $6.7644\cdot 10^{-12}$ \\
\hline
	\end{tabular}\label{tab:SLeg1}
\end{table}

\subsubsection{RMEP from the multiparameter Sturm-Liouville equations}
We use the two-parameter ODE problem in \cite[Section 6.4]{atmi:2010} to demonstrate the performance of our  method,
RMEPvTSVD.

\begin{example}\label{eg:SL1}
Consider \eqref{eq:S-L} with
$$
p_1=p_2=q_2\equiv 1, q_1\equiv -1, f_1=f_2= 0, \mbox{ in } [0,1].
$$
It is known \cite[section 6.4]{atmi:2010}  that,
the continuous problem \eqref{eq:S-L} has
the eigenvalue tuples $(\lambda,\mu)$ in the closed forms:
\begin{equation}\label{eq:SL1-eigenvalues}
 \lambda=\lambda(i,j):=\frac12(i^2+j^2)\pi^2,~\mu=\mu(i,j):=\frac12(j^2-i^2)\pi^2,~i,j=1,2,\dots,
\end{equation}
and, the corresponding eigenfunctions are
 $
 u_1(s)=\sin(i\pi s)$ and $u_2(t)=\sin(j\pi t).
 $
 \end{example}

Taking $n_1=n_2=50$ in \eqref{eq:approximation}, we apply  RMEPvTSVD (\Cref{alg:RMEP-approx:ell=N}) to solve
the resulting  RMEP \eqref{eq:SLRMEP} 
and report $10$ eigen-tuples with the smallest normalized residuals.
Table~\ref{tab:SLeg1} compares the compute eigenvalues against the exact ones in \eqref{eq:SL1-eigenvalues} for the continuous
problem. It shows that the compute eigenvalues are highly accurate.
%The accuracy (absolute error) of each computed eigenvalue and  the related integers $(i,j)$ are reported in Table \ref{tab:SLeg1}.
%
To evaluate the accuracy in each computed eigen-tuple
$(\tilde\lambda,\tilde\mu,\tilde\bc_1,\tilde\bc_2)$ as a whole, we plot, in the top row of Figure~\ref{fig:eigennorm}, the normalized residuals
$\rho_i(\tilde\lambda,\tilde\mu,\tilde\bc_i)$ for $i=1,2$ and $\rho=\rho_1+\rho_2$ as defined in \eqref{eq:NRes-ind}
and the corresponding $\varsigma_i(\tilde\lambda,\tilde\mu, \tilde u_i)$   for $i=1,2$
       and $\varsigma=\varsigma_1+\varsigma_2$ as defined in \eqref{eq:integerr}.
In the bottom row of the same figure, we plot five computed
eigenfunction tuples $(\tilde u_1,\tilde u_2)$.
%in the left panel of Figure~\ref{fig:eigennorm}, while the right panel plots
%.
%Finally, Figure~\ref{fig:eigenfuns} displays the first

\begin{figure}[H]
	\centering
\begin{tabular}{cc}
\includegraphics[width= 2.8in, height=2.0in]{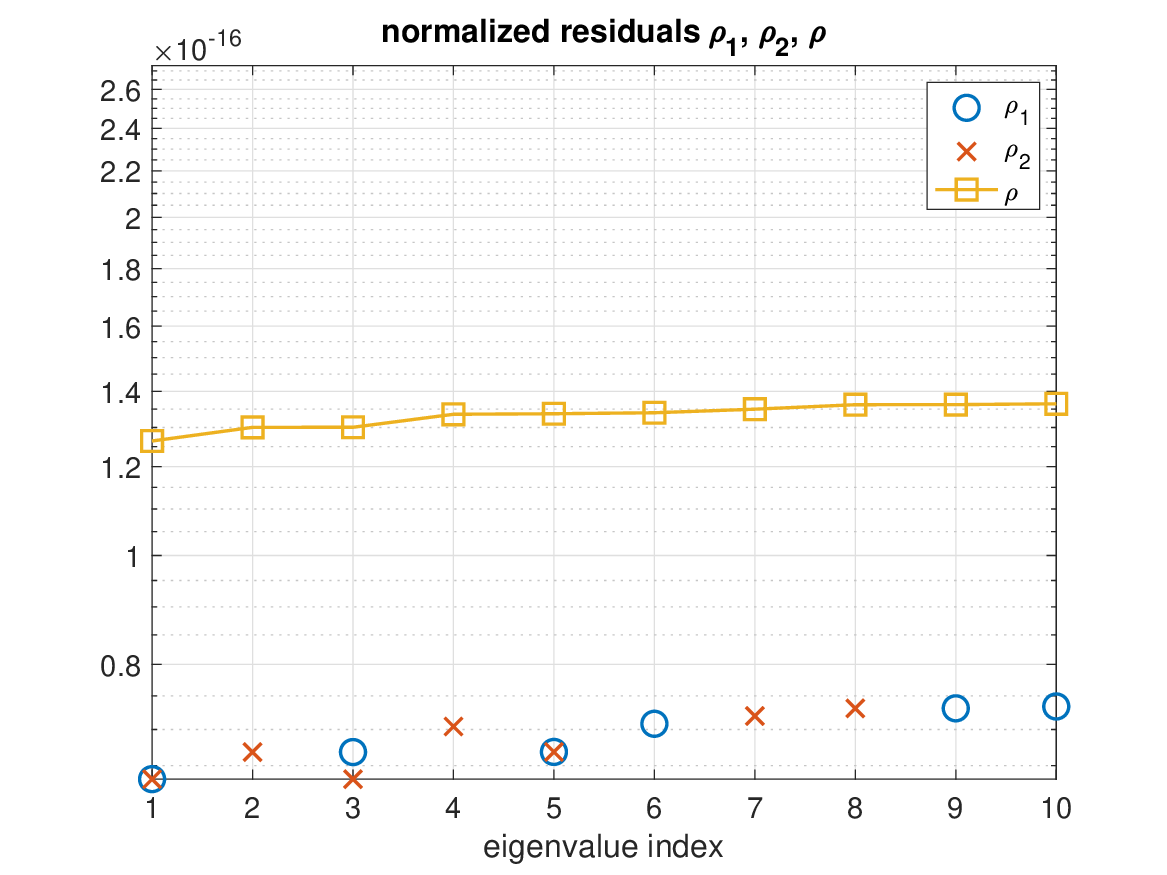} &
 \includegraphics[width= 2.8in, height=2.0in]{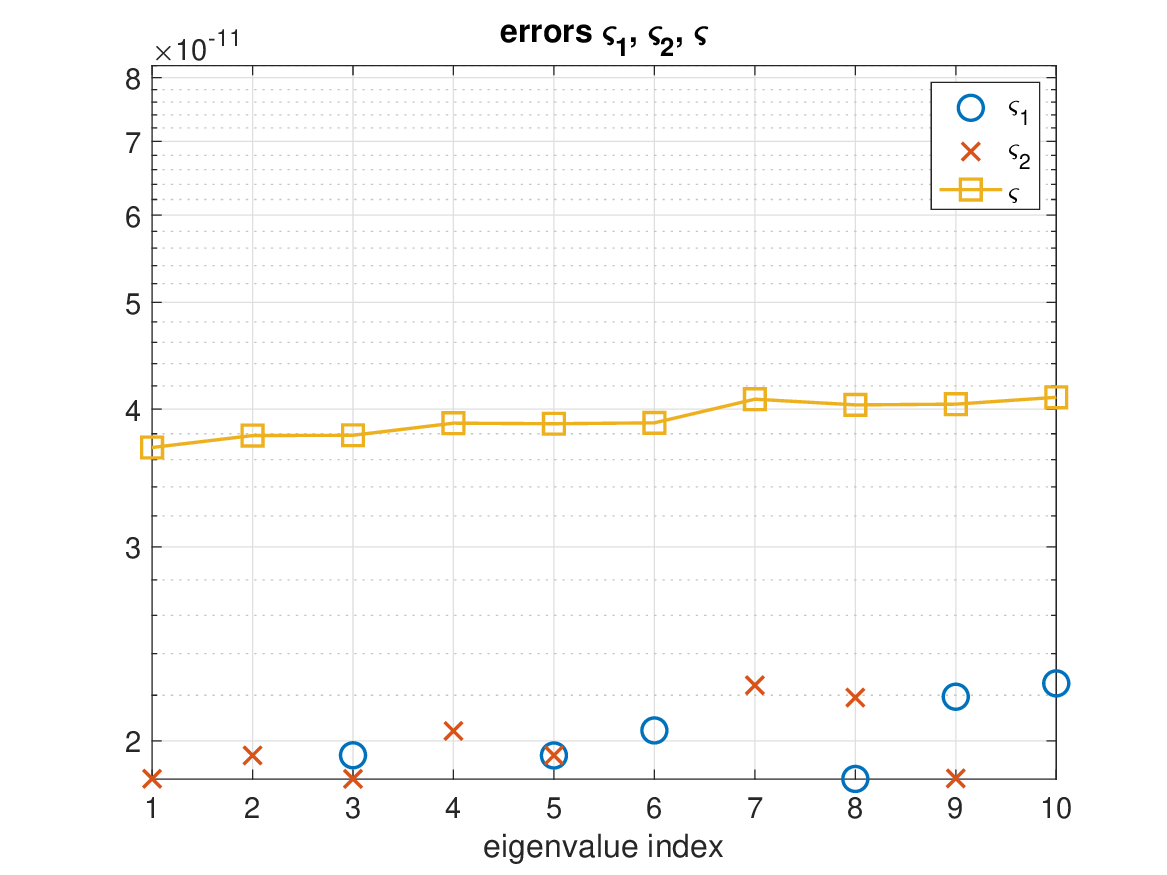} \\
\includegraphics[width= 2.8in, height=2.0in]{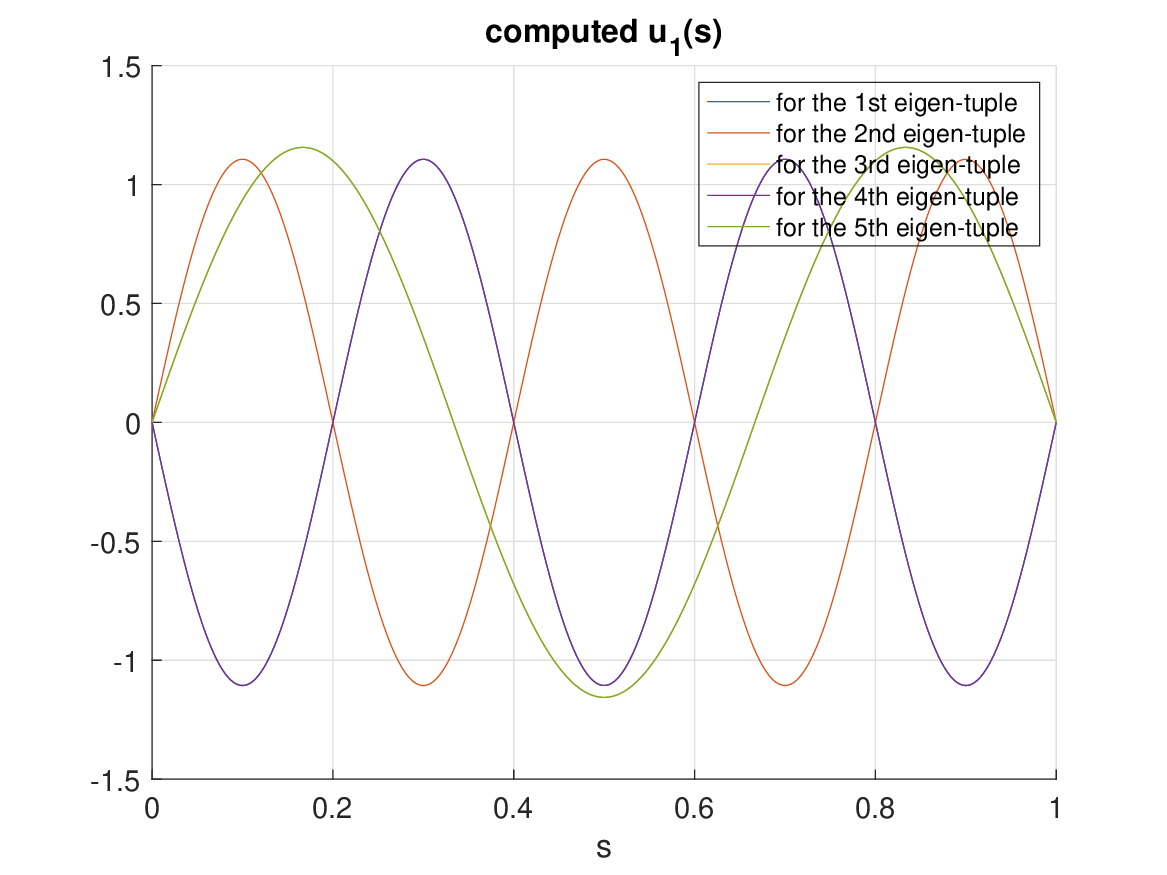} &
 \includegraphics[width= 2.8in, height=2.0in]{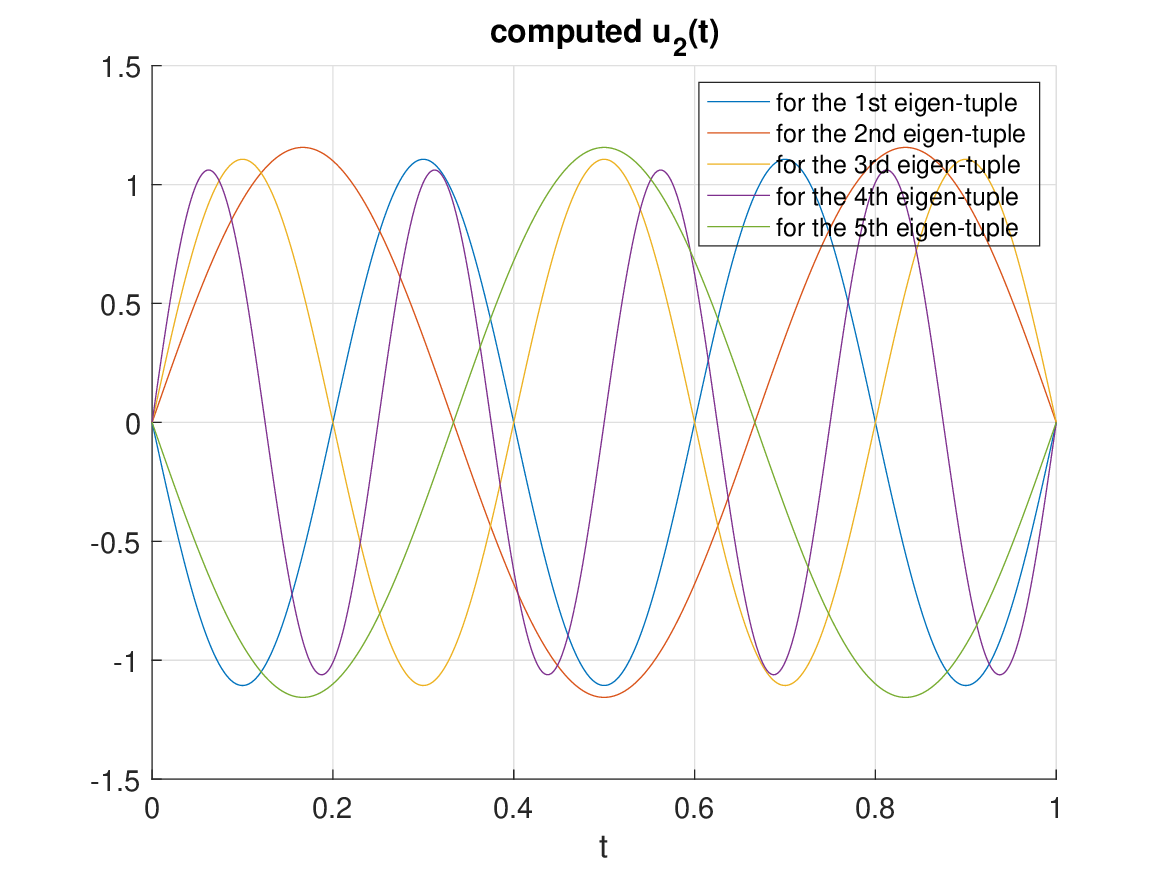}
\end{tabular}
\caption{\small Example \ref{eg:SL1}:
       {\em top left} -- normalized residuals $\rho_i=\rho_i(\tilde\lambda,\tilde\mu,\tilde\bc_i)$   for $i=1,2$
       and $\rho=\rho_1+\rho_2$
       as defined in \eqref{eq:NRes-ind}  for 10 computed eigen-tuples;
       {\em top right} -- corresponding $\varsigma_i=\varsigma_i(\tilde\lambda,\tilde\mu, \tilde u_i)$  for $i=1,2$
       and $\varsigma=\varsigma_1+\varsigma_2$ as defined in \eqref{eq:integerr}
       for each computed eigen-tuple for \eqref{eq:S-L};
       {\em bottom left} --  5 computed eigenfunctions $\tilde u_1$;
       {\em bottom right} -- corresponding 5 computed eigenfunctions $\tilde u_2$.}
	\label{fig:eigennorm}
\end{figure}

%\begin{figure}[H]
%	\centering
%{\includegraphics[width= 2.9in, height=2.3in]{SL1u1.eps}
% \includegraphics[width= 2.9in, height=2.3in]{SL1u2.eps}}
%	\caption{\small Example \ref{eg:SL1}:  5 computed eigenfunction tuples $(\tilde u_1,\tilde u_2)$.}
%	\label{fig:eigenfuns}
%\end{figure}

%\marginpar{\tiny In \Cref{fig:eigennorm}: xlabel -- eigenvalue index; legend -- $\rho_1$, $\rho_2$, $\rho$, $\varsigma_1$, $\varsigma_2$, $\varsigma$; last 2 subfigures: ``-- for 1st eigen-tuple'', ...; delete all ylabels; titles: normalized residual $\rho_1$, $\rho_2$, $\rho$;
%errors $\varsigma_1$, $\varsigma_2$, $\varsigma$; computed $u_1$; computed $u_2$}

%%%%%%%%%%%%%%%%%%%%%%%%%%%%%%%%%%%%%%%%%%%
\subsubsection{RMEP from the Helmholtz equation}\label{ssubsec:Helms}
\begin{figure}[t]
	\centering
\begin{tabular}{cc}
\includegraphics[width= 2.8in, height=2.0in]{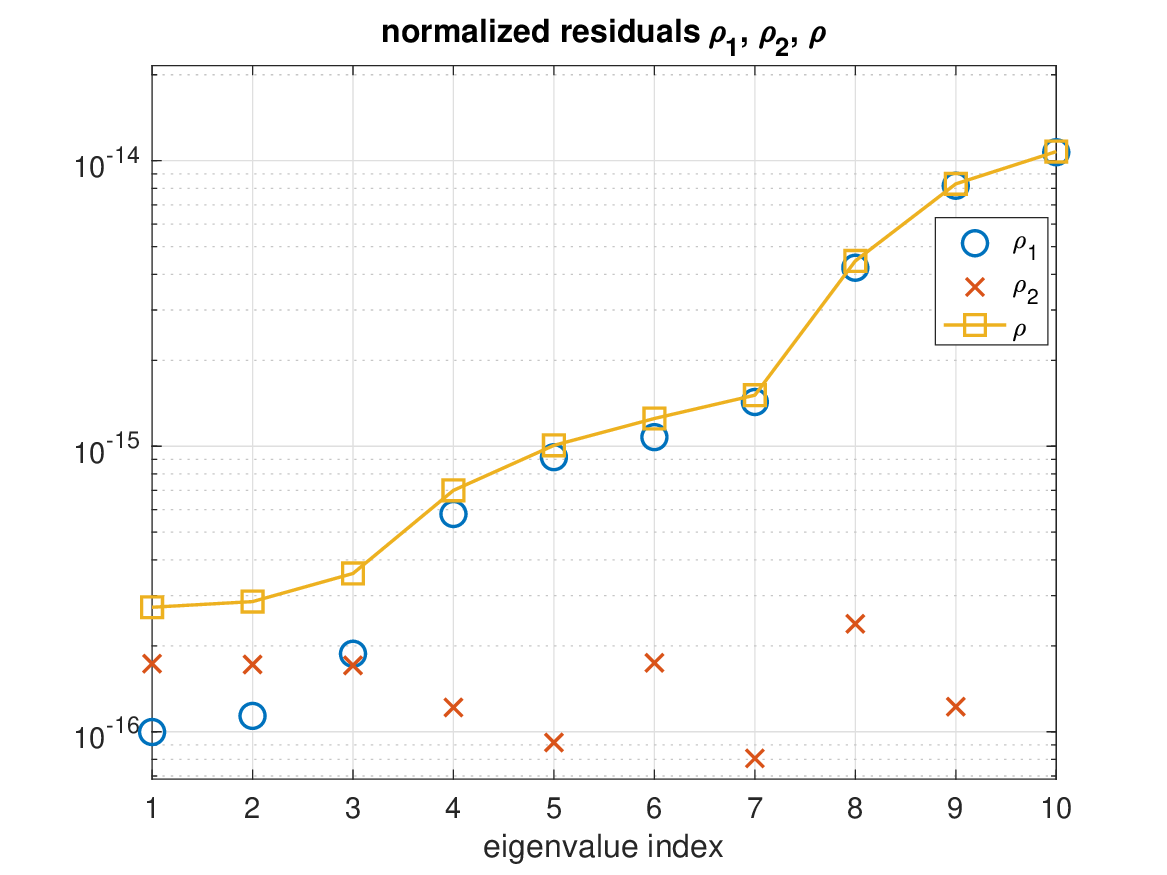} &
 \includegraphics[width= 2.8in, height=2.0in]{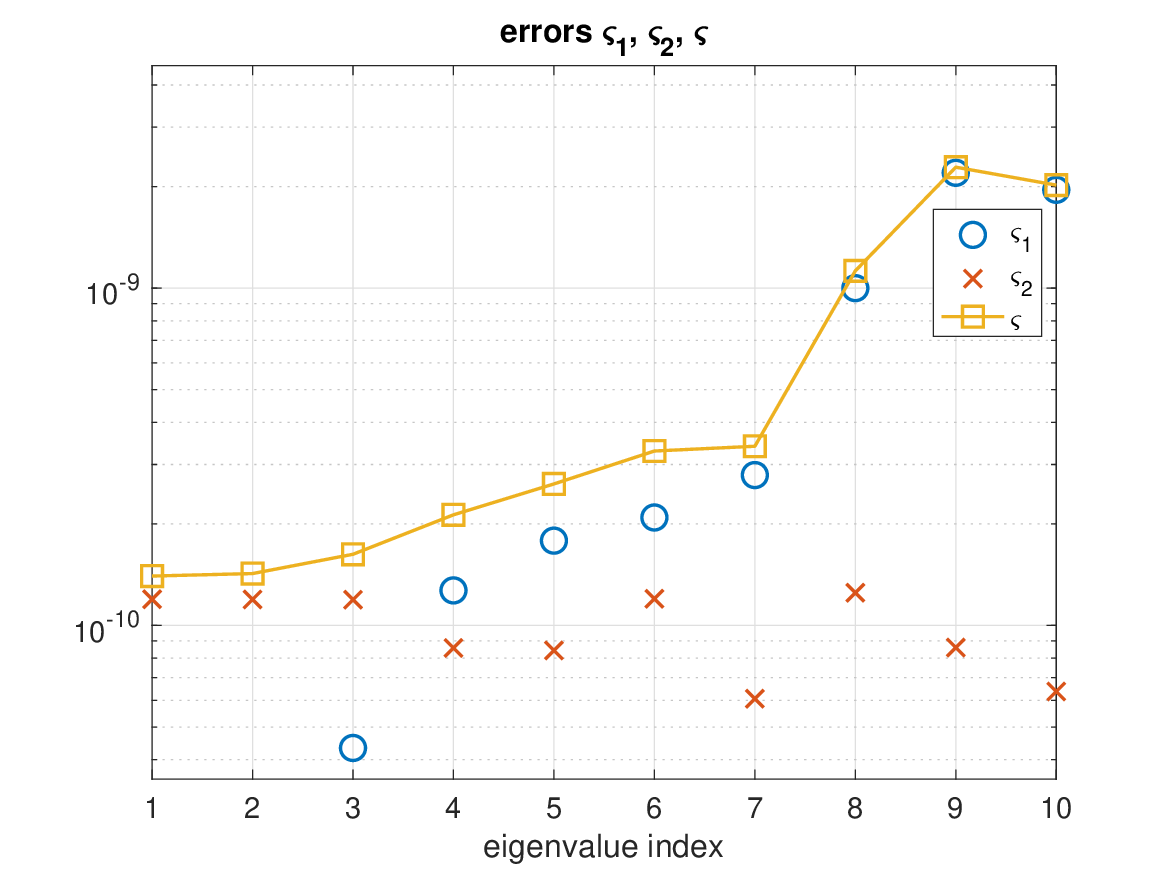} \\
\includegraphics[width= 2.8in, height=2.0in]{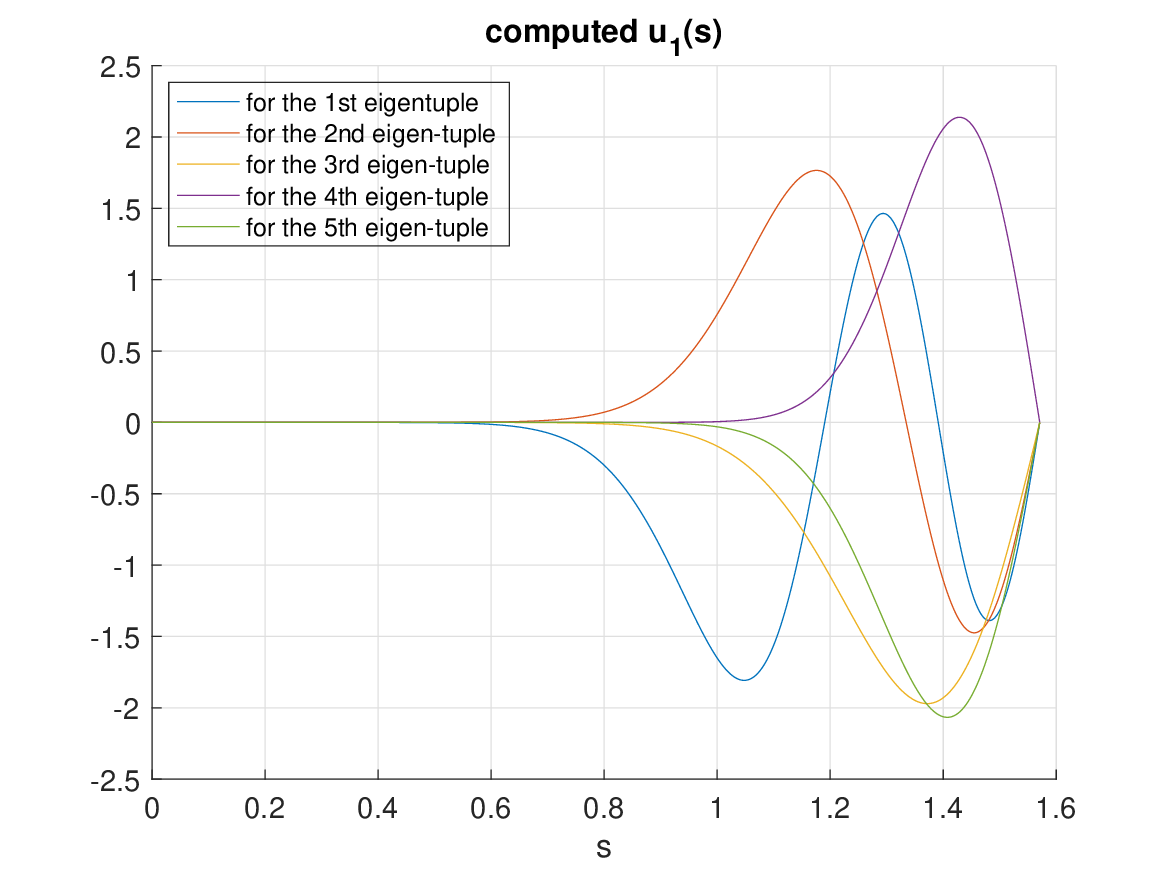} &
 \includegraphics[width= 2.8in, height=2.0in]{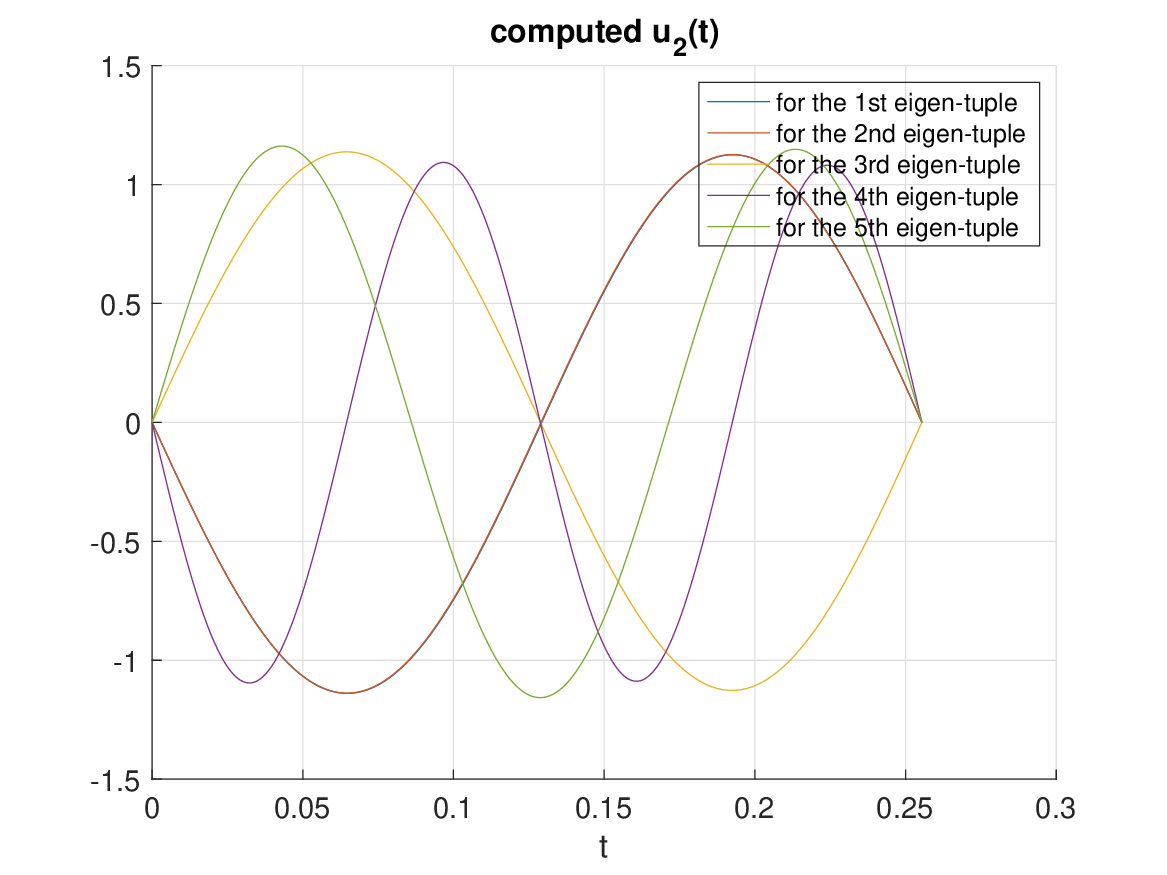}
\end{tabular}
\caption{\small Example \ref{eg:Mathieu3} with $\alpha=4$ and $\beta=1$:
       {\em top left} -- normalized residuals $\rho_i=\rho_i(\tilde\lambda,\tilde\mu,\tilde\bc_i)$ for $i=1,2$
       as defined in \eqref{eq:NRes-ind} and $\rho=\rho_1+\rho_2$;
       {\em top right} -- corresponding $\varsigma_i=\varsigma_i(\tilde\lambda,\tilde\mu, \tilde u_i)$  for $i=1,2$
       and $\varsigma=\varsigma_1+\varsigma_2$ as defined in \eqref{eq:integerr}
       for each computed eigen-tuple for \eqref{eq:S-L};
       {\em bottom left} --  5 computed eigenfunctions $\tilde u_1$;
       {\em bottom right} -- corresponding 5 computed eigenfunctions $\tilde u_2$.
       }
	\label{fig:Helmeigennorm}
\end{figure}

\begin{example}\label{eg:Mathieu3}
Consider the Helmholtz equation  of a vibrating elliptic membrance \cite{amls:2014,eina:2022,ghhp:2012,volk:1988}
\begin{equation}\label{eq:Helm1}
	\begin{aligned}
			\Delta \psi+\omega^2 \psi & = 0 \quad \text{in } \Omega=\left\{(x,y):\frac{x^2}{\alpha^2}+ \frac{y^2}{\beta^2} <1\right\}\subset \bbR^2,\\
		\psi& = 0 \quad \text{on } \partial\Omega.
	\end{aligned}
\end{equation}
A common technique to solve \eqref{eq:Helm1} is through using
the elliptical coordinates $s$ (angular) and $t$ (radial):
$$
x:=h \cosh(s) \cos(t),~y:=h \sinh(s) \sin(t),
%~0\le t\le +\infty,~0\le s\le 2\pi,
$$
 and then separation of variables: $\psi(x,y)=u_1(s)u_2(t)$,
where $h=\sqrt{\alpha^2-\beta^2}$ and $\xi_0=\arccosh(\alpha/h)$,
%and the eigenfrequency $\omega$ in \eqref{eq:Helm1} satisfies $\mu=h^2\omega^2/4$,
to yield Mathieu's system \cite{ghhp:2012,volk:1988}:
\begin{subequations}\label{eq:Mathieu3}
\begin{alignat}{3}
u_1^{\prime\prime}(s)+(\lambda-2\mu \cos(2s))u_1(s)&=0, &&~u_1(0)=u_1\left( {\pi}/{2}\right)=0,&&~s\in (0,\pi/2),\\
u_2^{\prime\prime}(t)-(\lambda-2\mu \cosh(2t))u_2(t)&=0,&&~u_2(0)=u_2(\xi_0)=0,&&~t\in (0,\xi_0),
\end{alignat}
\end{subequations}
where $\mu=h^2\omega^2/4$ and  $\lambda$ is the parameter arising from separation of variables.
This is the $\pi$-odd problem in \cite{ghhp:2012} and falls into the form \eqref{eq:S-L}.
%It is known that \eqref{eq:Mathieu3} comes from
%and assuming the eigenmode $\psi(x,y)=u_1(t)u_2(t)$,
A solution $(\lambda, \mu,u_1(t),u_2(t))$ to \eqref{eq:Mathieu3} provides an eigenmode $\psi(x,y)$ for \eqref{eq:Helm1}.
Let $\alpha=4$ and $\beta=1$ in \eqref{eq:Helm1} and take $n_1=n_2=50$ in \eqref{eq:approximation}. We then use  RMEPvTSVD
(Algorithm~\ref{alg:RMEP-approx:ell=N}) to solve  the resulting RMEP \eqref{eq:SLRMEP}.
\end{example}

In the top row of Figure~\ref{fig:Helmeigennorm}, we plot the normalized residuals
$\rho_i(\tilde\lambda,\tilde\mu,\tilde\bc_i)$ for $i=1,2$ and $\rho=\rho_1+\rho_2$ as defined in \eqref{eq:NRes-ind}
and the corresponding $\varsigma_1(\tilde\lambda,\tilde\mu, \tilde u_i)$
for $i=1,2$ and $\varsigma=\varsigma_1+\varsigma_2$ as defined in \eqref{eq:integerr}.
In the bottom row of the same figure, we plot five computed
eigenfunction tuples $(\tilde u_1,\tilde u_2)$.

Finally, eight eigenmodes $\psi(x,y)$ and their related eigenfrequencies $\omega$ are plotted in Figure~\ref{fig:HelmsolA}.
% and \ref{fig:HelmsolB}.
These results demonstrate that a combination of the least-squares spectral method, RMEPvTSVD, \texttt{Chebfun} and \texttt{MultiParEig}
constitutes a flexible, robust and efficient solver for \eqref{eq:Helm1}.

%%%%%%%%%%%%%%%%%%
\section{Conclusions}\label{sec:conclusion}
Minimal perturbation formulations are proposed to overcome the potential issue of solution nonexistence
%, whether due to inherent structure or noise,
for the general rectangular multiparameter eigenvalue problem (RMEP), whose coefficient matrices usually have more rows than columns,
so that approximate eigen-tuples can be well defined and computed.
Specifically, we investigate  efficient computation of one approximate eigen-tuple or a complete set of approximate eigen-tuples.
For one eigen-tuple, we design an alternating iterative scheme with  guaranteed convergence.
Concerning the complete set of approximate eigen-tuples, our minimal perturbation formulation
can lead to a standard multiparameter eigenvalue problem (MEP) via truncated SVDs and the MEP can then be either solved by
existing MEP solvers from \texttt{MultiParEig} or
converted to
several related generalized eigenvalue problems (GEPs) that can be solved by any existing GEP solvers.
We apply the method to solve RMEPs arising from discretizing multiparameter ODE eigenvalue problems by
the least-squares spectral method with the help of \texttt{Chebfun} \cite{dris:2010,hant:2022,hana:2022}.
Our numerical results well demonstrate the method's flexibility and robustness.

\begin{figure}[H]
	\centering
\begin{tabular}{cc}
\includegraphics[width= 2.9in, height=3.1in]{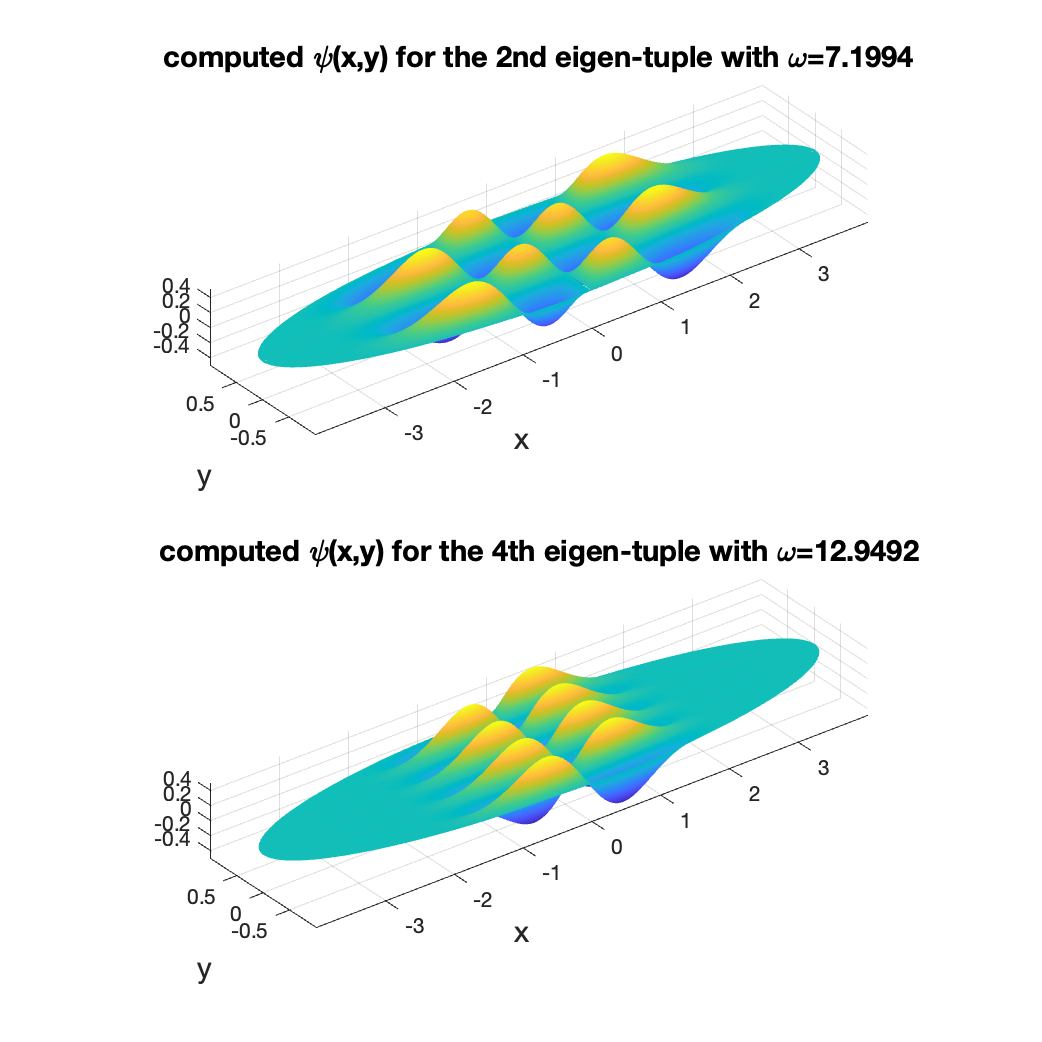}
         &  \includegraphics[width= 2.9in, height=3.1in]{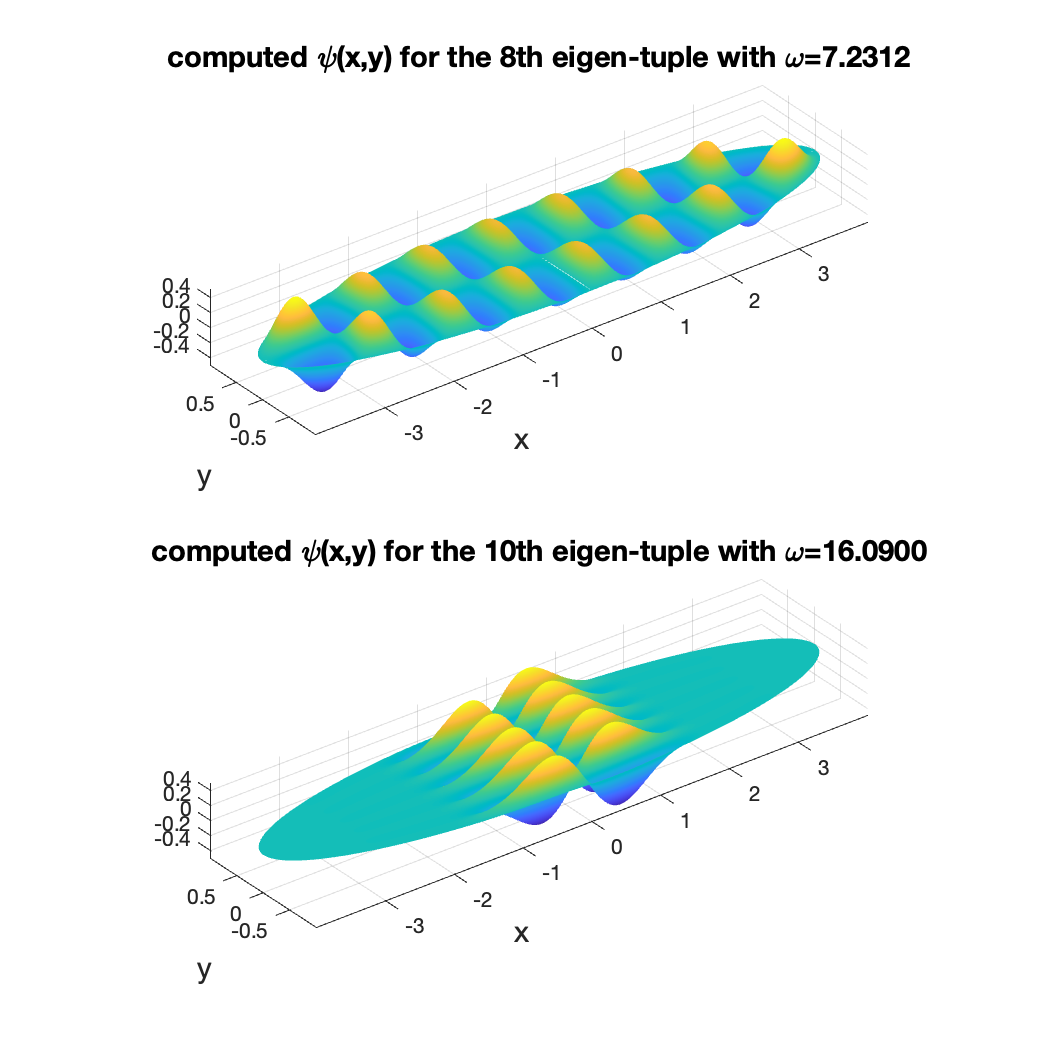} \\
\includegraphics[width= 2.9in, height=3.1in]{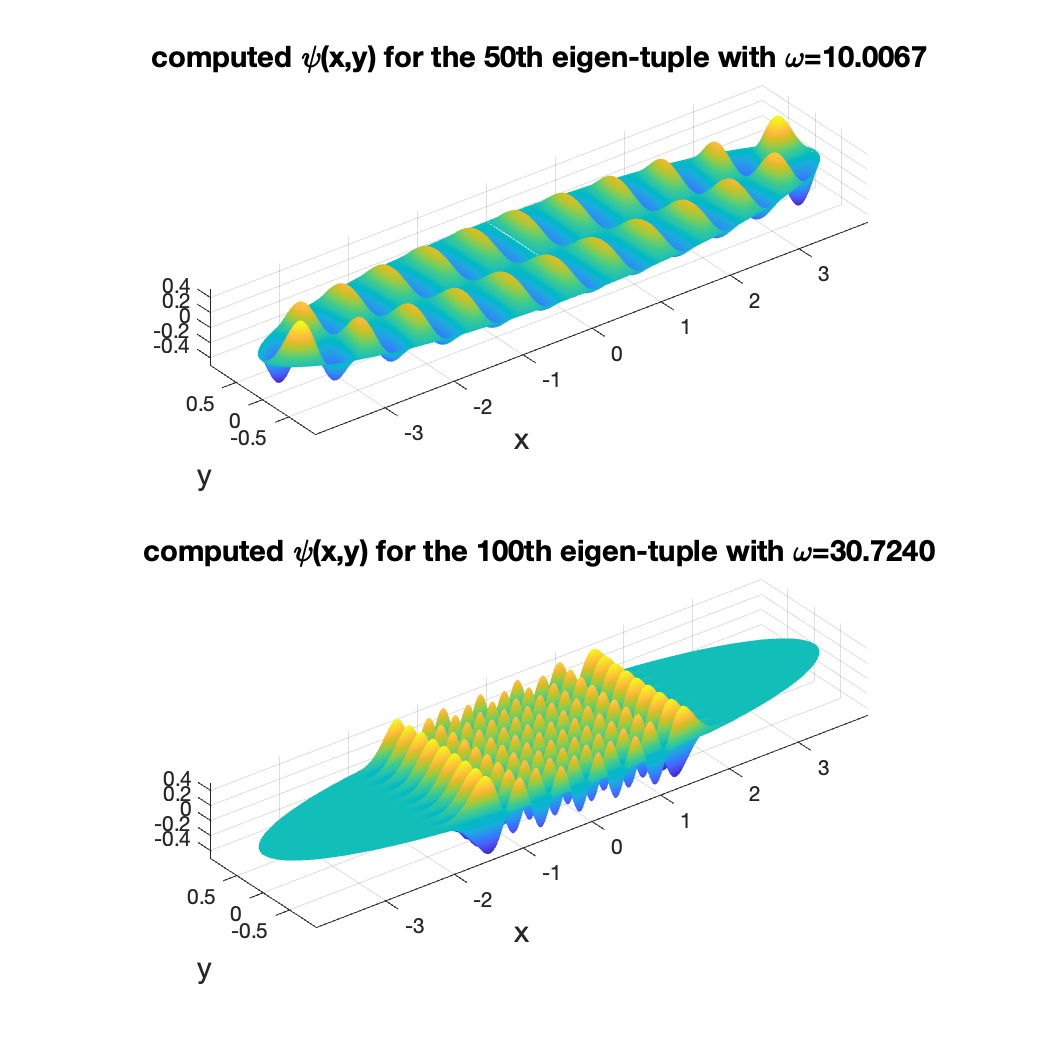}
         &  \includegraphics[width= 2.9in, height=3.1in]{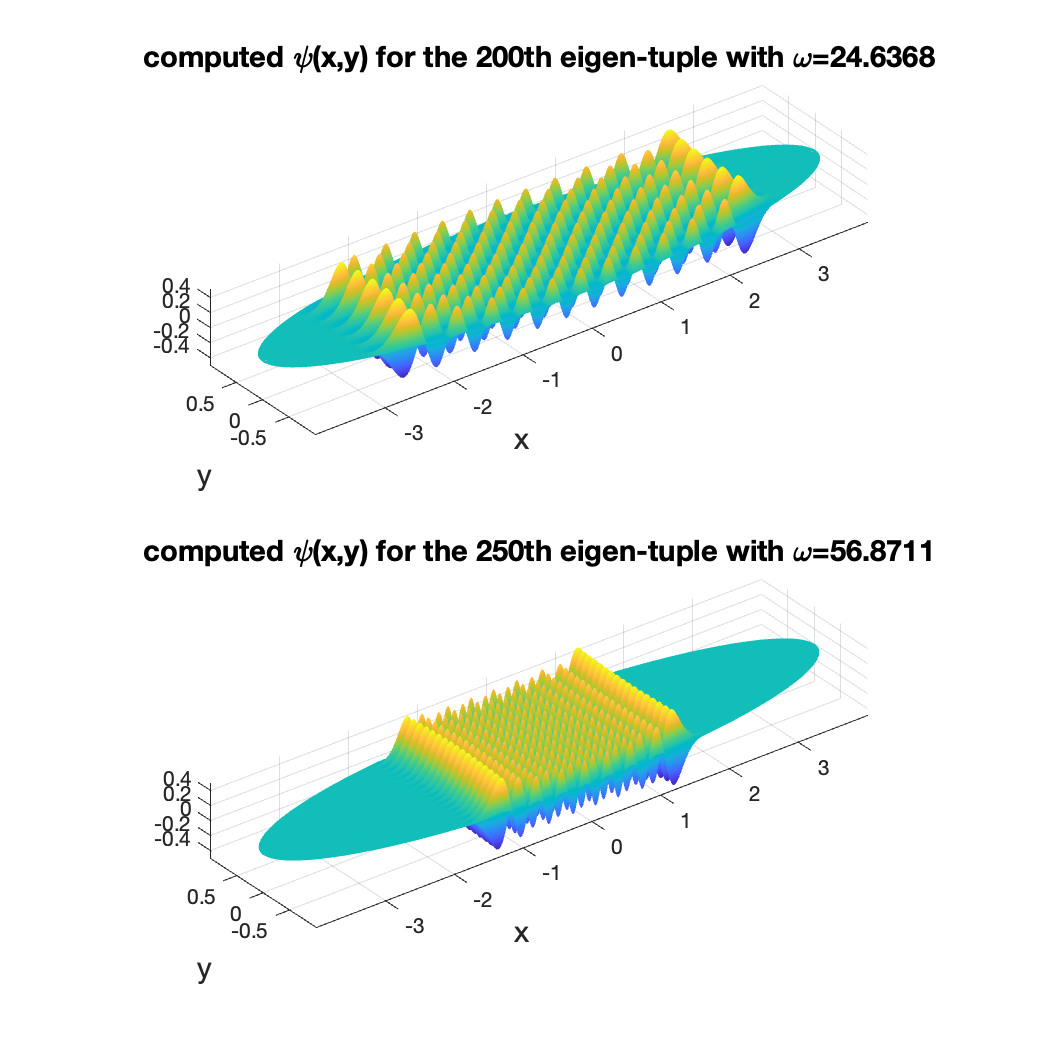}
\end{tabular}
	\caption{\small Eight computed eigenmodes $\psi(x,y)$ of \eqref{eq:Helm1} from Example \ref{eg:Mathieu3}  with $\alpha=4$ and $\beta=1$.}
	\label{fig:HelmsolA}
\end{figure}

\section*{Acknowledgments}
The authors would like to thank Michiel Hochstenbach for stimulating and useful discussions.

\appendix

\section{Solve MEP by GEPs}\label{sec:MEPvGEP}
MEP~\eqref{eq:RMEP} can be turned into GEPs \cite{atki:1968,atki:1972}.
Let
$$
\Delta_0=\left|\begin{array}{ccc}B_{11} & \dots & B_{1k} \\\vdots &   & \vdots \\B_{k1} & \dots & B_{kk}\end{array}\right|_{\otimes}, ~\Delta_i=\left|\begin{array}{ccccccc}B_{11} & \dots & B_{1~{i-1}} &A_1&B_{1~{i+1}} \dots&B_{1k} \\\vdots &   & \vdots &\vdots &   \vdots  & \vdots \\ B_{k1} & \dots & B_{k~{i-1}} &A_k&B_{k~{i+1}} \dots&B_{kk} \end{array}\right|_{\otimes}\in \bbC^{N\times N},
$$
where $|\cdot|_\otimes$ is the $k\times k$ operator determinants using Kronecker product $\otimes$ instead of multiplication (see e.g., \cite{atki:1968,atki:1972,plgh:2015}), and $N= n_1\cdots  n_k$.
It is known \cite{atki:1968,atki:1972} that if $\Delta_0$ is nonsingular, then  $\Delta_0^{-1} \Delta_i$   for $i=1,\dots,k$
commute; moreover, any eigen-tuple $((\lambda_1,\ldots,\lambda_k),(\bx_1,\ldots,\bx_k))$ of MEP~\eqref{eq:RMEP}
yields a solution to each of the following $k$ GEPs:
\begin{equation}\label{eq:MEP-GEP}
\Delta_1\bz = \lambda_1 \Delta_0 \bz,\quad
\Delta_2\bz = \lambda_1 \Delta_0 \bz,\quad
\dots,\quad
\Delta_k \bz = \lambda_k \Delta_0 \bz,
\end{equation}
with $\bz=\bx_1\otimes \dots \otimes \bx_k\ne 0$.
This provides the foundation for solving MEP~\eqref{eq:RMEP} via GEP solvers \cite{bibr:1989,doyy:2016,eina:2022,hokp:2004,hopl:2002,jixi:1991,sito:1986}.

%\Blue{
We say that MEP~\eqref{eq:RMEP} is {\em regular} (see \cite[Theorem 6.6.1]{atki:1972}), if
$$
\det(w_1 \Delta_0+w_2\Delta_1+\dots+w_{k+1}\Delta_k)\not\equiv 0
\quad\mbox{for $[w_1,w_2,\dots,w_{k+1}]^{\T}\in\bbC^{k+1}$}.
$$
It can happen that the MEP is regular, even though $\Delta_0$ is singular.
 Suppose that MEP~\eqref{eq:RMEP} is regular and let
%\marginpar{\tiny can it be done for any $k$? \textcolor{blue}{Yes, see \cite[Theorem 6.6.1]{atki:1972}.}}
$\Delta=w_1 \Delta_0+w_2\Delta_1+\dots+w_{k+1}\Delta_k$ be nonsingular for some $[w_1,\dots,w_{k+1}]^{\T}$.
Then $\Delta^{-1} \Delta_0$, $\dots$, $\Delta^{-1} \Delta_{k+1}$ commute \cite[Theorem 6.7.2]{atki:1972}, and
$\lambda_1,\dots,\lambda_k$ expressed as pairs $(\alpha_1,\gamma),\dots,(\alpha_k,\gamma)$ in the sense of
\eqref{eq:(lambda,mu):home} can be recovered from GEPs \cite[Chapter 6]{atki:1972}
\begin{equation}\label{eq:MEP-GEP'}
\Delta_0\bz = \gamma \Delta \bz,~
\Delta_1\bz = \alpha_1 \Delta \bz,~ \dots,~
\Delta_k\bz = \alpha_k \Delta \bz,~~\bz\ne 0.
\end{equation}
It is tempting to simplify these GEPs into:
$\gamma\Delta_1\bz = \alpha_1 \Delta_0\bz,\dots,\gamma\Delta_k \bz = \alpha_k \Delta_0 \bz$
in the homogenous form. This will work if  $(\Delta_i, \Delta_0)$  for $i=1,\dots,k$  are regular matrix pairs, but it is not clear if these matrix pairs are regular, given the regularity of the MEP.
Some theoretical properties and numerical methods have also been developed for singular $\Delta_0$
\cite{kopl:2022,mupl:2009} for $k=2$.
%}

While the idea of solving MEP by GEPs as explained looks alright in theory, it suffers from a major drawback in computation:
GEPs in \eqref{eq:MEP-GEP} and \eqref{eq:MEP-GEP'} are usually large scale even for a very small scale MEP, e.g.,
for $m_i=n_i=100$ and $k=2$, the involved GEPs have dimension of $10^4$. Hence often the GEPs have to be solved iteratively (unless $n_i$ in the 10s or smaller). In that regard, iterative eigensolvers for large scale GEPs \cite{bddrv:2000}
can be adapted, taking advantage of the structures in $\Delta_i$ for fast matrix-vector multiplications.
Already in \cite{hokp:2004,hopl:2002}, the Jacobi-Davidson type method has been extended for a large-scale  MEP.

\def\noopsort#1{}\def\l{\char32l}\def\v#1{{\accent20 #1}}
  \let\^^_=\v\def\hbk{hardback}\def\pbk{paperback}
\providecommand{\href}[2]{#2}
\providecommand{\arxiv}[1]{\href{http://arxiv.org/abs/#1}{arXiv:#1}}
\providecommand{\url}[1]{\texttt{#1}}
\providecommand{\urlprefix}{URL }

%
% {\small
%\bibliographystyle{aims}
%%\bibliography{/Users/longzlh/360Icloud/References/Latex/Bib/strings,/Users/longzlh/360Icloud/References/Latex/Bib/zhang-li}
%\bibliography{strings,zhang-li}

\begin{thebibliography}{10}
{\small 
\bibitem{amls:2014}
\newblock P.~Amodio, T.~Levitina, G.~Settanni and E.~Weinm{\"u}ller,
\newblock Numerical simulation of the whispering gallery modes in prolate
  spheroids,
\newblock \emph{Comput. Phys. Commun.}, \textbf{185} (2014), 1200--1206.

\bibitem{atki:1968}
\newblock F.~V. Atkinson,
\newblock Multiparameter spectral theory,
\newblock \emph{Bull. Amer. Math. Soc.}, \textbf{74} (1968), 1--27.

\bibitem{atki:1972}
\newblock F.~V. Atkinson,
\newblock \emph{Multiparameter Eigenvalue Problems, Vollum I: Matrices and
  Compact Operators},
\newblock Acad. Press, New York, 1972.

\bibitem{atmi:2010}
\newblock F.~V. Atkinson and A.~B. Mingarelli,
\newblock \emph{Multiparameter Eigenvalue Problems: {Sturm-Liouville} Theory},
\newblock CRC Press, Boca Raton, 2010.

\bibitem{demo:2020}
\newblock D.~M. B.,
\newblock Least squares optimal realisation of autonomous {LTI} systems is an
  eigenvalue problem,
\newblock \emph{Commun. Inf. Syst.}, \textbf{20} (2020), 163--207.

\bibitem{bddrv:2000}
\newblock Z.~Bai, J.~W. Demmel, J.~Dongarra, A.~Ruhe and H.~van~der
  Vorst~{(editors)},
\newblock \emph{Templates for the Solution of Algebraic Eigenvalue Problems: A
  Practical Guide},
\newblock SIAM, Philadelphia, 2000.

\bibitem{bhli:1996}
\newblock R.~Bhatia and R.-C. Li,
\newblock On perturbations of matrix pencils with real spectra. {II},
\newblock \emph{Math. Comp.}, \textbf{65} (1996), 637--645.

\bibitem{bibr:1989}
\newblock P.~Binding and P.~J. Browne,
\newblock Two parameter eigenvalue problems for matrices,
\newblock \emph{Linear Algebra Appl.}, \textbf{113} (1989), 139--157.
%\newblock
  %\urlprefix\url{https://www.sciencedirect.com/science/article/pii/0024379589902930}.

\bibitem{blum:1978}
\newblock E.~K. Blum and A.~R. Curtis,
\newblock A convergent gradient method for matrix eigenvector-eigentuple
  problems,
\newblock \emph{Numer. Math.}, \textbf{31} (1978), 247--263.
%\newblock \urlprefix\url{https://doi.org/10.1007/BF01397878}.

\bibitem{blum:1978b}
\newblock E.~K. Blum and P.~B. Geltner,
\newblock Numerical solution of eigentuple-eigenvector problems in {H}ilbert
  spaces by a gradient method,
\newblock \emph{Numer. Math.}, \textbf{31} (1978), 231--246.

\bibitem{bole:1990}
\newblock D.~Boley,
\newblock Estimating the sensitivity of the algebraic structure of pencils with
  simple eigenvalue estimates,
\newblock \emph{SIAM J. Matrix Anal. Appl.}, \textbf{11} (1990), 632--643.
%\newblock \urlprefix\url{https://doi.org/10.1137/0611046}.

\bibitem{boeg:2005}
\newblock G.~Boutry, M.~Elad, G.~H. Golub and P.~Milanfar,
\newblock The generalized eigenvalue problem for nonsquare pencils using a
  minimal perturbation approach,
\newblock \emph{SIAM J. Matrix Anal. Appl.}, \textbf{27} (2005), 582--601.
%\newblock \urlprefix\url{https://doi.org/10.1137/S0895479803428795}.

\bibitem{chgo:2006}
\newblock D.~Chu and G.~H. Golub,
\newblock On a generalized eigenvalue problem for nonsquare pencils,
\newblock \emph{SIAM J. Matrix Anal. Appl.}, \textbf{28} (2006), 770--787.
%\newblock \urlprefix\url{https://doi.org/10.1137/050628258}.

\bibitem{demm:1997}
\newblock J.~W. Demmel,
\newblock \emph{Applied Numerical Linear Algebra},
\newblock SIAM, Philadelphia, PA, 1997.

\bibitem{deka:1987}
\newblock J.~W. Demmel and B.~K{\aa}gstr\"om,
\newblock Computing stable eigendecomposltions of matrix pencils,
\newblock \emph{Linear Algebra Appl.}, \textbf{88} (1987), 139--186.

\bibitem{doyy:2016}
\newblock B.~Dong, B.~Yu and Y.~Yu,
\newblock A homotopy method for finding all solutions of a multiparameter
  eigenvalue problem,
\newblock \emph{SIAM J. Matrix Anal. Appl.}, \textbf{37} (2016), 550--571.
%\newblock \urlprefix\url{https://doi.org/10.1137/140958396}.

\bibitem{drht:2014}
\newblock T.~A. Driscoll, N.~Hale and L.~N. Trefethen,
\newblock \emph{{Chebfun User's Guide}},
\newblock Pafnuty Publications, Oxford, 2014,
\newblock See also www.chebfun.org.

\bibitem{dris:2010}
\newblock T.~A. Driscoll,
\newblock Automatic spectral collocation for integral, integro-differential,
  and integrally reformulated differential equations,
\newblock \emph{J. Comput. Phys.}, \textbf{229} (2010), 5980--5998.

\bibitem{ecka:1936}
\newblock C.~Eckart and G.~Young,
\newblock The approximation of one matrix by another of lower rank,
\newblock \emph{Psychometrika}, \textbf{1} (1936), 211--218.
%\newblock \urlprefix\url{https://doi.org/10.1007/BF02288367}.

\bibitem{edek:1997}
\newblock A.~Edelman, E.~Elmroth and B.~K\r{a}gstr\"{o}m,
\newblock A geometric approach to perturbation theory of matrices and matrix
  pencils. part {I}: {Versal} deformations,
\newblock \emph{SIAM J. Matrix Anal. Appl.}, \textbf{18} (1997), 653--692.
%\newblock \urlprefix\url{https://doi.org/10.1137/S0895479895284634}.

\bibitem{eina:2022}
\newblock H.~Eisenmann and Y.~Nakatsukasa,
\newblock Solving two-parameter eigenvalue problems using an alternating
  method,
\newblock \emph{Linear Algebra Appl.}, \textbf{643} (2022), 137--160.

\bibitem{elmg:2004}
\newblock M.~Elad, P.~Milanfar and G.~Golub,
\newblock Shape from moments - an estimation theory perspective,
\newblock \emph{IEEE Trans. Signal Processing}, \textbf{52} (2004), 1814--1829.

\bibitem{fohm:1972}
\newblock L.~Fox, L.~Hayes and D.~F. Mayers,
\newblock The double eigenvalue problems,
\newblock in \emph{in Topics in Numerical Analysis, Proceedings of the Royal
  Irish Academy Conference on Numerical Analysis} (ed. J.~J.~H. Miller),
\newblock Academic Press, New York, 1972,
\newblock 93--112.

\bibitem{gant:1959I}
\newblock F.~R. Gantmacher,
\newblock \emph{The Theory of Matrices, Vol. I},
\newblock Chelsea Publishing Company, New York, 1959.

\bibitem{gant:1959II}
\newblock F.~R. Gantmacher,
\newblock \emph{The Theory of Matrices, Vol. II},
\newblock Chelsea Publishing Company, New York, 1959.

\bibitem{ghhp:2012}
\newblock C.~Gheorghiu, M.~Hochstenbach, B.~Plestenjak and J.~Rommes,
\newblock Spectral collocation solutions to multiparameter {M}athieu's system,
\newblock \emph{Appl. \& Math. Comput.}, \textbf{218} (2012), 11990--12000.

\bibitem{govl:2013}
\newblock G.~H. Golub and C.~F. {Van Loan},
\newblock \emph{Matrix Computations},
\newblock 4th edition,
\newblock Johns Hopkins University Press, Baltimore, Maryland, 2013.

\bibitem{hant:2022}
\newblock B.~Hashemi, Y.~Nakatsukasa and L.~N. Trefethen,
\newblock Rectangular eigenvalue problems,
\newblock \emph{Adv. Comput. Math.}, \textbf{80} (2022), 1--16.

\bibitem{hana:2022}
\newblock B.~Hashemi and Y.~Nakatsukasa,
\newblock Least-squares spectral methods for {ODE} eigenvalue problems,
\newblock \emph{SIAM J. Sci. Comput.}, \textbf{44} (2022), A3244--A3264.
%\newblock \urlprefix\url{https://doi.org/10.1137/21M1445934}.

\bibitem{hokp:2004}
\newblock M.~E. Hochstenbach, T.~Ko\v{s}ir and B.~Plestenjak,
\newblock A {J}acobi-{D}avidson type method for the two-parameter eigenvalue
  problem,
\newblock \emph{SIAM J. Matrix Anal. Appl.}, \textbf{26} (2004), 477--497.

\bibitem{hokp:2024}
\newblock M.~E. Hochstenbach, T.~Ko$\check{s}$ir and B.~Plestenjak,
\newblock Numerical methods for rectangular multiparameter eigenvalue problems,
  with applications to finding optimal {ARMA} and {LTI} models,
\newblock \emph{Numer. Linear Algebra Appl.}, \textbf{31} (2023), e2540.

\bibitem{homp:2012}
\newblock M.~E. Hochstenbach, A.~Muhi{\v c} and B.~Plestenjak,
\newblock On linearizations of the quadratic two-parameter eigenvalue problem,
\newblock \emph{Linear Algebra Appl.}, \textbf{436} (2012), 2725--2743.

\bibitem{hopl:2002}
\newblock M.~E. Hochstenbach and B.~Plestenjak,
\newblock A {J}acobi--{D}avidson type method for a right definite two-parameter
  eigenvalue problem,
\newblock \emph{SIAM J. Matrix Anal. Appl.}, \textbf{24} (2002), 392--410.
%\newblock \urlprefix\url{https://doi.org/10.1137/S0895479801395264}.

\bibitem{itmu:2016}
\newblock S.~Ito and K.~Murota,
\newblock An algorithm for the generalized eigenvalue problem for nonsquare
  matrix pencils by minimal perturbation approach,
\newblock \emph{SIAM J. Matrix Anal. Appl.}, \textbf{37} (2016), 409--419.
%\newblock \urlprefix\url{https://doi.org/10.1137/14099231X}.

\bibitem{jaho:2009}
\newblock E.~Jarlebring and M.~E. Hochstenbach,
\newblock Polynomial two-parameter eigenvalue problems and matrix pencil
  methods for stability of delay-differential equations,
\newblock \emph{Linear Algebra Appl.}, \textbf{431} (2009), 369--380,
%\newblock\urlprefix\url{https://www.sciencedirect.com/science/article/pii/S0024379509000871},
\newblock Special Issue in honor of Henk van der Vorst.

\bibitem{jixi:1991}
\newblock X.~Ji,
\newblock An iterative method for the numerical solution of two-parameter
  eigenvalue problems,,
\newblock \emph{Int. J. Comput. Math.}, \textbf{41} (1991), 91--98.

\bibitem{khaz:1998}
\newblock V.~B. Khazanov,
\newblock On spectral properties of multiparameter polynomial matrices,
\newblock \emph{Journal of Mathematical Sciences}, \textbf{89} (1998),
  1775--1800.
%\newblock \urlprefix\url{https://doi.org/10.1007/BF02355378}.

\bibitem{kopl:2022}
\newblock T.~Ko{\v s}ir and B.~Plestenjak,
\newblock On the singular two-parameter eigenvalue problem {II},
\newblock \emph{Linear Algebra Appl.}, \textbf{649} (2022), 433--451.
%\newblock \urlprefix\url{https://www.sciencedirect.com/science/article/pii/S0024379522002002}.

\bibitem{kosi:1994}
\newblock T.~Ko\v{s}ir,
\newblock Finite-dimensional multiparameter spectral theory: The nonderogatory
  case,
\newblock \emph{Linear Algebra Appl.}, \textbf{212-213} (1994), 45--70.
%\newblock \urlprefix\url{https://www.sciencedirect.com/science/article/pii/0024379594903964}.

\bibitem{legc:2008}
\newblock P.~Lecumberri, M.~G\'{o}mez and A.~Carlosena,
\newblock Generalized eigenvalues of nonsquare pencils with structure,
\newblock \emph{SIAM J. Matrix Anal. Appl.}, \textbf{30} (2008), 41--55.
%\newblock \urlprefix\url{https://doi.org/10.1137/060669267}.

\bibitem{lild:2021}
\newblock J.-f. Li, W.~Li, X.-f. Duan and M.~Xiao,
\newblock Newton's method for the parameterized generalized eigenvalue problem
  with nonsquare matrix pencils,
\newblock \emph{Adv. in Comput. Math.}, \textbf{47} (2021), 29.
%\newblock \urlprefix\url{https://doi.org/10.1007/s10444-021-09855-w}.

\bibitem{lilv:2020}
\newblock J.-f. Li, W.~Li, S.-W. Vong, Q.-L. Luo and M.~Xiao,
\newblock A {R}iemannian optimization approach for solving the generalized
  eigenvalue problem for nonsquare matrix pencils,
\newblock \emph{J. Sci. Comput.}, \textbf{82} (2020), 67.
%\newblock \urlprefix\url{https://doi.org/10.1007/s10915-020-01173-5}.

\bibitem{li:1994a}
\newblock R.-C. Li,
\newblock On perturbations of matrix pencils with real spectra,
\newblock \emph{Math. Comp.}, \textbf{62} (1994), 231--265.

\bibitem{li:2003}
\newblock R.-C. Li,
\newblock On perturbations of matrix pencils with real spectra, a revisit,
\newblock \emph{Math. Comp.}, \textbf{72} (2003), 715--728.

\bibitem{li:2014HLA}
\newblock R.-C. Li,
\newblock Matrix perturbation theory,
\newblock in \emph{Handbook of Linear Algebra} (eds. L.~Hogben, R.~Brualdi and
  G.~W. Stewart), 2nd edition,
\newblock CRC Press, Boca Raton, FL, 2014,
\newblock Chapter 21.

\bibitem{maoc:1991}
\newblock E.~Marchi, J.~A. Oviedo and J.~E. Cohen,
\newblock Perturbation theory of a nonlinear game of von {Neumann},
\newblock \emph{SIAM J. Matrix Anal. Appl.}, \textbf{12} (1991), 592--596.
%\newblock \urlprefix\url{https://doi.org/10.1137/0612045}.

\bibitem{mirs:1961}
\newblock L.~Mirsky,
\newblock Symmetric gauge functions and unitarily invariant norms,
\newblock \emph{Q. J. Math.}, \textbf{11} (1961), 50--59.

\bibitem{mupl:2009}
\newblock A.~Muhi$\check{c}$ and B.~Plestenjak,
\newblock On the singular two-parameter eigenvalue problem,
\newblock \emph{Electron. J. Lin. Alg.}, \textbf{18} (2009), 420--437.

\bibitem{ples:2021}
\newblock B.~Plestenjak,
\newblock Multipareig, 2021,
\newblock
 \urlprefix\url{https://www.mathworks.com/matlabcentral/\\fileexchange/47844-multipareig}.
%\newblock
 %Https://www.mathworks.com/matlabcentral/fileexchange/\\47844-multipareig.

\bibitem{plgh:2015}
\newblock B.~Plestenjak, C.~I. Gheorghiu and M.~E. Hochstenbach,
\newblock Spectral collocation for multiparameter eigenvalue problems arising
  from separable boundary value problems,
\newblock \emph{J. Comput. Phys.}, \textbf{298} (2015), 585--601.

\bibitem{sclv:2006}
\newblock M.~Schuermans, P.~Lemmerling, L.~{De Lathauwer} and S.~{Van Huffel},
\newblock The use of total least squares data fitting in the shape-from-moments
  problem,
\newblock \emph{Signal Process.}, \textbf{86} (2006), 1109--1115.
%\newblock \urlprefix\url{https://www.sciencedirect.com/science/article/pii/S0165168405002884}.

\bibitem{shsh:2009}
\newblock B.~Shapiro and M.~Shapiro,
\newblock On eigenvalues of rectangular matrices,
\newblock \emph{Proc. Steklov Inst. Math.}, \textbf{267} (2009), 248--255.
%\newblock \urlprefix\url{https://doi.org/10.1134/S0081543809040208}.

\bibitem{slee:2008}
\newblock B.~Sleeman,
\newblock Multiparameter spectral theory and separation of variables,
\newblock \emph{J. Phys. A: Math. Theor.}, \textbf{41} (2008), 1--20.

\bibitem{sito:1986}
\newblock T.~Slivnik and G.~Tom$\check{s}$i$\check{c}$,
\newblock A numerical method for the solution of two-parameter eigenvalue
  problems,
\newblock \emph{J. Comput. Appl. Math.}, \textbf{15} (1986), 109--115.
%\newblock\urlprefix\url{https://www.sciencedirect.com/science/article/pii/0377042786902438}.

\bibitem{stsu:1990}
\newblock G.~W. Stewart and J.-G. Sun,
\newblock \emph{Matrix Perturbation Theory},
\newblock Academic Press, Boston, 1990.

\bibitem{stew:1994}
\newblock G.~Stewart,
\newblock Perturbation theory for rectangular matrix pencils,
\newblock \emph{Linear Algebra Appl.}, \textbf{208-209} (1994), 297--301.
%\newblock \urlprefix\url{https://www.sciencedirect.com/science/article/pii/0024379594904456}.

\bibitem{vemo:2022}
\newblock C.~Vermeersch and B.~{De Moor},
\newblock Two complementary block {M}acaulay matrix algorithms to solve
  multiparameter eigenvalue problems,
\newblock \emph{Linear Algebra Appl.}, \textbf{654} (2022), 177--209.

\bibitem{vemo:2023}
\newblock C.~Vermeersch and B.~De~Moor,
\newblock Recursive algorithms to update a numerical basis matrix of the null
  space of the block row, (banded) block {T}oeplitz, and block {M}acaulay
  matrix,
\newblock \emph{SIAM J. Sci. Comput.}, \textbf{45} (2023), A596--A620.

\bibitem{vemo:2019}
\newblock C.~Vermeersch and B.~D. Moor,
\newblock Globally optimal least-squares {ARMA} model identification is an
  eigenvalue problem,
\newblock \emph{IEEE Control Syst. Lett.}, \textbf{3} (2019), 1062--1067.

\bibitem{volk:1988}
\newblock H.~Volkmer,
\newblock \emph{Multiparameter Eigenvalue Problems and Expansion Theorems,
  Lecture Notes in Mathematics}, vol. 1356,
\newblock Springer-Verlag, Berlin, 1988.

\bibitem{wrtr:2002}
\newblock T.~G. Wright and L.~Trefethen,
\newblock Pseudospectra of rectangular matrices,
\newblock \emph{IMA J. Numer. Anal.}, \textbf{22} (2002), 501--519.
}
\end{thebibliography}
%}
\end{document}